\documentclass{amsart}
\usepackage[utf8]{inputenc}
\pdfoutput=1

\usepackage{graphicx, booktabs,enumerate,microtype}
\usepackage[dvipsnames]{xcolor}
\usepackage{amssymb,amsmath,amsthm, stmaryrd, mathtools} 
\usepackage{tikz, tikz-cd}
\usepackage{hyperref}
    \newcommand{\lm}[1]{\texorpdfstring{#1}{replacedLaTeXcode}} 
\usepackage{mathdots}
\usepackage[foot]{amsaddr}

\usepackage{mathrsfs}
\newcommand{\ms}[1]{\mathscr{#1}}

\hypersetup{colorlinks=true, linkcolor=RoyalBlue, citecolor=JungleGreen}

\let\oldtocsection=\tocsection

\let\oldtocsubsection=\tocsubsection

\let\oldtocsubsubsection=\tocsubsubsection

\renewcommand{\tocsection}[2]{\hspace{0em}\oldtocsection{#1}{#2}}
\renewcommand{\tocsubsection}[2]{\hspace{1em}\oldtocsubsection{#1}{#2}}
\renewcommand{\tocsubsubsection}[2]{\hspace{2em}\oldtocsubsubsection{#1}{#2}}

\newcounter{prcounter}

\newcommand{\resume}{\setcounter{enumi}{\value{prcounter}}}

\newcommand{\lf}{\left}
\newcommand{\ri}{\right}

\newcommand{\f}{\frac} 
 
\newcommand{\into}{\hookrightarrow}
\newcommand{\onto}{\twoheadrightarrow}

\newcommand{\iso}{\xrightarrow{\sim}}
\newcommand{\wh}{\widehat}
\DeclareMathOperator{\sgn}{sgn}

\DeclareMathOperator{\rank}{rank}
\newcommand{\triv}{\mathrm{triv}}

\DeclareMathOperator{\Out}{Out}

\DeclareMathOperator{\Gal}{Gal}

\DeclareMathOperator{\Hom}{Hom}

\DeclareMathOperator{\Res}{Res}
\DeclareMathOperator{\Ind}{Ind}
\DeclareMathOperator{\Sym}{Sym}
\DeclareMathOperator{\vol}{vol}

\DeclareMathOperator{\diag}{diag}
\DeclareMathOperator{\ad}{ad}
\DeclareMathOperator{\Ad}{Ad}
\DeclareMathOperator{\tr}{tr}

\newcommand{\id}{\mathrm{id}}

\DeclareMathOperator{\Lie}{\mathrm{Lie}}

\newcommand{\Ld}[1]{{}^L\!{#1}}

\newcommand{\m}[1]{\mathbf{#1}}
\newcommand{\mf}[1]{\mathfrak{#1}}
\newcommand{\mc}[1]{\mathcal{#1}}
\newcommand{\td}[1]{\tilde{#1}}
\newcommand{\wtd}[1]{\widetilde{#1}}

\newcommand{\C}{\mathbb C}
\newcommand{\R}{\mathbb R}
\newcommand{\Q}{\mathbb Q}

\newcommand{\Z}{\mathbb Z}

\newcommand{\A}{\mathbb A}

\newcommand{\SL}{\mathrm{SL}}
\newcommand{\Sp}{\mathrm{Sp}}

\newcommand{\SO}{\mathrm{SO}}

\newcommand{\GL}{\mathrm{GL}}

\newcommand{\St}{\mathrm{St}}
\newcommand{\eps}{\epsilon}
\newcommand{\om}{\omega}
\newcommand{\lb}{\lambda}
\newcommand{\Om}{\Omega}

\newcommand{\bs}{\backslash}
\newcommand{\1}{\m 1}
\newcommand{\tdj}{\td{\jmath}}

\newcommand{\adj}{\mathrm{ad}}
\newcommand{\scn}{\mathrm{sc}}

\newcommand{\disc}{\mathrm{disc}}
\newcommand{\cusp}{\mathrm{cusp}}

\newcommand{\el}{\mathrm{ell}}

\newcommand{\ur}{\mathrm{ur}}

\newcommand{\spl}{\mathrm{spl}}
\newcommand{\ns}{\mathrm{ns}}

\newcommand{\pl}{\mathrm{pl}}

\newcommand{\der}{\mathrm{der}}

\newcommand{\sm}{\mathrm{sim}}
\newcommand{\temp}{\mathrm{temp}}

\newcommand{\st}{\mathrm{st}}
\newcommand{\rss}{\mathrm{rss}}

\newcommand{\BC}{\mathbb C}

\newcommand{\BZ}{\mathbb Z}

\newcommand{\BA}{\mathbb A}

\newcommand{\fp}{\mathfrak{p}}

\newcommand{\fn}{\mathfrak{n}}

\newcommand{\EP}{\mathrm{EP}}

\newcommand{\symp}{\mathrm{symp}}
\newcommand{\orth}{\mathrm{orth}}
\usepackage{stackengine,scalerel}
\DeclareMathOperator*{\prodf}{%
  \ThisStyle{\mathop{\ensurestackMath{\stackinset{c}{0\LMpt}{c}{}{%
  \rotatebox[origin=lb]{-90}{$\SavedStyle\scalerel*{\oplus}{%
  i}$}}{\SavedStyle\prod}}}}}

\newcommand\numberthis{\addtocounter{equation}{1}\tag{\theequation}}

\newtheorem{thm}{Theorem}[subsection]
\newtheorem{prop}[thm]{Proposition}

\newtheorem{cor}[thm]{Corollary}
\newtheorem{lem}[thm]{Lemma}
\newtheorem{conj}[thm]{Conjecture}
\newtheorem{assumption}[thm]{Assumption}

\theoremstyle{remark}
\newtheorem*{note}{Note}

\theoremstyle{definition}
\newtheorem{dfn}[thm]{Definition}

\numberwithin{equation}{subsection}

\allowdisplaybreaks

\title[Root Number Equidistribution]{Root Number Equidistribution for 
Self-Dual Automorphic Representations on $\GL_N$}
\author{Rahul Dalal$^1$}
\address{$^1$Faculty of Mathematics, University of Vienna}
\email{rahul.dalal@univie.ac.at}
\author{Mathilde Gerbelli-Gauthier$^2$}
\address{$^2$Department of Mathematics, University of Toronto}
\email{m.gerbelli@utoronto.ca}
\date{\today}

\begin{document}

\begin{abstract}
Let $F$ be a totally real field. We study the root numbers $\eps(1/2, \pi)$ of self-dual cuspidal automorphic representations $\pi$ of~$\GL_{2N}/F$ with conductor~$\mf n$ and regular integral infinitesimal character~$\lb$. If $\pi$ is orthogonal, then $\eps(1/2, \pi)$ is known to be identically one. We show that for symplectic representations, the root numbers~$\eps(1/2, \pi)$ equidistribute between~$\pm 1$ as $\lb \to \infty$, provided that there exists a prime dividing~$\fn$ with power $>N$.We also study conjugate self-dual representations with respect to a CM extension $E/F$, where we obtain a similar result under the assumption that $\fn$ is divisible by a large enough power of a ramified prime and provide evidence that equidistribution does not hold otherwise. In cases where there are known to be associated Galois representations (e.g by \cite{Shin11}), we deduce root number equidistribution results for the corresponding families of $N$-dimensional Galois representations.

The proof generalizes a classical argument for the case of $\GL_2/\Q$ by using Arthur's trace formula and the endoscopic classification for quasisplit classical groups similarly to a previous work (\cite{DGG22}). The main new technical difficulty is evaluating endoscopic transfers of the required test functions at central elements.  
\end{abstract}

\maketitle

\tableofcontents

\section{Introduction}

\subsection{Context}
Recall a classical open question: consider the set of elliptic curves $E/\Q$ with conductor $\leq n$. The BSD conjecture leads us to expect that the parity of the rank of the group of $\Q$-points of $E$ depends on the root number $\eps(1/2, E) = \pm 1$ of the corresponding $L$-function. We can therefore ask whether this root number is equidistributed between $\pm 1$ as $n \to \infty$. This paper studies an automorphic, weight-aspect variant of this question for higher-rank groups. Under very mild technical assumptions, we prove that equidistribution holds in all cases where it should be expected, see Theorem \ref{mainthmintro}.

\subsubsection{Automorphic Translation}
By the modularity theorem, every elliptic curve $E/\Q$ of conductor~$n$ corresponds to a cuspidal, weight-$2$ modular (eigen)newform~$f$ of level~$n$ with matching $L$-function/root number: $\eps(1/2, E) = \eps(1/2, f)$. We can therefore ask instead if the root numbers of cuspidal weight-$2$ modular newforms of level~$n$ equidistribute as~$n \to \infty$. This does not amount to computing the same statistic for rational elliptic curves since there are many extra modular forms. However, it is more natural from some perspectives
: by the modularity theorem, cuspidal weight-$2$ newforms are in bijection with irreducible $2$-dimensional  representations of~$\Gal(\overline \Q/\Q)$ with Hodge-Tate weights $\{0,1\}$. These counts may be more relevant in certain contexts. 

Most importantly for our purposes, counting modular forms opens the way for the application of trace formula techniques. These have been used to prove equidistribution in \cite{SZ88, Martin18, Martin23} and are effective enough to compute finer statistics such as lower-order biases as in the results of Martin. Trace formulas even allow for computing subtle correlations between root numbers and Hecke eigenvalues depending on second-order estimates---for example, the ``murmurations'' phenomenon proven in \cite{Zub23}. 

\subsubsection{Generalization to Higher Dimensions} \label{ss_generalization_higher_dim}
Recall that cuspidal modular newforms correspond to certain automorphic representations $\pi$ on $\GL_2/\Q$ with regular, integral infinitesimal character at infinity. Any automorphic representation factors as a restricted direct product over places $v$ of $\Q$:
\[
\pi = {\bigotimes_v}' \pi_v,
\]
with each $\pi_v$ an irreducible representation of $\GL_2(\Q_v)$. In this context, the generalization of weight of a modular form is the so-called infinitesimal character of $\pi_\infty$.
The automorphic representations corresponding to holomorphic modular forms are those where $\pi_\infty$ has regular, integral infinitesimal character---i.e. the same as a finite-dimensional irreducible representation and therefore as ``simple'' as possible. Furthermore, as pointed out in \cite{BG14}, having integral infinitesimal character is exactly the expected condition needed to have an associated Galois representation. Thus automorphic representations of $GL_N/\Q$ with regular, integral infinitesimal character appear to be a natural generalization of holomorphic modular forms.  

Generalizations of modular forms actually satisfy one more condition. For automorphic representations $\pi$ on $\GL_2$, having regular integral infinitesimal character at infinity is equivalent to 
$\pi_\infty$ 
being either discrete series or trivial. This discrete-at-infinity requirement can be thought of as a  
``niceness'' condition; that these automorphic representations are the simplest possible in the same way that modular forms of weight~$\geq 2$ are simpler than Maass forms or weight-$1$ forms. Unfortunately, for~$N > 2$,~$\GL_N(\R)$ has no discrete series, so it at first appears that modular forms do not actually generalize to higher dimension. 

A good resolution is to  
consider automorphic representations $\pi$ that are either:
\begin{itemize}
    \item \emph{self-dual
    }: $\pi$ is an automorphic representation for $\GL_N/F$ for $F$ a totally real number field and $\pi^\vee \cong \pi$.
    \item \emph{conjugate self-dual 
    }: $\pi$ is an automorphic representation for $\GL_N/E$ for $E/F$ a CM quadratic extension of number fields and $\pi^\vee \cong \bar \pi$ (i.e. the $E/F$-conjugate).
\end{itemize}
Any (conjugate) self-dual $\pi$ has component $\pi_\infty$ that is also (conjugate) self-dual.
Moreover, such $\pi_\infty$ with regular enough and integral infinitesimal character are discrete within the (conjugate) self-dual spectrum. Lastly, these are also the automorphic representations to which, even in higher dimension, there are expected to be attached Galois representations---e.g. by work of Shin and collaborators \cite{Shin11}, \cite{KS20}, \cite{KS23} in some cases and \cite{Shi24} for a ``weak transfer'' result in all cases. 

Self-dual automorphic representations can be classified into two types: orthogonal- or symplectic-type. This should be thought of as encoding whether the
conjectural associated Galois representations preserve a symmetric or skew-symmetric bilinear form (the precise definition will be recalled in \S\ref{sec globalduality} and depends on \cite{Art13}). One of the outputs of the endoscopic classification \cite{Art13} is that orthogonal-type automorphic representations have root number determined by their central character, and as such, easily understood. In the self-dual case, the most interesting statistics are those of the symplectic-type automorphic representations. We similarly restrict our attention to the conjugate-symplectic case for conjugate self-dual representations, this time justified by results in the unitary endoscopic classification \cite{Mok15}. 

Note that in the self-dual case, restriction to symplectic-type is also natural from the Galois side. The natural generalization of elliptic curves is higher-dimensional abelian varieties. Just like elliptic curves, these are conjecturally associated to automorphic representations through their Galois representations. However, Galois representations of abelian varieties always preserve the symplectic Weil pairing, so their associated automorphic representations should  
be symplectic-type.

In summary, it is most appropriate to consider the following generalization of modular forms: (conjugate) symplectic-type automorphic representations of $\GL_N$ with regular, integral infinitesimal character at infinity.

\subsection{Results}

\subsubsection{Setup}
To make the results precise, fix a totally real number field $F$. We consider two cases:
\begin{itemize}
    \item self-dual: define $G_N = \GL_N/F$,
    \item conjugate self-dual: define $G_N = \Res^E_F \GL_N$.
\end{itemize}
Let $\mc{AR}_\cusp(G_N)$ be the set of cuspidal automorphic representations for $G_N$---i.e those on $\GL_N(F) \bs \GL_N(\A_F)$ in the self-dual case and those on $\GL_N(E) \bs \GL_N(\A_E)$ in the conjugate self-dual case. 

As discussed in \S\ref{ss_generalization_higher_dim} and made precise in \S\ref{sec globalduality}, there is a subset of $\mc{AR}_\cusp(G_N)$ of (conjugate) self-dual representations which can be partitioned in to (conjugate) symplectic and (conjugate) orthogonal pieces. Call these~$\mc{AR}^\symp_\cusp(G_N)$ and~$\mc{AR}^\orth_\cusp(G_N)$ respectively. 

Any $\pi \in \mc{AR}_\cusp(G_N)$ has two associated invariants: its infinitesimal character $\lb$ depending on $\pi_\infty$ that is a tuple of complex numbers $\lb_v$ for each $v|\infty$ (see \S\ref{sec infchars}), and its conductor $\mf n$ depending on $\pi^\infty$ that is an ideal of $\mc O_F$ (see \S\ref{sec newvectors}). For $\star = \text{orth, symp}$, let $\mc{AR}^\star_N(\lb, \mf n)$ be the subset of $\mc{AR}^\star_\cusp(G_N)$ with the corresponding invariants. Following the discussion of \S\ref{ss_generalization_higher_dim}, we only consider regular integral $\lb$ (see \S\ref{integralinfchars}).

Finally, every (conjugate) self-dual representation $\pi$ has a root number invariant $\eps(1/2,\pi) = \pm 1$ (see \S\ref{selfdualrootnumbers} for normalization details). We are interested in the average value 
\[
\lim_{\substack{\lb \to \infty \\ \lb \text{ integral}}} \f1{|\mc{AR}^\star_N(\lb, \mf n)|} \sum_{\pi  \in \mc{AR}^\star_N(\lb, \mf n)} \eps(1/2, \pi). 
\]
for $\star = \orth, \symp$ and where $\lb \to \infty$ is as a dominant weight of $G_\infty$, i.e. so that the minimum difference between elements in each tuple $\lb_v$ goes to $\infty$; see \S\ref{sec infcharnorms}.

\subsubsection{Main Theorem}
First, by the endoscopic classifications \cite{Art13} and \cite{Mok15}, the root number  $\eps(1/2, \pi)$ of $\pi \in \mc{AR}^\orth_\cusp(G_N)$ depends only on the central character of $\pi$. Therefore, we focus on the symplectic-type representations. We also assume that $N$ is even, as this restriction applies only in the conjugate self-dual case since there are no self-dual, symplectic-type representations with $N$ odd.

Our main result is that $\eps(1/2, \pi)$ equidistributes under these conditions:
\begin{thm}[\ref{mainthm}, \ref{mainthmconj}]\label{mainthmintro}
Let $N$ be even and let $\mf n = \prod_v \mf p_v^{e_v}$ be such that: 
\begin{itemize}
    \item One of the $e_v$ is either odd or greater than $N$ (self-dual case),
    \item One of the $e_v$ is large enough for one of the ramified $\mf p_v$ and $\mf n$ is coprime to $2$ (conjugate self-dual case).
\end{itemize}
Then
\[
\lim_{\substack{\lb \to \infty \\ \lb \text{ integral}}} \f1{|\mc{AR}^\symp_N(\lb, \mf n)|} \sum_{\pi \in \mc{AR}^\symp_N(\lb, \mf n)} \eps(1/2, \pi) \to 0.
\]
Furthermore, in the self-dual case, the conditions on $\mf n$ are also necessary. 
\end{thm}

Theorems \ref{mainthm} and \ref{mainthmconj} are in fact stronger: they show that averages of~$\eps(1/2, \pi)$  are $0$ even when weighted by $\tr_{\pi_S}(f_S)$ for any ``positive" unramified test function~$f_S$ at a finite set of places $S$ coprime to $\mf n$: see Definition \ref{def fullweighting}. This is analogous to weighting the average on $\GL_2$ by a classical Hecke eigenvalue. The main theorems also provide a bound on the rate of convergence in terms of this weighting. 

Finally, by Conjecture \ref{expectedEtransfernonvanishingconj}, we expect the precise conditions on $\mf n$ in Theorem~\ref{mainthmconj} to also be necessary conditions as long as there is no obstruction to the existence of $\pi \in \mc{AR}^\symp_N(\lb, \mf n)$ coming from the central character of $\pi^\infty$. 

\begin{note}
It is somewhat surprising that equidistribution in this conjugate self-dual setting seems to only holds at ramified level. In fact, checking examples on \cite{LMFDB} seems to suggest that the root number of a local conjugate self-dual representation over an unramified quadratic extension is determined by its conductor.  
\end{note}

\subsubsection{Side Effects}
Our techniques have two interesting side effects:

First, one of the key technical tools is a ``weight-aspect'' version of the bounds in~\cite{DGG22} that also applies to the self-dual setting (\S\ref{sec tfdecomp}, \ref{sec weightasymptotics}). This is of independent interest---for example, Theorem \ref{thm:shapeboundI} is a key input in proving that the gate sets for quantum circuits constructed in \cite{DEP} achieve the bounds needed to be so-called ``golden gates''.

Second, Theorem \ref{stablefourierinversion} roughly states that the Plancherel measure on a $p$-adic quasisplit classical group $G_v$ is ``constant'' on $L$-packets. In particular, this shows that all elements of a given supercuspidal $L$-packet for $G_v$ have the same formal degree (Corollary \ref{cor formaldegree}).

\subsubsection{Methods and Potential Extensions}
We rely on generalizations of the classical trace formulas that were so effective in studying root numbers of modular forms \cite{Martin23,Zub23}. Specifically, through the stabilization of the \emph{twisted} trace formula as used in the endoscopic classifications of~\cite{Art13} and~\cite{Mok15}, we produce an expression for our desired statistic made up of terms as in~\cite{Art89}, which have further been studied and bounded in \cite{ST16}. This analysis builds on the authors' previous works \cite{Dal22} and~\cite{DGG22}. It should be compared to an alternate, non-endoscopic strategy for analyzing the twisted trace formula in~\cite{TW24}. 

These more intricate methods currently come with serious costs (beyond just the length of the argument). First, the very precise bounds needed for the finer estimates of \cite{Martin23} and especially \cite{Zub23} depend on interpretations of terms on the geometric side of the trace formula in terms of $L$-values/class numbers. These are currently only available in the case of \emph{untwisted} $\GL_N$ from \cite{Yun13}.

Second, it would be desirable to have asymptotics in the level aspect $\mf n \to \infty$ in addition to $\lb \to \infty$. The complexity of the test functions $\wtd C^\infty_\mf n$ and $\wtd E^\infty_\mf n$ constructed in \S\ref{sectiontestfunctions} make this hard: we would require explicit asymptotics in $\mf n$ of their orbital integrals instead of just the very coarse information computed in \S\S\ref{sec shalika germs}-\ref{sec stable plancherel}. We are still hopeful that the $\mf n \to \infty$ limit is tractable with more work and leave it for a future paper, possibly using techniques from \cite{KY12} and~\cite{TW24}.  

Finally, our techniques fall short in the conjugate self-dual case for $N$ odd. Conceptually, we are really studying $(N-1)$st powers of root numbers by taking advantage of the particular simplicity of the $N \times (N-1)$ Rankin-Selberg integral representation: that the involution on functions that proves this functional equation is ``geometric''---i.e. comes from an involution of the group instead of something like a Fourier transform in the case of Godement-Jacquet. Our parity restriction is also related to the miraculous vanishing of signs in twisted endoscopic transfers for odd $N$ that was key to the main result of \cite{AOY23}. In our context, these signs exactly carry the information of root numbers we want to probe, so their vanishing is a tragedy instead of a miracle.

\subsection{Application to Galois Representations}
In the conjugate self-dual case when~$N$ is even and $\lb$ is integral, \cite{Shin11} attaches to each $\pi \in \mc{AR}^\symp_N(\mf n, \lb)$ a Galois representation of conductor $\mf n$, Hodge-Tate weights matching $\lb$, and the same root number.  Let $\overline \Phi^\symp_N(E/F, \mf n, \lb)$ be the resulting set of conjugate symplectic-type Galois representations $\Gamma_F \to \GL_N$. 

We deduce from our main theorems an equidistribution result for the family~$\overline \Phi^\symp_N(E/F, \mf n, \lb)$:
\begin{cor}
Let $N$ be even and assume that $\mf n$ is divisible by a sufficiently large power of a ramified prime. Then
\[
\lim_{\substack{\lb \to \infty \\ \lb \text{ integral}}} \f1{|\overline \Phi^\symp_N(E/F, \mf n, \lb)|} \sum_{\varphi \in \overline \Phi^\symp_N(E/F, \mf n, \lb)} \eps(1/2, \varphi) \to 0.
\]
\end{cor}

Assuming the Langlands correspondence, we expect $\overline \Phi^\symp_N(E/F, \mf n, \lb)$ to be exactly the set of irreducible Galois representations of specific Hodge-Tate weight and conductor. This gives:

\begin{cor}
Let $N$ be even and assume the Langlands correspondence for conjugate symplectic-type automorphic representations of $\GL_N/E$ with regular integral infinitesimal character at infinity. 

Then, if  $\Phi_N^\symp(E/F, \mf n, \lb)$ is the set of irreducible $N$-dimensional Galois representations that are 
\begin{itemize}
    \item conjugate symplectic,
    \item conductor $\mf n$ coprime to $2$ divisible by a high enough power of a ramified prime,
    \item Hodge-Tate weights matching $\lb$,
\end{itemize}
we have,
\[
\lim_{\substack{\lb \to \infty \\ \lb \text{ integral}}} \f1{|\Phi^\symp_N(E/F, \mf n, \lb)|} \sum_{\varphi \in \Phi^\symp_N(E/F, \mf n, \lb)} \eps(1/2, \varphi) \to 0
\]
\end{cor}

Known analogous self-dual results, \cite{KS20}, \cite{KS23}, do not yet cover the symplectic-type case---in particular, the weak matching of \cite{Shi24} does not yet give matching of root numbers. Therefore we only get a Galois representation corollary conditional on the Langlands correspondence:

\begin{cor}
Let $N$ be even and assume the Langlands correspondence for symplectic-type automorphic representations of $\GL_N/F$ with regular integral infinitesimal character at infinity. 

Let $\Phi_N^\symp(F, \mf n, \lb)$ be the set of irreducible symplectic Galois representations $\Gal(\bar F/F) \to \Sp_N$ that have
\begin{itemize}
    \item conductor $\mf n$ divisible by a sufficiently large power of some prime,
    \item Hodge-Tate weights matching $\lb$,
\end{itemize}
Then
\[
\lim_{\substack{\lb \to \infty \\ \lb \text{ integral}}} \f1{|\Phi^\symp_N(F, \mf n, \lb)|} \sum_{\varphi \in \Phi^\symp_N(F, \mf n, \lb)} \eps(1/2, \varphi) \to 0.
\]
if and only if there is finite place $v$ such that $v(\mf n) > N$ or $v(\mf n)$ is odd. 
\end{cor}

As in Theorem \ref{mainthmintro}, we can weight the averages by functions of Frobenius conjugacy classes at places coprime to $\mf n$, and we maintain the same control over rates of convergence---see again Theorems \ref{mainthm} and \ref{mainthmconj} for full generality.

\subsection{Summary of Argument}

Once the outline of the argument in \S\ref{sec fulloutline} is understood well, the details of the individual steps are reasonably self-contained and can more-or-less be read in any order. 

\subsubsection{$\GL_2$-model}
Our method is very closely modeled on the classical $\GL_2/\Q$-argument which we now present in modernized language to emphasize the parallelism. 

Let $\mc{AR}_\cusp(\GL_2/\Q, \lb)$ be the set of automorphic representations for $\GL_2/\Q$ with regular, integral infinitesimal character $\lb$ at infinity. Our goal is to study the asymptotic fractions of representations $\pi \in \mc{AR}_\cusp(\GL_2/\Q, \lb)$ with conductor $c(\pi) = n$ and $\eps(1/2, \pi) = \pm 1$---in other words, the ratio
\begin{equation}\label{eq gl2goal}
\lf. \sum_{\substack{\pi \in \mc{AR}_\cusp(\GL_2/\Q, \lb) \\ c(\pi) = n}} \eps(1/2, \pi) \ri/ \sum_{\substack{\pi \in \mc{AR}_\cusp(\GL_2/\Q, \lb) \\ c(\pi) = n}} 1. 
\end{equation}
as $\lb \to \infty$. 

\noindent \underline{Step 1:} 

Following the theory of newvectors, there exists a sequence of compact-open subgroups $\Gamma_0(n) \subset GL_2(\A^\infty)$ such that \[
(\pi^{\infty})^{\Gamma_0(n)} =  \begin{cases}
    1 & c(\pi) = n \\ 0 &c(\pi) > n.
\end{cases}
\]
If $c(\pi) = n$, we call an appropriately normalized element of~$(\pi^{\infty})^{\Gamma_0(n)}$ the newvector~$\varphi$ for~$\pi$. The newvector satisfies a key property of being a ``test vector'' for Rankin-Selberg integral representations. 

\noindent \underline{Step 2:}

The conductor $n$ has an associated Atkin-Lehner involution $\iota_n$. As a corollary of the proof of the functional equation of the $L$-function $L(s, \pi)$ through its $\GL_2 \times \GL_1$-Rankin-Selberg integral representation (e.g. \cite[5.10.2]{DS05}), we get
\[
\varphi^{\iota_n} = \eps^\infty(1/2, \pi) \varphi
\]
(leaving out an archimedean factor). The ``test vector'' property of the newvector $\varphi$ is a crucial input here. 

\noindent \underline{Step 3:}

We can find $g_n \in \GL_2(\A^\infty)$ such that $f^{\iota_n} = g_n f$ for all $f$ in the space of automorphic forms. By taking the trace of automorphic representations $\pi$ against the volume-normalized indicator functions of the corresponding cosets, we get:
\[
\tr_{\pi^\infty} \bar \1_{\Gamma_0(n)} = \begin{cases}
0 & c(\pi) > n \\
1 & c(\pi) = n 
\end{cases}, \qquad 
\tr_{\pi^\infty} \bar \1_{g_n \Gamma_0(n)} = \begin{cases}
0 & c(\pi) > n \\
\eps^\infty(1/2, \pi) & c(\pi) = n 
\end{cases}, 
\]
where the contributions from the archimedean place are omitted. 

\noindent \underline{Step 4:}

Using an explicit description of the space of oldforms---the $\Gamma_0(n')$-fixed points in $\pi^\infty$ when $c(\pi) = n$ and  $n|n'$---we compute
\[
\tr_{\pi^\infty} \bar \1_{\Gamma_0(n')} = \dim((\pi^\infty)^{\Gamma_0(n')}), \qquad \tr_{\pi^\infty} \bar \1_{g_n \Gamma_0(n')}.
\]
The values depend only on $n'/n$. This allows us to use linear combination of functions of the type $\bar \1_{\Gamma_0(n)}$ or $\bar \1_{g_n \Gamma_0(n)}$ to construct functions $E^\infty_n$ and $C^\infty_n$ respectively so that
\[
\tr_{\pi^\infty} C^\infty_n = \begin{cases}
1 & c(\pi) = n \\
0 & c(\pi) \neq n \\
\end{cases}, \qquad
\tr_{\pi^\infty} E^\infty_n = \begin{cases}
\eps^\infty(1/2, \pi) & c(\pi) = n \\
0 & c(\pi) \neq n \\
\end{cases}.
\]
Since $\eps_\infty(\pi_\infty, 1/2) =: \eps_\infty(1/2, \lb)$ is constant for $\pi_\infty$ with fixed infinitesimal character, we can express the quantity in \eqref{eq gl2goal}, which we aim to estimate, as the following ratio of traces:
\[
\eps_\infty(1/2, \lb) \lf. \sum_{\pi \in \mc{AR}_\cusp(\GL_2/\Q, \lb)} \tr_{\pi^\infty}(E^\infty_n) \ri/ \sum_{\pi \in \mc{AR}_\cusp(\GL_2/\Q, \lb)} \tr_{\pi^\infty}(C^\infty_n). 
\]

\noindent \underline{Step 5:}

The Eichler-Selberg trace formula gives a tractable expression for
\[
\sum_{\pi \in \mc{AR}_\cusp(\GL_2/\Q, \lb)} \tr_{\pi^\infty}(\psi)
\]
for any test function $\psi$. Using this expression, we show that if $F_\lambda$ is the finite-dimensional representation of $GL_2$ with infinitesimal character $\lb$, then
\[
(\dim F_\lb)^{-1} \sum_{\pi \in \mc{AR}_\cusp(\GL_2/\Q, \lb)} \tr_{\pi^\infty}(\psi) \to \psi(1)
\]
as $\lb \to \infty$. In particular the expression \eqref{eq gl2goal} has asymptotics:
\[
\eqref{eq gl2goal} \to \eps_\infty(1/2, \lb) \f{E^\infty_n(1)}{C^\infty_n(1)}.
\]

\noindent \underline{Step 6:}

We compute that $C_n(1) > 0$. Furthermore, $E_n(1) = 0$ if and only if $n$ is not a product of squares of distinct primes. This shows that we have asymptotic equidistribution of $\eps(1/2, \pi)$ if and only if $n$ is not product of squares of distinct primes.

\subsubsection{Full Argument}\label{sec fulloutline}
We now summarize the general $\GL_N$-argument and point out where complications occur. We are interested in the analogous
\begin{equation}\label{eq sdgoal}
\lf. \sum_{\substack{\pi \in \mc{AR}^\symp_\cusp(\GL_N/F, \lb) \\ c(\pi) = \mf n}} \eps(1/2, \pi) \ri/ \sum_{\substack{\pi \in \mc{AR}^\symp_\cusp(\GL_N/F, \lb) \\ c(\pi) = \mf n}} 1. 
\end{equation}
for totally real $F$ in the self-dual case and 
\begin{equation}\label{eq csdgoal}
\lf. \sum_{\substack{\pi \in \mc{AR}^\symp_\cusp(\Res^E_F \GL_N, \lb) \\ c(\pi) = \mf n}} \eps(1/2, \pi) \ri/ \sum_{\substack{\pi \in \mc{AR}^\symp_\cusp(\Res^E_F \GL_N, \lb) \\ c(\pi) = \mf n}} 1. 
\end{equation}
for $E/F$ CM in the conjugate self-dual case. Note here that over a general number field, the level becomes an ideal $\mf n$ of $F$ instead of an integer $n$.

\noindent \underline{Step 1:}

Newvector theory directly generalizes to $\GL_N(\A_F)$ or $\GL_N(\A_E)$ through \cite{JPSS81}. The corresponding level subgroups are the ``mirahorics'' $K_1(\mf n)$. We summarize the results and the necessary ``test vector'' properties in \S\S\ref{sec newvectors}, \ref{sec RS}.

\noindent \underline{Step 2:}

As in \cite[\S4.2]{Tom24}, we generalize the Atkin-Lehner involution to define involutions ~$\iota_\mf n, \bar \iota_\mf n$ in the self-dual or conjugate self-dual cases respectively. To connect this to root numbers, it turns out to be best to look at~$\GL_N \times \GL_{N-1}$ Rankin-Selberg $L$-functions instead of those associated to $\GL_{N} \times \GL_1$.  The proof of the functional equation (e.g. \cite[\S6]{cogdell2004long}) can be algebraically manipulated to involve either $\iota_\mf n$ or $\bar \iota_\mf n$ instead of the standard involution $g \mapsto g^{-T}$. This roughly shows that for newvectors $\varphi$ in (conjugate) self-dual representations,
 \begin{equation} \label{eq eigenvalue intro}
      \varphi^{\iota_\mf n} \approx \eps^\infty(1/2, \pi)^{N-1} \varphi.
 \end{equation}
After reviewing some background in \S\ref{sec rootnumbers}, this is work of \S\S\ref{sec involutions},\ref{sec involutionsconj} culminating in local Theorems \ref{iotafn} and \ref{iotafnconj}. 

The exponent in~\eqref{eq eigenvalue intro} is the reason for our restriction to even $N$ in the conjugate self-dual result. Luckily, there are no odd symplectic-type representations in the self-dual case.   

\noindent \underline{Step 3:}

This is where we make our main departure from the classical argument. Since~$\iota_\mf n, \bar \iota_\mf n$ are inner to~$g \mapsto g^{-T}$ and $g \mapsto \bar g^{-T}$ which is outer when $N > 2$, the element $g_\mf n$ such that 
\[
g_\mf n \varphi = \varphi^{\iota_\mf n}
\]
no longer lies $\GL_N(\A_F)$ itself, but instead in a larger twisted group. Consider involutions $\theta, \bar \theta$ that are inner to $g \mapsto g^{-T}$, $g \mapsto \bar g^{-T}$ respectively. Then $g_\mf n$ lies in one of the cosets $\wtd G_N := \GL_N(\A_F) \rtimes \theta$ or $\GL_N(\A_E) \rtimes \bar \theta$ in a corresponding semidirect product to which the action of $\GL_N$ on $\pi$ can be extended. See \S\ref{sec twistedaction} for details. 

Therefore, the generalized versions of the building blocks $\bar \1_{g_n K_1(n)}$ for $E_n$ lie on the twisted group $\wtd G_N$ instead. This is a feature and not a bug since it allows us to use the twisted trace formula and isolate the (conjugate) self-dual spectrum for~$\GL_N$ as desired. 

On the other hand, generalizing the building blocks $\bar \1_{K_1(n)}$ of $C_n$ is much more difficult since we do not have a twisted analogue. We will therefore need to construct a generalized $C_n$ in another way. 

\noindent \underline{Step 4:}

A result of Reeder \cite{Ree91} gives a description of oldform spaces $(\pi^\infty)^{K_1(\mf n')}$ for $c(\pi) | \mf n'$ in terms of Hecke operators for $\GL_{N-1}$. The action of these $\GL_{N-1}$-Hecke operators can be commuted with the $\mf \iota_\mf n$'s to give a formula for $\tr_{\pi^\infty}(\1_{g_{\mf n'} K_1(\mf n')})$ that only depends on $\mf n'/c(\pi)$ (from local Propositions \ref{twistedtraceoldvectors}, \ref{twistedtraceoldvectorsconj}). 

In particular, we can similarly solve a system of linear equations and find a linear combination $\wtd E^\infty_\mf n$ of the twisted test functions $\bar \1_{g_\mf n K_1(\mf n)}$ on $\wtd G_N$ so that roughly:
\[
\tr_{\pi^\infty}(\wtd E^\infty_\mf n) \approx \begin{cases}
\eps^\infty(1/2, \pi) & c(\pi) = \mf n \\
0 & c(\pi) \neq \mf n
\end{cases}
\]
This is the work of \S\S\ref{sec oldforms}, \ref{Edef} culminating in Theorems \ref{globalepsilontrace} and \ref{globalepsilontraceconj}. 

The twisted generalization of $C^\infty_n$ requires a completely different method in \S\ref{sec Cdef}. We first show that root numbers are constant on Bernstein components of $\GL_N$ intersect the set of (conjugate) self-dual representations of a fixed conductor (Propositions \ref{epsconstant}, \ref{epsconstantconj}). Therefore, an abstract Paley-Wiener theorem \cite{Rog81} lets us bootstrap the existence of $\wtd E^\infty_\mf n$ into the existence of a $\wtd C^\infty_\mf n$ such that
\[
\tr_{\pi^\infty}(\wtd C^\infty_\mf n) = \begin{cases}
1 & c(\pi) = \mf n \\
0 & c(\pi) \neq \mf n
\end{cases}
\]
(Corollary \ref{localC}). 

\noindent \underline{Step 5:}

The techniques of \cite{DGG22} using the inductive argument of \cite{Tai17} modify in a straightforward way to generate a tractable expression for both
\[
\sum_{\pi \in \mc{AR}_\cusp^\symp(\GL_N/F, \lb)} \tr_{\pi^\infty}(\wtd \psi)
\]
and 
\[
\sum_{\pi \in \mc{AR}_\cusp^\symp(\Res^E_F \GL_N, \lb)} \tr_{\pi^\infty}(\wtd \psi),
\]
for suitable twisted test functions $\wtd \psi$ (Proposition \ref{stablecount} and \S\ref{sec tfdecomp}).

With the bounds from \cite{ST16}, we find that, up to technical issues with centers, both expressions have scaling limits:
\[
\approx \wtd \psi^G(1),
\]
for $\wtd \psi^G$ an endoscopic transfer to the twisted endoscopic group $G=\SO_{N+1}$ of $\wtd G_N$ in the self-dual case and $G=U_N$ of $\wtd G_N$ embedded through base change in the conjugate self-dual case (\S\ref{sec weightasymptotics} culminating in Theorem \ref{weightasymptotics}). 

In total, we get that, up to center technicalities:
\[
\eqref{eq sdgoal}, \eqref{eq csdgoal} \overset{\approx}{\to} \eps_\infty(1/2, \lb) \f{(\wtd E^\infty_\mf n)^G(1)}{(\wtd C^\infty_\mf n)^G(1)}
\]
as $\lb \to \infty$. Note that \cite{ST16} also gives us a power saving on the error term.

\noindent \underline{Step 6:}

It remains to understand $(\wtd E^\infty_\mf n)^G(1)$ and $(\wtd C^\infty_\mf n)^G(1)$. Unlike the $\GL_2$ case, this is by far the hardest part of the argument. 

Our explicit description for $\wtd E^\infty_\mf n$ allows us to use the orbital integral transfer identities to compute $(\wtd E^\infty_\mf n)^G(1)$. The identity class of our relevant $G$ is a twisted analogue of the ``equisingular'' classes studied in \cite[2.4.A]{LS07}. Therefore in \S\ref{sec shalikaargument}, we recall a Shalika germ argument from M\oe{}glin and Waldspurger's stabilization of the twisted trace formula \cite{MW16} to get that $(\wtd E^\infty_\mf n)^G(1)$ equals a particular non-regular twisted $\kappa$-orbital integral on $\wtd G_N$ (Corollary \ref{equisingulartransfer}). Computing when this integral vanishes is the work of \S\ref{sec E(1)},\ref{sec E(1)conj} and is somewhat involved in the conjugate self-dual case. 

The function $\wtd C^\infty_\mf n$ is much more daunting since it only exists by an abstract Paley-Wiener theorem. Luckily, we only need to show that $(\wtd C^\infty_\mf n)^G(1) > 0$ (up to details with the center). By the endoscopic character identities and the spectral definition of $\wtd C^\infty_\mf n$, we can show that a formally defined ``stable Fourier transform'' of~$(\wtd C^\infty_\mf n)^G$ is always positive and somewhere non-vanishing. This allows us to apply a positivity criterion (Proposition~\ref{spectralpositivity}) that we prove from a ``stable Fourier inversion'' result (Theorem \ref{stablefourierinversion}). 

We obtain this stable Fourier inversion result by showing that the Plancherel measure for each $G_v$ descends to a formally-defined ``stable unitary dual'' in a key Proposition \ref{stablefourierinversioninput}. We first interpret the Plancherel measure of a local test function~$\wh \psi_v$ on the unitary dual as an asymptotic count of global automorphic representations. By ideas from \cite{DGG22}, this is then shown to match a stable trace formula asymptotic  which necessarily only depends on the stable Fourier transform of~$\psi_v$.  This is the work of \S\ref{stableinversion}. We conclude the positivity statement in \S\ref{sec positivitytest}.

That the stable Fourier transform is somewhere non-vanishing requires 
number-theoretic input: the existence of local (conjugate) self-dual representations with desired invariants---type, conductor, central character, etc. This is done in \S\ref{sec repexist} and is again quite involved in the conjugate self-dual case. Our final statements for $\wtd C^\infty_\mf n$ are in \S\ref{sec C(1)}.

\subsection{Conditionality}
Since we need to separate (conjugate) symplectic from (conjugate) orthogonal global representations with the trace formula, this result is conditional on the endoscopic classification---\cite{Art13} for the self-dual case and \cite{Mok15} for the conjugate self-dual. These classifications depend on the unpublished weighted, twisted fundamental lemma. 

We note that the dependence of \cite{Art13} on its unpublished references A25-27 and \cite{Mok15} on their unitary analogues has recently been resolved by \cite{AGIKMS24}.

\subsection{Acknowledgments}
The idea for this paper came from watching talks on the IAS YouTube channel about murmurations of families of elliptic curves and how they could be studied with the trace formula---we would like to thank 
the IAS for making them publicly available. Yiannis Sakellaridis pointed out how higher-rank Rankin-Selberg integrals would be a good way to generalize the classical $\GL_2$ argument relating root numbers to Atkin-Lehner involutions, and Radu Toma alterted us to a important sign mistake in an earlier version. We would also like to thank Alexander Bertoloni-Meli, Peter Dillery, Peter Humphries, Ashwin Iyengar, Arno Kret, Huajie Li, Kimball Martin, Masao Oi, David Schwein, Sug Woo Shin, Joel Specter, and Sandeep Varma for helpful conversations. 

The first author was supported by NSF postdoctoral grant 2103149 while working on this project.

\subsection{Notation}\label{notation}
\subsubsection{Groups Considered}\label{groups}
Throughout this paper we will be looking at two cases of involution/self-duality on $\GL_N$. Summarizing our parallel notational conventions for each case:

\noindent The self-dual or orthogonal/symplectic case where we consider:
\begin{itemize}
    \item $G_N = \GL_N$ over a totally real number field $F$ (so that for $G \in \wtd{\mc E}_\sm(N)$, the group $G_\infty$ has discrete series)
    \item Involution $\iota : g \mapsto g^{-T}$
    \item Involution $\theta = \Ad_{w_N} \circ \iota$ inner to $\iota$ fixing the standard pinning
    \item Local groups $G_{N,v} = \GL_N(F_v)$
    \item Local conductors $\mf p_v^i$ for non-negative integer $i$
    \item Global conductors $\mf n$ that are ideals of $\mc O_F$
    \item Twisted group $\wtd G_N = G_N \rtimes  \theta$
    \item Simple Endoscopic groups $G \in \mc E_\sm(R)$ that are either $\SO_{k}$ or $\Sp_{2k}$'s
\end{itemize}

\noindent The conjugate self-dual or unitary case where we consider:
\begin{itemize}
    \item $G_N = \Res^E_F G_N$ for a CM quadratic extension $E/F$ of number fields (so that for $G \in \wtd{\mc E}_\sm(N)$, $G_\infty$ has discrete series)
    \item Involution $\bar \iota : g \mapsto \bar g^{-T}$
    \item Involution $\bar \theta = \Ad_{w_N} \circ \bar \iota$ inner to $\bar \iota$ fixing the standard pinning
    \item Local groups $G_{N,v} = \GL_N(F_v \otimes_F E)$
    \item Local conductors $\mf p_v^i$ for $i$ either
    \begin{itemize}
        \item A pair of non-negative integers if $v$ is split
        \item A non-negative integer if $v$ is inert
        \item A non-negative half-integer if $v$ is ramified. When speaking locally, we will usually instead index by the exponent $2i$ of $\mf p_w^{2i} = \mf p_v^i$ for $w$ lying over $v$. 
    \end{itemize}
    \item Global conductors $\mf n$ that are finite products of local levels at different places
    \item Twisted group $\wtd G_N = G_N \rtimes \bar \theta$
    \item Simple Endoscopic groups $G \in \wtd{\mc E}_\sm(N)$ that are quasisplit $U_N$'s
\end{itemize}

\subsubsection{Other}

\noindent Basics:
\begin{itemize}
 \item $F/\Q$ a totally real number field, $E/F$ an imaginary quadratic extension
    \item places of $F$ are denoted $v$, and $q_S$ the product of residue field degrees over a finite set $S$ of finite places of $F$ 
    \item $\Gamma_F$ the absolute Galois group of $F$
  \item $\star^S$ and $\star_S$ are components at $S$ and away from $S$ of structure $\star$ for $S$ some set of places
  \item $\spl, \ns$ the sets of finite places of $F$ that are split, non-split in $E$ respectively
  \item $v_1 \boxtimes v_2$ is the corresponding representation of $H_1 \times H_2$ if $v_i$ is a representation of $H_i$ 
  \item $[d]$ is the $d$-dimesional irreducible representation of $\SL_2$
  \item $w_N$ is the antidiagonal matrix of $\GL_N$ with alternating signs so that $g \mapsto w_N g^{-T} w_N$ fixes the standard pinning.
  \item $\1_X$ is the indicator function of the set $X$ and $\bar \1_X = \vol(X)^{-1}\1_X$ 
  \item $\tau_v(A)$ is the eigenvalue of the operator $A$ on the vector $v$
  \item $\om_\pi$ is the central character of representation $\pi$
\end{itemize}
\noindent Groups 
\begin{itemize}
    \item $Z_G$ is the center of $G$
    \item $G_\gamma$ is the centralizer of $\gamma$ in $G$
    \item $G^0$ is the connected component of the identity
    \item $G^\der$ is the derived subgroup
    \item $G_\adj$ is the adjoint quotient
    \item $G^\scn$ is the simply-connected cover of $G^\der$
    \item $P_G$ is the number of positive roots of $G$
    \item $\wh G$ is the complex dual group
    \item $\Ld G = \wh G \rtimes \Gamma_F$ is the Langlands dual
    \item $\check G_v, \check G_S$ are the unitary duals (with Fell topology and Plancherel measure)
    \item $\check G_v^\temp, \check G_S^\temp$ are the tempered unitary duals
    \item $\wh f$ is the Fourier transform of function $f$ on $G_v$ or $G_S$. Beware that this is modified from the usual definition at non-tempered $\pi$; see \S\ref{sec:inversionsetup}. 
\end{itemize}
\noindent With respect to either the self-dual or conjugate self-dual endoscopic classification:
\begin{itemize}
    \item $\wtd{\mc E}_\el(N), \mc E_\el(G)$ are the elliptic endoscopic groups of $\wtd G_N, G$ respectively as in \S\ref{sec twistedendoscopicgroups}
    \item $\wtd{\mc E}_\sm(N)$ are the simple endoscopic groups of $\wtd G_N$ as in \S\ref{groups}
    \item $f^H$ for $f$ a test function on reductive $G/F$ is a choice of endoscopic transfer to some $H \in \mc E_\el(G)$
    \item $f^N$ for $f$ a test function on $G \in \wtd{\mc E}_\el(N)$ is a test function on $\wtd G_N$ that transfers to $f$
    \item $\psi = \oplus \tau_i[d_i]$ is an Arthur parameter with cuspidal components $\tau_i$
    \item $\Psi(N), \Psi(G)$ are the sets of parameters associated to $\wtd G_N$ and $G$ respectively (or for $G = \SO_{2N}^\eta$, outer automorphism orbits of such)
    \item The ``Arthur-$\SL_2$'' of a parameter is an unordered partition $Q$ representing its restriction to the Arthur-$\SL_2$
    \item $\pi_\psi$ is the automorphic representation of $G_N$ corresponding to $\psi$
    \item $\td \pi_\psi$ is the extension of $\pi_\psi$ to $\wtd G_N$
    \item $\varphi_\psi$ is the $L$-parameter associated to the $A$-parameter $\psi$
    \item $\mc S_\psi, \mc S_{\psi_v}$ are Arthur's component groups associated to global or local parameters
    \item $s_\psi, s_{\psi_v}$ are the special elements identified in these component groups
    \item $\eps_\psi$ is the identified character on global $\mc S_\psi$
    \item $\Pi_\psi(G), \Pi_{\psi_v}(G_v)$ are the $A$-packets associated to global or local $A$-parameters on group $G,G_v$ (or for $G = \SO_{2N}^\eta$, outer automorphism orbits of such). 
    \item $\eta^\psi_\pi$ is the character of $\mc S_\psi$ associated to $\pi \in \Pi_\psi$
    \item $\tr_{\psi_v} := \tr^{G_v}_{\psi_v}$ is the stable packet trace for parameter $\psi_v$ on group $G_v$
\end{itemize}
\noindent Shapes
\begin{itemize}
   \item $\Delta = ((T_i, d_i, \eta_i, \lb_i))_i$ is a refined shape as in Definition \ref{def refinedshape}.
    \item $\Box = ((T_i, d_i, \eta_i))_i$ is an unrefined shape as in \S\ref{sec unrefinedshapes}
    \item $\Box(\Delta)$ is the unrefined shape from forgetting information in $\Delta$
    \item $\Sigma_{\eta,\lb}, \Sigma_\eta = \Box(\Sigma_{\eta,\lb})$ are the simple (refined) shapes as in Definition \ref{def simpleshape}
    \item $G(\Delta), G(\Box)$ are the groups associated to (refined) shapes $\Delta$ or $\Box$ as in Proposition \ref{shapetogroup}
    \item $G_F(\Box)$ is the reduced group associated to shape $\Box$ as in \S\ref{sec unrefinedshapes}
    \item $\Box_\lb$ is a set of refined shapes $\Delta$ with total infinitesimal character $\lb$ associated to unrefined shape $\Box$ as in \S\ref{sec unrefinedshapes}
    \item $\dim \lb$ is the dimension of the finite-dimensional representation with regular, integral infinitesimal character $\lb$
    \item $m(\lb)$ is a norm of the infinitesimal character $\lb$ from \S\ref{sec infcharnorms}
    \item $\dim_\Box \lb$ is a modified dimension based on the unrefined shape $\Box$ from \S\ref{sec infcharnorms} 
\end{itemize}
\noindent Trace Formula Decompositions:
\begin{itemize}
    \item $\mc{AR}_\disc(G)$, $\mc{AR}_\cusp(G)$ is the set of discrete (resp. cuspidal) automorphic representations of $G/F$
    \item
    $I^G, S^G$ are Arthur's invariant and stable trace formulas for $G$
    \item $I_\disc^G, S_\disc^G$ are their discrete parts
    \item $R^G_\disc$ is the trace against $\mc{AR}_\disc(G)$
    \item $I^G_\psi, S^G_\psi$ are the summands of $I^G_\disc, S^G_\disc$ associated to parameter $\psi$
    \item $I^G_\Delta, S^G_\Delta, I^G_\Box, S^G_\Box$ are the summands of $I^G_\disc, S^G_\disc$ associated to the refined shape $\Delta$ or shape $\Box$
    \item $\star^N$ is the version of any of the variants above associated to $\wtd G_N$ 
    \item $\EP_\lb$ is the (endoscopically normalized) Euler-Poincar\'e function for regular integral infinitesimal character $\lb$ as in \S\ref{sec weightsetup}
\end{itemize}
\noindent Conductors:
\begin{itemize}
    \item $W(\psi)$ is the space of Whittaker functions on $\GL_N(F_v)$ for character $\psi$
    \item $W_\varphi := W^\psi_\varphi$ is the Whittaker function for vector $\varphi$ in representation $\pi$.
    \item $K_v$ is the hyperspecial subgroup of the $p$-adic group $G_v$ 
    \item $K_1(\mf n)$ is the ``mirahoric'' level-subgroup for conductor $\mf n$ from \ref{def mirahoric}
    \item $K_{1,v}(k) := K_1(\mf p_v^k)$ is the local factor at $v$ of $K_1(\mf n)$
    \item $c(\pi)$ is the conductor of global representation $\pi$.
    \item $w$ is place of $E$ lying over a non-split place $v$ of $F$ 
    \item $c(\pi_v)$ is the $k$ such that local $\pi_v$ has conductor $\mf p_v^k$ in the self-dual case and $\mf p_w^k$ in the conjugate self-dual case (beware that this local indexing is twice the global indexing for ramified $v$ in the conjugate self-dual case)
    \item $e_v$ is the ramification index of $v$ in $E$ 
    \item $\mf D_{E/F}$ is the different ideal of $E/F$ 
    \item $j_v$ is a ``depth-of-ramification'' invariant of $v$ from \S\ref{sec repexist} in the conjugate self-dual case
    \item $b_v$ is an invariant of $v$ depending on roots of unity defined in Lemma \ref{lem ccharglobconj} in the conjugate self-dual case. 
    \item $\wtd E^\infty_{v,k}, \wtd E^\infty_\mf n$ are twisted test functions counting representations weighted by root number defined in \S\ref{Edef} 
    \item $\wtd C^\infty_{v,k}, \wtd C^\infty_\mf n$ are twisted test functions counting representations without weighting defined in \S\ref{sec Cdef} 
    \item $\td f_{v,k}$, $\td f_{\mf n}$ are indicators of cosets defined in \S\ref{sec E(1)} and \S\ref{sec fvkconj}
    \item $\mc{AR}_N^\star(\lb, \mf n)$ is the set of (conjugate) $\star$-type cuspidal automorphic representations on $\GL_N$ with infinitesimal character $\lb$ at infinity and conductor $\mf n$ for $\star = \symp$ or $\orth$. 
\end{itemize}

\section{Self-Dual and Conjugate Self-Dual Representations}\label{sec duality}

We will consider, in parallel, self-dual and conjugate self-dual automorphic representations of $GL_N$, and their local constituents.

\subsection{Local Notions}
Let $F$ be a local field and $\pi$ be a unitary irreducible representation of $GL_N(F)$. The dual $\pi^\vee$ of $\pi$ is the representation acting on the same space, but such that \[ \pi^\vee(g) = \pi(g^{-T}). \] 

\subsubsection{Self-dual representations} \label{ss def conj}

We will say that $\pi$ is self-dual if $\pi \simeq \pi^\vee$. In this case, $\pi$ corresponds under the local Langlands correspondence \cite{HT99, Hen00, Sch13} to an $N$-dimensional representation $\rho$ of the Local Langlands group $L_F$, i.e. a representation such that 
\[
\rho(\sigma) = \rho(\sigma)^{-T} \quad \sigma \in L_F.
\] 

If $\rho$ is additionally irreducible, then it preserves a nondegenerate bilinear form~$B$. We say that $\rho$ is \emph{orthogonal} or \emph{symplectic} depending on whether $B$ can be chosen to be. We correspondingly say that $\pi$ is orthogonal-type or symplectic-type. If $\pi$ is symplectic-type, then $N$ is necessarily even, and the central character of $\pi$ is trivial. When $\pi$ is orthogonal-type, its central character is a quadratic character $\omega_\pi: F^\times \to \pm 1$. Note that being orthogonal and symplectic aren't mutually exclusive. 

If $\rho$ is reducible but semisimple, we say it is orthogonal or symplectic if all its irreducible constituents are. 

\subsubsection{Conjugate self-dual representations}
Let $E/F$ be a (possibly split) extension of local fields, with involution $x \mapsto \bar{x}$. We say that a representation $\pi$ of $GL_N(E)$ is \emph{conjugate self-dual} if $\pi \simeq \bar{\pi}^\vee$ where 
\[ 
\bar{\pi}^{\vee}(g) = \pi(\bar{g}^{-T}). 
\]
Such a representation $\pi$ corresponds under the local Langlands correspondence  to an $N$-dimensional conjugate self-dual representation $\rho$ of $\Gamma_E$. That is, we have $\rho^{-T} \simeq \rho^c$ where $\rho^c(\sigma) := \rho(w_c^{-1}\sigma w_c)$ for $w_c \in L_E \setminus L_F$. 

If $\rho$ is further irreducible, there is a bilinear form $B$ on $\rho$ such that for all $\sigma \in \Gamma_E$
\[
B(\rho^c(\sigma)x, \rho(\sigma) x) = B(x,y), \qquad B(x,y) = \eta B(y, \rho(w_c^2)x)
\]
for $\eta \pm 1$. If we can choose $\eta = 1$, then $\rho$ is \emph{conjugate-orthogonal} and if we can choose $\eta = -1$, $\rho$ is \emph{conjugate-symplectic}. 
Note that these conditions are not mutually exclusive. We accordingly say that $\pi$ is orthogonal-type or symplectic-type.

If $\rho$ is reducible but semisimple, we say it is conjugate orthogonal or conjugate symplectic if all its irreducible constituents are.

\subsection{Twisted Groups and Endoscopy}\label{sec twistedintro}

\subsubsection{Twisted Groups}
Self-dual and conjugate self-dual representations of $\GL_N$ are best studied using twisted endoscopy. We first define the key twisted groups. 

Over field $F$ or quadratic extension $E/F$, let
\[
G_N = \begin{cases}
\GL_N & \text{self-dual case} \\
\Res^E_F \GL_N & \text{conjugate self-dual case.}
\end{cases}
\]
In the self-dual case, consider the involution
\[
\iota: G_N := \GL_N \to \GL_N : g \mapsto g^{-T}
\]
over the number field $F$. In the conjugate self-dual case, consider
\[
\bar \iota: G_N := \Res^E_F \GL_N \to \Res^E_F \GL_N : g \mapsto \bar g^{-T},
\]
also defined over $F$. 

Next, define the matrix
\begin{equation} \label{eq long Weyl element}
w_N := \begin{pmatrix}
& & & (-1)^{N-1} \\
& & \iddots & \\
& -1  & & \\
1 & & &
\end{pmatrix}
\end{equation}
so that conjugating $\iota, \bar \iota$ by $w_N$ gives involutions $\theta, \bar \theta$ of $\GL_N$ and $\Res^E_F \GL_N$ respectively that preserve the standard pinnings. 
\begin{note}
As computational reminders, $w_N^{-1} = (-1)^{N-1} w_N$ and the automorphism $A \mapsto w_N A^T w_N^{-1}$ reflects the entries of an $N \times N$ matrix $A$ along the anti-diagonal and then multiplies the $(i,j)$th entry by the ``checkerboard pattern'' $(-1)^{i+j}$. 
\end{note}
Following \cite{Art13} and \cite{Mok15}, define twisted groups
\[
\wtd G_N = \begin{cases}
\GL_N \rtimes \theta & \text{self-dual case} \\
\Res^E_F \GL_N \rtimes \bar \theta & \text{conjugate self-dual case.}
\end{cases}
\]
The (conjugate) self-dual representations of $G_N(\A)$ or $G_N(F_v)$ are exactly the ones that extend to $\wtd G_N$. Any such irreducible representation has exactly two extensions, though in our cases we will choose one as canonical in section \ref{sec twistedaction}.

\subsubsection{Twisted Endoscopic Groups}\label{sec twistedendoscopicgroups}
(Conjugate) self-dual automorphic representations can be studied through the twisted version of Arthur's trace formula on $\wtd G_N$ as in \cite[\S3]{Art13} or \cite[\S4]{Mok15}. The trace formula decomposes into terms for elliptic endoscopic groups of $G_N$. We will only need to worry about the simple endoscopic groups $\wtd{\mc E}_\sm(N)$:

In the self-dual case, the simple twisted endoscopic groups $G \in \wtd{\mc E}_\sm(N)$ are (\cite[\S 1.2]{Art13}):
\begin{itemize}
    \item
    $N$ even: $\wh G \cong \Sp_N$ embedded in one possible way with trivial Galois action. This corresponds to $G \cong \SO_{N+1}$. 
    \item 
    $N$ even: $\wh G \cong \SO_N$ with Galois action by $\Out(G) \cong \Z/2$ through the quadratic character $\eta$. There is a unique embedding for each $\eta$.  
    This corresponds to $G \cong \SO^\eta_N$, a quasisplit orthogonal group that is not split if $\eta \neq 1$.  
    \item 
    $N$ odd: $\wh G \cong \SO_N$ with trivial Galois action. There is one possible embedding for each quadratic character $\eta$ of $\Gamma_F$ (equivalently, of $F^\times \bs \A_F^\times$).  This corresponds to $G \cong \Sp_{N-1}$. We denote the versions for different embeddings $\Sp^\eta_{N-1}$. 
\end{itemize}

In the conjugate self dual case, they can be (\cite[\S2.4.1]{Mok15}):
\begin{itemize}
    \item $\wh G \cong \GL_N$ with $\Gal(E/F)$ acting through an automorphism inner to $g \mapsto g^{-T}$. This corresponds to $G \cong U^{E/F}_N$, the quasisplit unitary group. There are two possible embeddings $\Ld G \to \Ld G_N$, and we denote the endoscopic data associated to different embeddings by $U^{\pm}_{E/F}(N)$. 
\end{itemize}

\subsection{Global Notions}\label{sec globalduality}

\subsubsection{Self-Dual Case}
Consider a number field $F$. A representation $\pi$ of $\GL_N(\A_F)$ is self-dual if $\pi \cong \pi^\vee$. The central character $\omega_\pi$ of $\pi$ is then necessarily quadratic.

As one of the key outputs of the endoscopic classification \cite{Art13}, every self-dual cuspidal automorphic representation $\pi$ of $\GL_N/F$ is assigned a group $G \in \wtd{\mc E}_\sm(N)$. If the dual group $\wh G$ is symplectic, we call $\pi$ symplectic-type and if $\wh G$ is orthogonal, we call $\pi$ orthogonal-type.
As part of the result, whenever $\pi$ is orthogonal/symplectic, then all its local factors $\pi_v$ are too since they have parameters factoring through $G$. 

Finally, \cite{Art13} gives a criterion for distinguishing orthogonal vs. symplectic  which can be taken as a more direct definition (though we will not use it): 
\begin{thm}[{\cite[Thm 1.5.3(a)]{Art13}}]
Let $\pi$ be a self-dual automorphic representation of $\GL_N/F$. If $\om_\pi$ is non-trivial, then the assigned $G$ is $\SO_N^{\om_\pi}$ or $\Sp_{N-1}^{\om_\pi}$ (through class field theory) and $\pi$ is therefore orthogonal. 

Otherwise, $\pi$ is symplectic if an only if its alternating square $L$-function $L(s, \pi, \wedge^2)$ has a pole at $s=1$ and orthogonal if and only if its symmetric square $L$-function $L(s, \pi, \Sym^2)$ has a pole at $s=1$. 
\end{thm}

\subsubsection{Conjugate Self-Dual Case}
 Consider a quadratic extension $E/F$.  A representation $\pi$ of $\GL_N(\A_E)$ is conjugate self-dual if $\pi \circ \sigma \cong \pi^\vee$ where $\sigma$ is the non-trivial element of $\Gal(E/F)$. 

 As one of the key outputs of the endoscopic classification \cite{Mok15}, every conjugate self-dual cuspidal automorphic representation $\pi$ of~$\Res^E_F \GL_N$ is assigned a group~$G \in \wtd{\mc E}_\sm(N)$. If this group is~$U_N^{(-1)^{N-1}}$, we call $\pi$ conjugate orthogonal. Otherwise, we call $\pi$ conjugate symplectic.  As part of the result, whenever $\pi$ is conjugate orthogonal/symplectic, then all its local factors $\pi_v$ are too since they have parameters factoring through $G$. 

Similar to the above, \cite{Mok15} gives a criterion for distinguishing conjugate orthogonal vs. conjugate symplectic which can be taken as a more direct definition (which we again will not use):
\begin{thm}[{\cite[Thm 2.5.4(a)]{Mok15}}]
Let $\pi$ be conjugate self-dual automorphic representation of $\Res^E_F \GL_N$. Then $\pi$ is conjugate symplectic if an only if its Asai $L$-function $L(s, \pi, \mathrm{Asai}^{-})$ has a pole at $s=1$ and conjugate orthogonal if and only if its Asai $L$-function $L(s, \pi, \mathrm{Asai}^{+})$ has a pole at $s=1$. 
\end{thm}

\subsubsection{Restrictions on $F$}
Finally, to keep the ``discreteness'' at infinity analogy with $\GL_2/\Q$-modular forms, we will only study cases when the relevant $G \in \wtd{\mc E}_\sm(N)$ have discrete series at infinity. Therefore, we make the following assumption:

\begin{assumption}\label{assump EF}
In the conjugate self-dual case assume that $F$ is totally real. In the conjugate self-dual case assume that $E/F$ is CM: a totally imaginary extension of a totally real field.
\end{assumption}

Outside of the $G = \SO_{2n}^\eta$ case, Assumption \ref{assump EF} implies that $G$ has discrete series at infinity.  In the $G = \SO_{2n}^\eta$ case, Assumption \ref{assump EF} gives that $G_\infty$ has discrete series if and only if for each $v|\infty$, the following holds: $\eta|_{\Gal_{F_v}}$ is trivial if and only if $n$ is even. As we will see in Proposition~\ref{simpleshapesselfdual}, only these $G$ will be relevant.

\section{New Vectors and Root Numbers} 
In this section, we exploit integral representations of $L$-functions to realize the root number of a symplectic automorphic $L$-function as the eigenvalue of a certain operator acting on the space of newforms. The main statements \ref{epsformula}, \ref{epsformulaconj} are local, and in the next section we will consider global applications.

\subsection{New Vectors}\label{sec newvectors}

\subsubsection{Basic Version}
We recall the theory of newvectors for local representation of 
$\GL_N$.  A useful reference is \cite[\S 3]{Hum20}, though results there are stated more generally than we need here.

Let $v$ be a place of our number field $F$ and $K_v = \GL_N(\mc O_v)$. 
\begin{dfn}\label{def mirahoric}
For each $i$, the mirahoric subgroup $K_1(\mf p_v^i) = K_{1,v}(i) \subseteq K_v  $ 
is
\[
K_1(\mf p_v^i) := K_{1,v}(i) := \lf\{ A \in K_v : A \cong 
\begin{pmatrix} * & \cdots & * & *\\
\vdots & \ddots & \vdots & \vdots \\
* & \cdots & * & * \\
0 & \cdots & 0 & 1
\end{pmatrix} \pmod{\mf p_v^i} \ri\}.
\]
Given an 
ideal $\mf n = \prod_v \mf p_v^{e_v}$ of $\mc O_F$, then
\[
K_1(\mf n) := \prod_v K_1(\mf p_v^{e_v}) \subseteq \prod_v \GL_N(\mc O_v).
\]
\end{dfn}

\begin{thm}[{\cite[Thm. 5]{JPSS81}}]\label{newformexistence}
Let $\pi_v$ be a generic irreducible representation of $\GL_N(F_v)$. Then there is  $i_{\pi_v}\geq 0$ such that
\[
\dim\lf((\pi_v)^{K_1(\mf p^i)}\ri) =  \begin{cases} 0 & i < i_{\pi_v} \\
1 & i = i_{\pi_v}
\end{cases}.
\]
\end{thm}

\begin{dfn}\label{oldformdimension}
We call $c(\pi_v) = \mf p_v^{i_{\pi_v}}$  
the conductor of $\pi_v$. The conductor $\mf n_\pi$ of a global automorphic representation $\pi = \bigotimes_v' \pi_v$ is defined as
\[
c(\pi) := \mf n = \prod_v c(\pi_v).
\]
As local shorthand, we will often say $c(\pi_v) = i_{\pi_v}$. 
\end{dfn} 

\begin{dfn}
If $\pi$ is a global automorphic representation of $\GL_N$, its newvector, defined only up to constant factor, is an element of the one-dimensional space $\pi^{K_1(\mf n_\pi)}$. 
\end{dfn}

\subsubsection{$\Res^E_F \GL_N$ version}\label{sec conductorsconj}
In the unitary case, $G_{N,v}$ is one of the following:
\begin{itemize}
    \item If $v$ is split, $\GL_N(F_v \otimes_F E) \cong \GL_N(F_v)^2$.
    \item If $v$ is inert or ramified, $\GL_N(F_v \otimes_F E) \cong \GL_N(E_w)$ for $w$ lying over $v$. 
\end{itemize}
Applying Theorem \ref{newformexistence} for $GL_N(F_v)^2$ and $GL_N(E_w)$ respectively lets us define local conductors as formal powers of $\mf p_v$ as follows: 
\begin{itemize}
    \item If $v$ is split: a pair $\mf p_w^i, \mf p_{w'}^j$ for integral $i,j$ and $w$, $w'$ above $v$. We represent this as $\mf p_w^{(i,j)}$ for ordered pair of integers $(i,j)$. If $\pi_v$ is conjugate self-dual, then $i = j$ so we can formally represent this as $\mf p_v^i$ for a single integer $i$. We use shorthand $c(\pi_v) = i$. 
    \item If $v$ is inert: a single $\mf p_w^i$ which we represent formally by $\mf p_v^i$ for integral $i$. We use shorthand $c(\pi_v) = i$
    \item If $v$ is ramified: a single $\mf p_w^{2i}$ which we represent formally as $\mf p_v^i$ for half-integral $i$. We use shorthand $c(\pi_v) = 2i$. 
\end{itemize}
These correspond to analogous subgroups $K_1(\mf p_w^i) = K_{1,v}(i) \subseteq G_{N,v}$ that satisfy the dimension-one property of Theorem \ref{newformexistence}. We similarly define newvectors. 

The global conductor $c(\pi)$ can therefore be thought of as an ideal of $\mc O_E$. If $\pi$ is further conjugate self-dual, then $c(\pi)$ can be represented as a formal product
\[
c(\pi) = \mf n = \prod_v \mf p_v^{k_v}
\]
where $v$ ranges over finite places of $\mc O_F$ and the $k_v$ are positive integers except when $v$ is ramified in which case they can also be half-integers. 

\begin{note}
Beware that there is a difference of a factor of $2$ between our local and global indexing at ramified primes. These can be distinguished by local indexing always being by an integer and global indexing always being by an ideal. 
\end{note}

\subsection{Rankin-Selberg Integral Representations}\label{sec RS} Our construction of test functions that pick out root numbers relies on the Rankin-Selberg integral for~$\pi$ and~$\pi'$. We recall this now as presented in \cite{cogdell2004long}, see also \cite[\S\S 2,3]{Hum20}.

\subsubsection{Whittaker Functions}
Consider the upper-triangular unipotent subgroup $U \subset GL_N$. Pick a nontrivial character $\psi$ of $F^\times \bs \A^\times$ and extend it to a character of $U(k) \bs U(\BA)$ through
\[
\psi \begin{pmatrix}
1 & x_1 &  \cdots &  \\
& \ddots  & \ddots &  \vdots \\
& & 1 & x_{n-1}  \\
& & & 1
\end{pmatrix} = \psi(x_1 + \cdots + x_{n-1}).
\]
Then, for any automorphic function $\varphi$ on $\GL_N$, we get a Whittaker function:
\[
W_\varphi(g, \psi) := W_\varphi(g) := \int_{U(F) \bs U(\A)} \varphi(xg) \psi^{-1}(x) \, dx
\]
in the Gelfand-Graev space $W(\psi)$ of functions satisfying $W_\varphi(xg) = \psi(x) W_\varphi(g)$ for all $x \in U(\A)$. Through this transformation property, we think of $\varphi \mapsto W_\varphi$ as attempting to embed $\pi$ as a $G(\A)$-rep into a ``twisted quotient'' $(U(\A), \psi) \bs G(\A)$. If this map $\pi \to W(\psi)$ is non-zero, $\pi$ is called \emph{generic}. It is well-known that all cuspidal automorphic representations of $\GL_N$ are generic, and that the embedding $\pi \into W(\psi)$ is unique if it exists. 

Therefore, If $\varphi = \prod_v \varphi_v$ is a pure tensor in $\pi = \bigotimes_v' \pi_v$, we have a factorization 
\[
W_\varphi =: \prod_v W_{\varphi_v}
\]
through a corresponding factorization of the unique $\pi \into W(\psi)$. 

We similarly define $\wh W_\varphi$ by integration against $\psi$ to get an element of $W(\psi^{-1})$.

\subsubsection{Integral Representations}
Let $\pi, \pi'$ be cuspidal automorphic representations of $\GL_N/F$ and $GL_{N-1}/F$ respectively. For any $\varphi \in \pi$ and $\varphi' \in \pi'$, define the Rankin-Selberg integral: 
\begin{equation}\label{globalRS}
I(s, \varphi, \varphi') := \int_{\GL_{N-1}(F) \bs \GL_{N-1}(\A)} \varphi\lf(\begin{pmatrix}
\m a & \\
& 1
\end{pmatrix}\ri)  \varphi'(\m a) |\det \m a|^{s - 1/2} \, d\m a. \\
\end{equation}

Rankin-Selberg theory provides a factorization of \eqref{globalRS}: if both $\varphi$, $\varphi'$ are pure tensors, then
\begin{multline}\label{RSfactorization}
I(s, \varphi, \varphi') = \prod_v I_v(s, W_{\varphi_v}, \wh W_{\varphi'_v}) \\
:= \prod_v \int_{U_{N-1,v} \bs \GL_{N-1,v}} W_{\varphi_v} \begin{pmatrix}
\m a_v & \\
& 1
\end{pmatrix}  \wh W_{\varphi'_v}(\m a_v) |\det \m a_v|_v^{s - 1/2} \, d\m a_v.
\end{multline}

Then as a special case of a theorem of Jacquet--Piatetskii-Shapiro--Shalika corrected by Matringe: 
\begin{thm}[{\cite[Thm (4)]{JPSS81}, corrected in \cite[Cor 3.3]{Mat13}}]\label{RStestvector}
Let $\pi_v$ be a generic representation of $\GL_{N,v}$ and $\varphi_v$ its newvector. Let $\pi'_v$ be an unramified representation of $\GL_{N-1,v}$ and $\varphi'_v$ its unramified vector. Normalize both so that $W_{\varphi_v}(1) = W_{\varphi'_v}(1) = 1$. Then $I(s, W_{\varphi_v}, W_{\varphi'_v}) = L(s, \pi_v \times \pi_v')$. 
\end{thm}

Moreover, we have a global functional equation
\begin{equation}\label{RSfunctionalequation}
I(s,\varphi, \varphi') = I(1-s, \varphi^\iota, (\varphi')^\iota), 
\end{equation}
where $\varphi^\iota(g) := \varphi(g^{-T}) = \varphi(w_N g^{-T})$, since $\varphi$ is automorphic.

We also have a local functional equation (\cite[Thm 6.2]{cogdell2004long} specialized to $N \times N-1$): define $W_\varphi^\iota(g) := W_\varphi(w_N g^{-T})$. Then,
\begin{equation}\label{localfunctionalequation}
I_v(1-s, W^\iota_{\varphi_v}, \wh W^\iota_{\varphi'_v}) = \om_{\pi'_v}(-1)^{N-1} \gamma(s, \pi_v \times \pi'_v, \psi) I_v(s, W_{\varphi_v}, \wh W_{\varphi'_v})
\end{equation}
defining a $\gamma$-factor $\gamma(s, \pi_v \times \pi'_v, \psi)$ that is necessarily a rational function in $q_v^{-s}$ and doesn't depend on $\varphi_v, \varphi_v'$.

\subsection{Root Numbers}\label{sec rootnumbers}
\subsubsection{$\eps$-factors and definition}
The local functional equation \eqref{localfunctionalequation} lets us define $\eps$-factors from the local $L$-functions:
\begin{equation}
\label{eq definition epsilon}
\eps(s, \pi \times \pi', \psi) := \f{\gamma(s, \pi_v \times \pi'_v, \psi) L(s, \pi_v \times \pi'_v)}{L(1 - s, \pi_v^\vee \times (\pi'_v)^\vee)}.
\end{equation} 

These are the factors whose product defines the global functional equation for $L$-functions:
\[
L(s, \pi \times \pi') = \eps(s, \pi \times \pi') L(1-s, \pi^\vee \times (\pi')^\vee). 
\]
Here 
\[
\eps(s, \pi \times \pi') = \prod_v \eps(s, \pi_v \times \pi'_v, \psi)
\]
is therefore necessarily independent of $\psi$.

Factors $\eps(s, \pi) := \eps(s, \pi, \triv)$ for a single representation $\pi$ can be defined through a similar $\GL_N \times \GL_1$-theory. Then:
\begin{thm}[{\cite[Prop 6.5, Thm 6.3]{cogdell2004long}}]
Let $\pi_v$ have conductor $\mf p_v^k$. Then
\[
\eps(s, \pi_v, \psi) = \eps(1/2, \pi_v,\psi) q_v^{-k(s - 1/2)}.
\]
\end{thm}
Globally,
\begin{equation}\label{epsconductor}
\eps(s, \pi) = \eps(1/2, \pi) |c(\pi)|^{1/2-s},
\end{equation}
which motivates:
\begin{dfn}
If $\pi$ is an automorphic representation of $\GL_N$, let the root number of $\pi$ be $\eps(1/2, \pi)$. 
\end{dfn}
As a final useful formula, if $\om_{\pi_v}$ is the central character of $\pi_v$,
\[
\eps(1/2, \pi_v, \psi) \eps(1/2, \pi_v^\vee, \psi) = \om_{\pi_v}(-1). 
\]

\subsubsection{The (conjugate) self-dual case}\label{selfdualrootnumbers}

We recall results about root numbers of (conjugate) self-dual representations (to translate statements from \cite{GGP11} to the automorphic side, note that that   
the determinant of the parameter of a representation of $\GL_{N,v}$ is the parameter of its central character).

\begin{prop}[{\cite[Thm 5.1(1)]{GGP11}}]\label{localepspm1}
Let the representation $\pi_v$ of $\GL_N(F_v)$ be self-dual with trivial central character. Then $\eps(1/2, \pi_v, \psi) = \pm 1$ and is independent of $\psi$. 
\end{prop}

\begin{prop}[{\cite[Thm 5.1(2)]{GGP11}}]\label{localepspm1conj}
Let the representation $\pi_v$ of $\GL_N(F_v \otimes_F E)$ be conjugate self-dual. Then, if $\psi \circ \sigma = \psi^{-1}$, $\eps(1/2, \pi_v, \psi) = \pm 1$. This value is independent of $\psi$ if either $N$ is even or $\pi_v$ is conjugate orthogonal.  
\end{prop}

\begin{proof}
This first statement follows from \cite[Thm 5.1(2)]{GGP11}. By the paragraph just afterwards, independence follows whenever the parameter $\varphi_v$ for $\pi_v$ has conjugate orthogonal determinant. This is equivalent to the stated conditions by the discussion after Lemma \ref{cchartransfer}.
\end{proof}

\begin{dfn} \label{def standard root number}
If $N$ is even or $\pi_v$ is conjugate orthogonal, call $\eps(1/2, \pi_v)$ for~$\psi \circ \sigma = \psi^{-1}$ as in \ref{localepspm1conj} the standard root number of 
$\pi_v$.
\end{dfn}

Next, we recall much more difficult consequences of the endoscopic classification:

\begin{thm}[{\cite[1.5.3(b)]{Art13}}]\label{globalortheps}
Let $\pi$ be a self-dual, orthogonal-type automorphic representation of $\GL_N$. Then $\eps(1/2, \pi) = 1$.
\end{thm}

\begin{proof}
Note that the trivial representation on $\GL_1$ is orthogonal-type. 
\end{proof}

\begin{thm}[{\cite[2.5.4(b)]{Mok15}}]\label{globalorthepsconj}
Let $\pi$ be a conjugate self-dual, conjugate orthogonal-type automorphic representation of $\Res^E_F \GL_N$. Then $\eps(1/2, \pi) = 1$.
\end{thm}

\begin{proof}
The trivial representation on $\Res^E_F \GL_1$ is conjugate orthogonal-type. 
\end{proof}

\subsection{Involutions for Newforms : Self-Dual Case}\label{sec involutions}
We want to understand the functional equation \eqref{RSfunctionalequation} and the variant we need to compute epsilon factors for $\varphi$ a newvector. 
Reformulating \cite[\S4.2]{Tom24} in adelic language, we define:
\begin{itemize}
\item
$\mf n, \mf p^k$ will as shorthand also be some chosen lifts of the corresponding ideals to $F_v^\times$ or $\A^\times$ determined by our choices of uniformizers $\varpi_v$. All terms subscripted with these ideals will depend on this choice. 
\item
$J_{\mf p_v^k} := J_{v,k} := \diag(\varpi_v^k, \cdots, \varpi_v^k, 1) \in \GL_N(F_v)$.
\item 
$J_{\mf n} \in \GL_N(\A)$ for $\mf n$ an ideal of $\mc O_F$ the analogous global product of $J_{\mf p_v^k}$.  
\item
$\iota_{\mf p_v^k} = \iota_{v,k}: \GL_N(F_v) \to \GL_n(F_v) : g \mapsto  g^{-T}J_{\mf p_v^k}$,
\item
$\iota_{\mf n} : \GL_N(\A) \to \GL_N(\A)$ the global version of $\iota_{\mf p_v^k}$,
\item
$\iota := \iota_{\mc O_F}$ (consistent with the previous definition). 
\item
$\varphi^{\iota_\mf n} : = \varphi \circ \iota_{\mf n}$ for global test vectors $\varphi$. 
\end{itemize}
We also define an involution on local and global Whittaker models: if $w_N$ is the long Weyl element as in \eqref{eq long Weyl element}, then
\begin{itemize}
    \item $W^{\iota_\mf n}_\varphi(g) := W_\varphi(w_N g^{\iota_\mf n}) := W_\varphi(w_N g^{-T})$,
    \item $W^{\iota_{v,k}}_{\varphi_v}(g) := W_{\varphi_v}(w_N g^{\iota_{v,k}}) := W_\varphi(w_N g^{-T} J_{v,k})$. 
\end{itemize}
We justify this notation since globally, the left-invariance of $\varphi$ under $G(F)$ gives:
\begin{align*}
W_{\varphi^{\iota_\mf n}}(g) &= \int_{U(F) \bs U(\A)} \varphi(x^{-T}g^{-T} J_{\mf n}) \psi^{-1}(x) \, dx \numberthis \label{iotaconsistency} \\ 
&=  \int_{U(F) \bs U(\A)} \varphi(w_N x^{-T}g^{-T} J_{\mf n}) \psi^{-1}(x) \, dx \\
&=  \int_{U(F) \bs U(\A)} \varphi((w_N x^{-T} w_N^{-1}) w_N g^{-T} J_{\mf n}) \psi^{-1}(x) \, dx \\
&=  \int_{U(F) \bs U(\A)} \varphi(x w_N g^{-T} J_{\mf n}) \psi^{-1}(x) \, dx \\
&= W^{\iota_\mf n}_\varphi(g).
\end{align*}
We use here that the map $x \mapsto w_N x^{-T} w_N^{-1}$ preserves $\psi$ since
\begin{align*}
w_N \begin{pmatrix}
1 & x_1 &  \cdots &  \\
& \ddots  & \ddots &  \vdots \\
& & 1 & x_{n-1}  \\
& & & 1
\end{pmatrix}^{-T} w_N^{-1} &= 
w_N \begin{pmatrix}
1 & -x_1 &  \cdots &  \\
& \ddots  & \ddots &  \vdots \\
& & 1 & -x_{n-1}  \\
& & & 1
\end{pmatrix}^T w_N^{-1} \\ &= 
\begin{pmatrix}
1 & x_{n-1} &  \cdots &  \\
& \ddots  & \ddots &  \vdots \\
& & 1 & x_1  \\
& & & 1
\end{pmatrix}.
\end{align*}
It also preserves the Haar measure on $U$ since it has finite order. 

The same invariance of $\psi$ gives that locally, if $x\in U(F_v)$, 
\begin{multline}\label{localiotaconsistency}
W^{\iota_{v,k}}_{\varphi_v} (xg)  = W_{\varphi_v}(w_N x^{-T} g^{-T} J_{v,k}) = W_{\varphi_v} ((w_N x^{-T} w_N^{-1}) w_N g^{-T} J_{v,k})  \\ = \psi(w_N x^{-T} w_N^{-1}) W_{\varphi_v} (w_N g^{-T} J_{v,k}) = \psi(x) W^{\iota_{v,k}}_{\varphi_v}(g).
\end{multline}
so that $\iota_{v,k}$ preserves the Gelfand-Graev representation. Next, we show that local Whittaker-newforms are actually eigenvectors of this involution:
\begin{lem}\label{newformeigenvalue}
Let $\pi_v$ be a self-dual generic irreducible representation of $\GL_{N,v}$ of conductor $k$ and $\varphi_v \in \pi_v$ a newvector. Then $W^{\iota_{v,k}}_{\varphi_v} = \pm W_{\varphi_v}$ where both are in $W(\psi)$.
\end{lem}

\begin{proof}
By \eqref{localiotaconsistency}, $W^{\iota_{v,k}}_{\varphi_v}$ is in $W(\psi)$. Since $\iota_{v,k}$ preserves invariance under right-translation by $K_{1,v}(k)$, we get that $W^{\iota_{v,k}}_{\varphi_v}$ is also fixed by $K_{1,v}(k)$. It generates a space isomorphic to $\pi_v^\vee \cong \pi_v$ since for $h \in GL_{N,v}$,
\begin{multline*}
    W^{\iota_{v,k}}_{h\varphi_v}(g) = W_{\varphi_v}(w_Ng^{-T}J_{v,k} h^{-1}) = W_{\varphi_v}(w_Ng^{-T}J_{v,k} h^{-1} J_{v,k}^{-1} J_{v,k}) =  \\ W_{\varphi_v}(w_N(gJ^{-1}_{v,k} h^{T}J_{v,k})^{-T}J_{v,k})  = W^{\iota_{v,k}}_{\varphi_v}(gJ_{v,k}^{-1} h^{T}J_{v,k}) = (J_{v,k} h^{-T}J_{v,k}^{-1})W^{\iota_{v,k}}_{\varphi_v}(g).
\end{multline*} 
In total, by uniqueness of Whittaker models and newvectors, $W_{\varphi_v}^{\iota_{v,k}}$ has to be constant multiple of $W_{\varphi_v}$. Since $g \mapsto w_N g^{-T} J_{v,k}$ is an involution, this multiple has to be $\pm 1$. 
\end{proof}

Classically for $\GL_2$, there is a link between the local eigenvalues from Lemma~\ref{newformeigenvalue} and the root number coming from the use of global $\iota_\mf n$ in the proof of the functional equation for Hecke $L$-functions (e.g. \cite[5.10.2]{DS05}). The correct generalization to $\GL_N$ involves $N \times (N-1)$ Rankin-Selberg $L$-functions. The standard functional equation only involves $\iota$; we give a variant for $\iota_\mf n$:
\begin{lem}\label{iotafn}
Let $\pi_v, \pi'_v$ be generic, unitary irreducible representations of $\GL_{N, v}$, $\GL_{N-1, v}$  with central characters $\om, \om'$ respectively. Let $\varphi_v, \varphi_v' \in \pi_v, \pi_v'$ respectively. Then
\begin{multline*}
I_v(1-s, W^{\iota_{v,k}}_{\varphi_v}, \wh W^{\iota}_{\varphi'_v})  \\= \om'(-1)^{N-1}  \om'(\varpi_v^{-k}) q_v^{k(N-1)(s-1/2)} \gamma(s, \pi_v \times \pi'_v, \psi) I_v(s, W_{\varphi_v}, \wh W_{\varphi'_v}).
\end{multline*}
\end{lem}

\begin{proof}
The local functional equation \eqref{localfunctionalequation} gives:
\[
I_v(1-s, W^\iota_{\varphi_v}, \wh W^{\iota}_{\varphi'_v}) = \om'(-1)^{N-1}  \gamma(s, \pi_v \times \pi'_v, \psi) I_v(s, W_{\varphi_v}, \wh W_{\varphi'_v}).
\]
However, using that $\varphi^\iota$ has central character $\om^{-1}$ and that $\iota_{v,k}(g)=\iota(g J^{-1}_{v,k})$: 
\begin{align*}
&I_v(s, W^{\iota_{v,k}}_{\varphi_v}, \wh W^{\iota}_{\varphi'_v}) \\
&= \int_{U_v \bs \GL_{N-1, v}} W^\iota_{\varphi_v} \begin{pmatrix}
\varpi_v^{-k} \m a_v & \\
& 1
\end{pmatrix} \wh W^\iota_{\varphi'_v}(\m a_v) |\det \m a_v|^{s - 1/2} \, d\m a_v \\
&= 
\om'(\varpi_v^{-k}) q_v^{k(N-1)(1/2-s)} \\ & \qquad \qquad \int_{U_v \bs \GL_{N-1, v}} W^\iota_{\varphi_v} \begin{pmatrix}
\varpi^{-k} \m a_v & \\
& 1
\end{pmatrix} \wh W^\iota_{\varphi'_v}(\varpi^{-k} \m a_v) |\det \varpi^{-k} \m a_v|^{s - 1/2} \, d(\varpi^{-k} \m a_v) 
\\
&=  \om'(\varpi_v^{-k}) q_v^{k(N-1)(1/2-s)} \int_{U_v \bs \GL_{N-1, v}} W^\iota_{\varphi_v} \begin{pmatrix}
\m a_v & \\
& 1
\end{pmatrix} \wh W^\iota_{\varphi'_v}(\m a_v) |\det \m a_v|^{s - 1/2} \, d\m a_v \\
&=  \om'(\varpi_v^{-k}) q_v^{k(N-1)(1/2-s)} I_v(s, W^\iota_{\varphi_v}, \wh W^{\iota}_{\varphi'_v}),
\end{align*}
so the result follows.
\end{proof}

Next, we show that for $\pi_v$ unramified the eigenvalue of $\iota$ is always $1$:

\begin{prop}\label{epsformulaunram}
Let $\pi_v$ be an unramified, self-dual generic irreducible representation of $\GL_{N,v}$ and $\varphi_v$ an unramified vector. Then $W_{\varphi_v} = W^\iota_{\varphi_v}$ independently of $\psi$.
\end{prop}

\begin{proof}
By the Iwasawa decomposition and since both sides of the equality are unramified, it suffices to check equality on elements of the form $u \lb(\varpi)$ for~$u \in U_{N,v}$ and $\lb$ a cocharacter of the diagonal torus. Here: 
\begin{multline*}
W^\iota_{\varphi_v}(u\lb(\varpi)) = W_{\varphi_v}(w_N u^{-T} \lb(\varpi)^{-T}) = W_{\varphi_v}(w_N u^{-T} w_N^{-1} \lb'(\varpi) w_N)  \\ = W_{\varphi_v}(w_N u^{-T} w_N^{-1} \lb'(\varpi)),
\end{multline*}
where if $\lb = (\lb_1, \dotsc, \lb_n)$, then $\lb' := -w_N \lb w_N^{-1} =  (-\lb_n, \dotsc, -\lb_1)$. 

Factoring out the $u$ terms, it therefore suffices to show that
\[
\psi(u) W_{\varphi_v}(\lb(\varpi)) = \psi(w_N u^{-T} w_N^{-1})W_{\varphi_v}(\lb'(\varpi)) \quad (= \psi(u) W_{\varphi_v}(\lb'(\varpi)).)
\]
Since $\pi_v$ is self-dual, its Satake parameter $(q_v^{\alpha_1}, \dotsc, q_v^{\alpha_N})$ is the same as the inverse $(q_v^{-\alpha_N}, \dotsc, q_v^{-\alpha_1})$. The result therefore follows from the Casselman-Shalika formula (e.g. \cite[Thm 6.5.1]{HKP10}).
\end{proof}

Bootstrapping off this gives our key result:
\begin{thm}\label{epsformula}
Let $\pi_v$ be a generic self-dual representation of $\GL_{N,v}$ of conductor $k$. Let $\varphi_v \in \pi_v$ be a newvector and let $\tau^\psi_{v,k}(\varphi_v)  = \pm 1$ be defined by $W_{\varphi_v} = \tau^\psi_{v,k}(\varphi_v)W^{\iota_{v,k}}_{\varphi_v}$ as in Lemma \ref{newformeigenvalue}. Then
\[
\tau^\psi_{v,k}(\varphi_v) = \eps(1/2, \pi_v, \psi)^{N-1}.
\] 
If we further assume that $\pi_v$ has trivial central character and $N$ is even, then 
\[
\tau^\psi_{v,k}(\varphi_v) = \eps(1/2, \pi_v)
\]
and neither depend on the choice of $\psi$.

\end{thm}

\begin{proof}
Pick an auxilliary tempered, unramified, self-dual representation~$\pi'_v$ of~$\GL_{N-1,v}$ with Satake parameter~$(q_v^{\alpha_1}, \dotsc, q_v^{\alpha_{N-1}})$. Let $\varphi'_v \in \pi'_v$ be the unramified vector. Normalize $\varphi_v, \varphi'_v$ so that $W_{\varphi_v}(1) = W_{\varphi'_v}(1) = 1$. Then by Theorem \ref{RStestvector},
\[
I_v(s, W_\varphi, W_{\varphi'}) = L_v(s, \pi_v \times \pi_v').
\]
We can further assume that $\pi'_v$ has trivial central character. Then the definition of epsilon factors \eqref{eq definition epsilon} and Lemma \ref{iotafn} give
\[
I_v(s, W^{\iota_\mf n}_{\varphi_v}, W^{\iota}_{\varphi'_v}) = \eps(s, \pi_v \times \pi'_v,\psi)  q_v^{k(N-1)(s-1/2)}L_v(s, \pi_v \times \pi'_v).
\]
Since $\pi'_v$ is self-dual, $W^{\iota}_{\varphi'_v} = W_{\varphi'_v}$ by Proposition \ref{epsformulaunram}. Together with Lemma \ref{newformeigenvalue}, this shows that
\begin{multline*}
\tau^\psi_{v,k}(\varphi_v) I_v(s, W_{\varphi_v}, W_{\varphi'_v}) = \eps(s, \pi_v \times \pi'_v,\psi)  q_v^{k(N-1)(s-1/2)} L_v(s, \pi_v \times \pi'_v) \\
\implies \eps(s, \pi_v \times \pi'_v, \psi) = \tau^\psi_{v,k}  |\mf n|^{(N-1)(1/2-s)}. 
\end{multline*}
However, since $\pi'$ is unramified, we also have
\[
L_v(s, \pi_v \times \pi'_v) = \prod_{i=1}^{N-1} L_v(s + \alpha_i, \pi_v).
\]
The definition of $\eps$-factors and multiplicativity of $\gamma$-factors (e.g. \cite[\S3.1.6]{cogdell2004long}) then gives:
\[
\eps(s, \pi_v \times \pi'_v, \psi) = \prod_{i=1}^{N-1} \eps(s + \alpha_i, \pi_v, \psi) = \eps(s, \pi_v,\psi)^{N-1}
\]
using at the end that $\sum_i \alpha_i=0$ since $\pi'_v$ is self-dual. This is the first claim. 

If $\pi_v$ is self-dual with trivial central character, by \ref{localepspm1} and \eqref{epsconductor},  
\[
\eps(s, \pi_v) = \eps(1/2, \pi_v) q_v^{k(1/2-s)} = \pm q_v^{k(1/2-s)}.
\] 
This implies that $\eps(1/2,s) = \tau^\psi_{v,k}(\varphi_v)$ if $N-1$ is odd. 
\end{proof}

\begin{note}
The above theorem does not compute the local root number when $N$ is odd or when $\pi_v$ has central character. However,  in those cases, any global $\pi$ is orthogonal-type and therefore satisfies $\eps(1/2, \pi) = 1$ by Theorem \ref{globalortheps}. Therefore this doesn't matter for our eventual global application. 
\end{note}

\subsection{Involutions for Newforms : Conjugate Self-Dual Case}\label{sec involutionsconj}
We perform a similar argument for conjugate self-dual $\pi$. Let $E/F$ be a quadratic extension with~$\sigma : \alpha \mapsto \bar \alpha$ the corresponding Galois conjugation action. 

We fix our character $\psi : E^\times \to \C^\times$ defining Whittaker functions to be self-conjugate, i.e. $\psi \circ \sigma =: \psi^\sigma = \psi$. It is very important to note that this is \emph{not} the standard choice $\psi^\sigma = \psi^{-1}$ used in \cite{GGP11}. This annoying technical detail will add some extra central character factors that do not appear in the self-dual case.

For appropriate $v, \mf p_v, \mf n$, we similarly define:
\begin{itemize}
\item
$J_{\mf p_v^k} := \diag(\varpi_v^{k}, \cdots, \varpi_v^{k}, 1) \in \GL_N(F_v) \subseteq \GL_N(F_v \otimes_F E)$  (note that if $v$ is ramified, we need to instead choose a uniformizer $\varpi_v^{1/2} =\varpi_w$ for $E_w$), 
\item
$J_{\mf n} \in \GL_N(\A_F) \subseteq \GL_N(\A_E)$ for $\mf n$ an ideal of $\mc O_F$, the analogous global product of $J_{\mf p_v^k}$,
\item 
$J_{v,k}$ the locally-indexed version of $J_{\mf p_v^k}$, as in \S\ref{sec conductorsconj}.
\item
$\bar \iota_{\mf p_v^k} : \GL_N(F_v \otimes_F E) \to \GL_N(F_v \otimes_F E) : g \mapsto  \bar g^{-T}J_{\mf p_v^k}$, and $\bar \iota_{v,k}$ the locally-indexed version.
\item
$\bar \iota_{\mf n} : \GL_N(\A_E) \to \GL_N(\A_E)$ the global version of $\bar \iota_{\mf p_v^k}$. 
\item
$\bar \iota := \bar \iota_{\mc O_F}$
\item
$\varphi^{\bar \iota_\mf n} : = \varphi \circ \bar \iota_{\mf n}$ for global test vectors $\varphi$. 
\end{itemize}
We can also define for Whittaker functions: 
\begin{itemize}
    \item $W^{\bar \iota_\mf n}_\varphi(g) := W_\varphi(w_N g^{\bar \iota_\mf n}) := W_\varphi(w_N \bar g^{-T})$,
    \item $W^{\bar \iota_\mf n}_{\varphi_v}(g) := W_{\varphi_v}(w_N g^{\bar \iota_\mf n}) := W_\varphi(w_N \bar g^{-T} J_{v,k})$.
\end{itemize}
Following our choice of $\psi^\sigma = \psi$, this satisfies analogues of \eqref{iotaconsistency} and \eqref{localiotaconsistency}:
\[
W_{\varphi^{\bar \iota_\mf n}_v} = W^{\bar \iota_\mf n}_{\varphi_v}, \qquad 
W^{\bar \iota_{v,i}}_{\varphi_v}(xg) = \psi(x) W^{\bar \iota_{v,k}}_{\varphi_v}(g).
\]

We also warn that if $\varpi \neq \bar \varpi$, then $\bar \iota_{v,i}$ isn't necessarily an involution:
\[
\bar\iota_{v,k} \bar \iota_{v,k} g = g \diag((\varpi/\overline \varpi)^{-k}, \dotsc, (\varpi/\overline \varpi)^{-k}, 1). 
\]
However, this extra factor belongs to all level subgroups~$K_{1,v}(k)$, so precomposition with~$\bar \iota_{v,k}$ does act as an involution on all new- and oldvector spaces. 

Using all this, we continue the analogy: 
\begin{lem}\label{newformeigenvalueconj}
Let $\pi_v$ be a generic conjugate self-dual irreducible representation of $\GL_N(F_v \otimes_F E)$ of conductor $k$ and $\varphi_v \in \pi_v$ a newvector. Then, provided that $\psi^\sigma = \psi$, we have $W^{\bar \iota_{v,k}}_{\psi, \varphi_v} = \pm W_{\psi, \varphi_v}$. 
\end{lem}

\begin{proof}
This is a very similar argument to Lemma \ref{newformeigenvalue}, after noting that precomposition with $\bar \iota_{v,i}$ does act as an involution on 
newvectors.
\end{proof}

For the functional equation: 
\begin{lem}\label{iotafnconj}
Assume $\psi^\sigma = \psi$. Let $\pi_v, \pi'_v$ be generic, unitary irreducible representations of $\GL_N(F_v \otimes_F E), \GL_{N-1}(F_v \otimes_F E)$ with central characters $\om, \om'$ respectively. Let $\varphi_v, \varphi_v' \in \pi_v, \pi_v'$ respectively. Then
\begin{multline*}
I_v(1-s, W^{\bar \iota_{v,k}}_{\varphi_v}, \wh W^{\bar \iota}_{\varphi'_v}) \\= \om'(-1)^{N-1}  \bar \om'(\overline \varpi_v^{-k}) {q_v}^{k(N-1)(1/2-s)} \gamma(s, \pi_v \times \pi'_v, \psi) I_v(s, W_{\varphi_v}, \wh W_{\varphi'_v}).
\end{multline*}
\end{lem}

\begin{proof}
As in the proof of Lemma \ref{iotafn}, we have that
\[
I_v(s, W^{\bar \iota_{v,k}}_{\varphi_v}, \wh W^{\bar \iota}_{\varphi'_v}) = \bar \om'(\overline \varpi^{-k}) q_v^{k(N-1)(1/2-s)}I_v(s, W^{\bar \iota}_{\varphi_v}, \wh W^{\bar \iota}_{\varphi'_v})
\] 
noting that the $J_{v,k}$ gets conjugated in the $\bar \iota_{v,k}$-version of the computation (Note also that this conjugation does nothing unless $v$ is ramified and $k$ is a half-integer). However, conjugation preserves Haar measures since it is an involution, so we also have that
\[
I_v(s, W^{\bar \iota}_{\varphi_v}, \wh W^{\bar \iota}_{\varphi'_v}) = I_v(s, W^{ \iota}_{\varphi_v}, \wh W^{\iota}_{\varphi'_v}).
\]
The result follows as in Lemma \ref{iotafn}.
\end{proof}

Continuing with the analogous unramified input to \ref{epsformulaunram}:

\begin{prop}\label{epsformulaunramconj}
Let $v$ be non-split and $\pi_v$ be an unramified, conjugate self-dual irreducible representation of $G_{N,v} = \GL_N(F_v \otimes_F E)$ and $\varphi_v$ an unramified vector. Then $W_{\varphi_v} = W^{\bar \iota}_{\varphi_v}$ whenever $\psi \circ \sigma = \psi$.
\end{prop}

\begin{proof}
This follows from \ref{epsformulaunram} applied to $\GL_N(E_w)$: for all $u \in U$, $\psi(\bar u) = \psi(u)$ and for all cocharacters $\lb$, $W_{\varphi_v}(\lb(\varpi)) = W_{\varphi_v}(\lb(\overline \varpi))$ since $\varpi/\overline \varpi$ has norm $1$. 
%
\end{proof}

We now arrive at the conjugate self-dual theorem identifying $\bar \tau_n(\varphi_v)$ with the appropriate root number. This will enable us to construct test functions counting automorphic forms of conductor $\mf n$ weighted by their root number. 

\begin{thm}\label{epsformulaconj}
Assume $\psi \circ \sigma = \psi$. Let $\pi_v$ be a generic conjugate self-dual representation of $\GL_N(F_v \otimes_F E)$ of conductor $c(\pi_v) = k$ (as defined in \S\ref{sec conductorsconj}). Let~$\varphi_v \in \pi_v$ be a newvector and let $\bar \tau^\psi_{v,k}(\varphi_v)  = \pm 1$ be defined by $W_{\varphi_v} = \bar \tau^\psi_{v,k}(\varphi_v)W^{\iota_{v,k}}_{\varphi_v}$ as in Lemma \ref{newformeigenvalueconj}. Then
\[
 \bar \tau^\psi_{v,k}(\varphi_v) = \begin{cases} \eps(1/2, \pi_v, \psi)^{N-1} & v\text{ non-split} \\
 1 & v \text{ split.}
 \end{cases}
\]

Furthermore, in the non-split case, if either $N$ is even or $\pi_v$ is conjugate-orthogonal, let $\eps(1/2, \pi_v)$ be the standard root number as in \ref{def standard root number}. 
Then:
\[
 \bar \tau^\psi_{v,k}(\varphi_v) = \begin{cases} \om(a) \eps(1/2, \pi_v) & N \text{ even} \\
 1 & N \text{ odd}, \, \pi_v \text{ conjugate-orthogonal}
 \end{cases}
\]
where $\omega$ is the central character of $\pi$ and $a \in F_v \otimes_F E$ is any element such that $\sigma(a) = -a$.
\end{thm}

\begin{proof}
If $v$ is split, then $\pi_v$ is a representation of the form $\rho_v \boxtimes \rho_v^\vee$ of $\GL_N(F_v)^2$. If~$W_f$ is a newvector in the $\psi$-Whittaker model for $\rho_v$, then $W_f \boxtimes W^{\iota_\mf n}_f$ is a newvector for $\pi_v$. Since $\iota_\mf n$ acts as an involution, $\bar \iota_\mf n$ acts as the identity on this newvector. 

The non-split case follows very similarly to the proof of Theorem \ref{epsformula}. Explaining the details that change: We pick an auxiliary tempered, unramified, conjugate self-dual representation $\pi'$ of $\GL_{N-1}(F_v \otimes_F E)$ with trivial central character and $\varphi' \in \pi'$. Using \ref{iotafnconj}/\ref{epsformulaunramconj} instead of \ref{iotafn}/\ref{epsformulaunram}, we get that
\[
\eps(s, \pi_v \times \pi'_v, \psi) = \bar \tau^\psi_{v,k}(\varphi_v)  q_w^{k(N-1)(1/2-s)}.
\]

When $v$ is non-split, we can also show by the same calculations as in \ref{epsformula} that
\[
\eps(s, \pi_v \times \pi_v', \psi) = \eps(s, \pi_v, \psi)^{N-1}.
\]
Note that for unramified representations at non-split places $v$, conjugate self-dual is equivalent to self-dual since the valuation is unchanged by conjugation.

The second statement follows since
\[
\eps(1/2, \pi_v, \pi, \psi) = \om(a) \eps(1/2, \pi_v, \psi_0)
\]
and that we always have that $\eps(1/2, \pi_v, \psi_0) = \pm 1$ and is independent of $\psi_0$ under our conditions by \ref{localepspm1conj}. We also note that $\om(a) = \pm 1$ when $\pi_v$ is conjugate-orthogonal or $N$ is even. 
\end{proof}

\begin{note}
In the split case, let $\pi_v = \rho_v \boxtimes \rho_v^\vee$ of $\GL_N(F_v)^2$. 
 Then 
 \[
 \eps(1/2, \pi_v, \psi) = \eps(1/2, \rho_v, \psi)\eps(1/2, \rho_v^\vee, \psi) = \om_{\rho_v}(-1). \]
 If $a \in F_v^2$ such that $a = - \bar a$, then $a$ is of the form $(x, -x)$ and 
 \[
 \om_\pi(a) = \om_{\rho_v}(x)\om_{\rho_v}(-x)^{-1} = \om_{\rho_v}(-1).
 \]In total $\eps(1/2, \pi_v, \psi) = \om_{\pi_v}(a)$. In other words, the value of $\bar\tau^\psi_{v,k}(\varphi_v)$ for $v$ split is \emph{not} $\eps(1/2,\pi_v,\psi)^{N-1}$, but rather $\eps(1/2,\pi_v)^{N-1}:=\eps(1/2,\pi_v,\psi_0)^{N-1}$. 

In particular, the split case differs from the non-split case by a factor of $\om(a)^{N-1}$.
\end{note}

\begin{note}
It is at this point that the restriction to $N$ even in our main Theorem \ref{mainthmconj} appears. In the self-dual case, we get lucky that when $N$ is odd, all representations are orthogonal-type and therefore have global root number $1$. Here, however, there are conjugate-symplectic-type representations when $N$ is odd which can have arbitrary root number. Our techniques are unable to address this case. 
\end{note}

\section{Test Functions}\label{sectiontestfunctions}
Using the results of the previous section, we now construct test functions that let us compute (twisted) counts of representations.

\subsection{Twisted Action}\label{sec twistedaction}
Recall from \eqref{eq long Weyl element} the matrix $w_N$ which conjugated $\iota, \bar \iota$ into the involutions $\theta, \bar \theta$ preserving the standard pinning. We define twisted groups:
\[
\wtd G_N = \begin{cases}
    G_N \rtimes \theta & \text{self-dual or orthogonal/symplectic case} \\
    G_N \rtimes \bar \theta & \text{conjugate self-dual or unitary case.}
\end{cases}
\] 
We first realize our involutions $\iota_\mf n, \bar \iota_\mf n$ through conjugation in the twisted group $\wtd G_N$.

If $\pi$ is a (conjugate) self-dual generic automorphic representation of $G_N$, the twisted group $\wtd G_N$ acts on $\pi$ since $\theta$ (or $\bar \theta$) acts on its Whittaker model through precomposition. Define $\wtd \pi$ to be this extended representation of $\wtd G_N$. This lets us define twisted traces $\tr_{\wtd \pi} \td f$ for test functions $\td f$ on $\wtd G_N$. Note that this all depends on the choice of $\psi$ used to define the Whittaker model. 

\begin{note}
In the conjugate self-dual case, this choice of $\psi$ is a priori different than in the previous section. Specifically: before, we were thinking of the local~$G_{N,v}$ as~$\GL_N(F_v \otimes_F E)$ and defining Whittaker models through characters $\psi$ of~$F_v \otimes_F E$. Now, we are working with the Whittaker models of the algebraic group $G_{N,v}/F_v$ depending on a choice of character $\psi$ of $F_v$. This choice in turn determines a choice of character of $F_v\otimes_F E$ through composition with the trace map and only produces characters of $F_v\otimes_F E$ satisfying $\psi^\sigma = \psi$, as was required in \S \ref{sec involutionsconj}. 
\end{note}

Since each of the $\pi_v$ are also generic, each $\wtd G_{N,v}$ also acts on $\pi_v$ through precomposition in the Whittaker model. This defines local twisted traces $\tr_{\td \pi_v} \td f_v$ for test functions $\td f_v$ on $\wtd G_{N,v}$. Then, as is well-known: 

\begin{prop}\label{twistedtracefactor}
 Let $\td f = \prod_v \td f_v$ be a test function on  $\wtd G_N$. Then for all generic automorphic representations $\pi$ of $G_N$
 \[
 \tr_{\wtd \pi} \td f = \prod_v \tr_{\wtd \pi_v} \td f_v. 
 \]
\end{prop}

\begin{proof}
Through the isomorphism of $G(\A)$-reps from $\pi$ to its global Whittaker model and computation \eqref{iotaconsistency}, we can calculate $\tr_{\td \pi} \td f$ as the trace of a convolution operator on the global Whittaker model and then factor. 
\end{proof}

Our consideration of twisted traces is motivated by the multiplication formulas:
\begin{equation}\label{twistedmultform}
\begin{cases}
   (J^{-1}_{v,k} w_N \rtimes \theta) W_\varphi = W_{\varphi_v}^{\iota_{v,k}} & \text{self-dual}, \\
   (J^{-1}_{v,k} w_N \rtimes \bar \theta) W_\varphi = W_{\varphi_v}^{\bar \iota_{v,k}} & \text{conjugate self-dual.}
\end{cases}
\end{equation}
---i.e. we can relate the eigenvalue of the appropriate $\iota$ on a newvector to traces against $J w_N \rtimes \theta$. 

In particular, if we define
\[
\td f_{v,k} = (J^{-1}_{v,k} w_N \rtimes \theta) \bar \1_{K_{1,v(k)}},
\]
and assume that $\pi_v$ has conductor $\mf p_v^k$ with newvector $\varphi_v$, we get:
\[
\tr_{\td \pi_v} \td f_{v,k} = \tr_{\td \pi_v^{K_{1,v}(k)}} (\iota_{v,k}) = \tau^\psi_{v,k}(\varphi_v).
\]
from \ref{epsformula}. Therefore, $\td f_{v,k}$ works as our test function to probe $\eps(1/2, \pi_v)$ if $\pi_v$ has conductor exactly $\mf p_v^k$. However, it also picks up the traces against oldform spaces for $\pi_v$ of conductor less than $k$.

\subsection{Twisted Traces on Spaces of Oldforms}\label{sec oldforms}
We now study the extra oldform contributions. 

\subsubsection{Description of the Oldform Space}
We first recall results from \cite{Ree91} describing the space of $K_{1,v}(k)$-fixed vectors in any generic representation $\pi_v$ of $\GL_{N,v}$ in terms of a Hecke algebra for $GL_{N-1,v}$. This lets us compute the action of~$J_{v,i} w_N \rtimes \theta$ on all oldforms instead of just the newvectors.

Let $\ms H_{N-1}$ be the unramified (with respect to $K = \GL_{N-1}(\mc O_v)$) Hecke algebra for $\GL_{N-1,v}$. Let $T \subseteq \GL_{N,v}$ be the diagonal torus so that $\ms H_{N-1}$ has a Cartan basis of double coset indicator functions $\1_{K \lb(\varpi_v) K}$ for $\lb$ a dominant cocharacter of $T$. 

Every such $\lambda$ is a unique linear combination of the fundamental coweights
\[
\lb_i := (\overbrace{1, \dotsc, 1}^i, \overbrace{0, \dotsc, 0}^{N-1 - i}) \qquad 1 \leq i \leq N-1
\]
with all coefficients positive except for that of $\lb_{N-1}$. For $a_1, \dotsc, a_{N-1} \in \Z^n$, define
\[
T(a_1, \dotsc, a_{N-1}) := \1_{K (a_1 \lb_1 + \cdots + a_{N-1} \lb_{N-1})(\varpi_v) K} \in \ms H_{N-1}
\]
and
\[
\ms H_{N-1}^{\leq k} := \left\langle T(a_1, \dotsc, a_{N-1}) : a_i \geq 0, \sum a_i \leq k  \right\rangle \subseteq \ms H_{N-1}.
\]

There is a normalized action of $\ms H_{N-1}$ on $\pi_v$ given by 
\[
h * \varphi := \int_{\GL_{N-1,v}} h(g) \begin{pmatrix} g^{-1} & \\ & 1 \end{pmatrix} \varphi |\det g|^{1/2} \, dg.
\]
\begin{thm}[Slight modification of {\cite[Thm. 1]{Ree91}}]\label{oldforms}
Let $\pi_v$ be a generic irreducible representation of $\GL_N(F_v)$ with conductor $c$ and newvector $\varphi_v$. Then for $k \geq 0$, the map 
\[ 
h \mapsto h * \varphi_v
\] 
induces a linear isomorphism
\[
\Psi_{\pi_v}: \ms H_{N-1}^{\leq k} \to \pi_v^{K_1(\mf p_v^{c+k})}.
\]
\end{thm}

\begin{proof}
The Satake transform lets us change between the basis for $\ms H_{N-1}^{\leq k}$ described here and the one in \cite{Ree91}. 
\end{proof}

\subsubsection{Twisted Action on Oldvectors: Self-Dual Case}
Using the isomorphism $\Psi_{\pi_v}$, we can compute a key formula the self-dual case:
\begin{prop}\label{twistedactionoldvectors}
Fix a generic self-dual representation $\pi_v$ of $\GL_N$ with conductor~$c$ and newvector $\varphi_v \in \pi_v$. Then we have 
\[
(J_{v,c+k}^{-1} w_N \rtimes \theta) \Psi_{\pi_v}(h |\det|^{-1/2}) =  \tau_c(\varphi_v) \Psi_{\pi_v} (h(\varpi_v^k(*)^{-T}) |\det|^{-1/2} )
\]
for all $h \in \ms H_{N-1}$ and where we recall $\tau_c(\varphi_v) = \pm 1$ is the eigenvalue of $(J_{v,c}^{-1} w_N \rtimes \theta)$ acting on $\varphi_v$. 
\end{prop}
\begin{proof}
Computing and recalling that $\Psi_{\pi_v}(\varphi) = h * \varphi_v$:  
\begin{align*}
(J_{v,c+k}^{-1} w_N \rtimes \theta) \Psi_{\pi_v} (h |\det|^{-1/2}) &= \int_{\GL_{N-1,v}} h(g) (J_{v,k+c}^{-1} w_N \rtimes \theta) \begin{pmatrix} g^{-1} & \\ & 1 \end{pmatrix} \varphi_v  \, dg \\
&= \int_{\GL_{N-1,v}} h(g)  \begin{pmatrix} g^{T} & \\ & 1 \end{pmatrix} (J_{v,k+c}^{-1} w_N \rtimes \theta) \varphi_v \, dg \\
&= \int_{\GL_{N-1,v}} h(g)  \begin{pmatrix} g^{T} & \\ & 1 \end{pmatrix} J_{v,k}^{-1} (J_{v,c}^{-1} w_N \rtimes \theta) \varphi_v \, dg \\
&=  \tau_c(\varphi_v) \int_{\GL_{N-1,v}} h(g)  \begin{pmatrix} \varpi_v^{-k}g^{T} & \\ & 1 \end{pmatrix}  \varphi_v \, dg  \\
&= \tau_c(\varphi_v)\int_{\GL_{N-1,v}} h(\varpi_v^kg^{-T})  \begin{pmatrix} g^{-1} & \\ & 1 \end{pmatrix}  \varphi_v \, dg \\
&=  \tau_c(\varphi_v) \Psi_{\pi_v} (h(\varpi_v^k(*)^{-T}) |\det|^{-1/2} ),
\end{align*}
which finishes the argument. 
\end{proof}

\begin{prop}\label{twistedtraceoldvectors}
Fix a generic self-dual representation $\pi_v$ of $\GL_N$ with conductor~$\mf p_v^c$, and newvector $\varphi_v$. Then,
\begin{itemize}
\item If $N$ is even: 
\[
\tr_{\wtd \pi_v^{K_1(c+k)}}(J_{v,c+k}^{-1} w_N \rtimes \theta) =  \tau_c(\varphi_v)
\begin{dcases}
\binom{k/2 + N/2 - 1}{N/2 - 1} & k \text{ even} \\
0 & k \text{ odd.}
\end{dcases}
\]
\item If $N$ is odd:
\[
\tr_{\wtd \pi_v^{K_1(c+k)}}(J_{v,c+k}^{-1} w_N \rtimes \theta) =  \tau_c(\varphi_v)
\begin{dcases}
\binom{k/2 + (N-1)/2}{(N-1)/2} & k \text{ even} \\
\binom{(k-1)/2 + (N-1)/2}{(N-1)/2}  & k \text{ odd.}
\end{dcases}
\]
\end{itemize}
We recall $\tau_c(\varphi_v) = \pm 1$ is the eigenvalue of $(J_{v,c}^{-1} w_N \rtimes \theta)$ acting on $\varphi_v$. 
\end{prop}

\begin{proof}
By Proposition \ref{twistedactionoldvectors} and Theorem \ref{oldforms}, it suffices to compute the trace of 
\[
h |\det|^{-1/2} \mapsto h(\varpi_v^k(*)^{-T}) |\det|^{-1/2}
\]
on $\ms H_{N-1}^{\leq k}$. We use the basis of elements $T(a_1, \dotsc, a_{N-1})|\det|^{-1/2}$. 

Since $-\lb_i$ is Weyl-conjugate to $\lb_{N- 1 - i} - \lb_{N-1}$ for $1 \leq i \leq N-2$, we get that if $\sum a_i = d$,
\begin{equation}
 T(a_1, \dotsc, a_{N-1})(\varpi_v^k g^{-T}) = T(a_{N-2}, a_{N-3}, \cdots, a_1, k - d )(g) .  
\end{equation}
The trace is then exactly the number fixed points of this involution: indices $(a_i)_i$ such that
\[
(a_1, \dotsc, a_{N-1}) = (a_{N-2}, a_{N-3}, \cdots, a_1, k - d ).
\]
If $N-2$ is even, there are only fixed points when $k$ is even. In this case, the choices for~$a_1, \dotsc, a_{N/2-1}$ determine the fixed point and can be any values whose sum is~$\leq k/2$. This gives the claimed formula.

If~$N-2$ is odd, the choices for $a_1, \dotsc, a_{(N-1)/2}$ still determine the fixed point. If $k$ is even, then $a_{(N-1)/2}$ is even and the rest can be any values whose sum is~$\leq k - a_{(N-1)/2}$. If $k$ is odd, then the situation is the same except $a_{(N-1)/2}$ needs to be odd. Both case give the claimed formula. 
\end{proof}

\subsubsection{Twisted Action on Oldvectors: Conjugate Self-Dual Case}
Similarly:
\begin{lem}\label{twistedactionoldvectorsconj}
Fix a generic conjugate self-dual representation $\pi_v$ of $\GL_N(F_v \otimes_F E)$ with conductor $\mf p_v^c$ and newvector $\varphi_v$. Then we have 
\[
(J_{v,c+k}^{-1} w_N \rtimes \bar \theta) \Psi_{\pi_v}(h |\det|^{-1/2}) =  \bar \tau_c(\varphi_v) \Psi_{\pi_v} (h(\varpi_v^k\bar{(*)}^{-T}) |\det|^{-1/2} )
\]
for all $h \in \ms H_{N-1}$ and where we recall $\bar \tau_c(\varphi_v) = \pm 1$ is the eigenvalue of $(J_{v,c}^{-1} w_N \rtimes \bar \theta)$ acting on $\varphi_v$. 
\end{lem}

\begin{proof}
This is the same argument as Proposition \ref{twistedactionoldvectors} after noting that conjugation sends any $K_{N-1,v}$-double coset to itself. 
\end{proof}

\begin{prop}\label{twistedtraceoldvectorsconj}
Fix a generic conjugate self-dual representation $\pi_v$ of $\GL_N(F_v \otimes_F E)$ with 
(locally-indexed as in \S\ref{sec conductorsconj}) conductor $c(\pi_v) = k$ and newvector $\varphi_v$. Then,

\begin{itemize}
\item If $v$ is non-split and $N$ is even:
\[
\tr_{\wtd \pi_v^{K_1(c+k)}}(J_{v,c+k}^{-1} w_N \rtimes \bar \theta) = \bar \tau_c(\varphi_v)
\begin{dcases}
\binom{k/2 + N/2 - 1}{N/2 - 1} & k \text{ even} \\
0 & k \text{ odd}
\end{dcases}
\]
\item If $v$ non-split and $N$ is odd:
\[
\tr_{\wtd \pi_v^{K_1(c+k)}}(J_{v,c+k}^{-1} w_N \rtimes \bar \theta) =  \bar \tau_c(\varphi_v)
\begin{dcases}
\binom{k/2 + (N-1)/2}{(N-1)/2} & k \text{ even} \\
\binom{(k-1)/2 + (N-1)/2}{(N-1)/2}  & k \text{ odd}
\end{dcases}
\] 
\item If $v$ is split:
\[
\tr_{\wtd \pi_v^{K_1(c+k)}}(J_{v,c+k}^{-1} w_N \rtimes \bar \theta) = \bar \tau_c(\varphi_v)
\binom{k + N - 1}{N - 1}.
\]
\end{itemize}
We recall $\tau_c(\varphi_v) = \pm 1$ is the eigenvalue of $(J_{v,c}^{-1} w_N \rtimes \bar \theta)$ acting on $\varphi_v$. 
\end{prop}

\begin{proof}
For the non-split cases, $\GL_N(F_v \otimes_F E) = \GL_N(E_w)$ for some $w|v$. The basis elements $T(a_1, \dotsc, a_{N-1})$ defined via the analogous uniformizer $\varpi_w$ of $\mc O_w$ are fixed  under conjugation. Therefore, these two cases follow from Proposition \ref{twistedtraceoldvectors} applied to representations of $GL_N(E_w)$. 

For the split case, $\GL_N(F_v \otimes_F E) = \GL_N(F_v)^2$ so the relevant version of $\ms H_{N-1}^{\leq k}$ is the product of two copies of the version from \ref{twistedtraceoldvectors}, where conjugation switches the two factors. There is then exactly one fixed point of the relevant action from~\ref{twistedactionoldvectorsconj} for each choice of the first factor. 
\end{proof}

\subsection{\lm{$\wtd E^\infty_\mf n$}: Counts Weighted by Epsilon Factors}\label{Edef}
Next, we find a twisted trace that counts representations $\pi$ of level $\mf n$ weighted by their root numbers $\eps(1/2, \pi)$. 
\subsubsection{Local Test Function}
Working locally, Propositions \ref{twistedactionoldvectors} and \ref{twistedactionoldvectorsconj} let us solve an upper-triangular system of linear equations to produce the local factors of the test function giving weighted counts of automorphic forms:
\begin{cor}\label{localepstestfunction}
Consider  
$G_{N,v}$ in either the self-dual or conjugate self-dual case---i.e. $G_v = \GL_{N,v}$ or $\Res^E_F \GL_{N,v}$ for $v$ split, ramified or inert. Then there are coefficients $a_N(k,i)$ such that if we define
\[
\wtd E_{v,k} := \sum_{i=0}^k a_N(k,i)( J_{v,k-i}^{-1} w_N \rtimes \theta )  \bar \1_{K_{1,v}(k-i)},
\]
then for $\pi_v$ a generic representation of $G_v$ we have
\[
\tr_{\pi_v}(\wtd E_{k,v}) = \begin{cases}
\tau & c(\pi_v) = k \\
0 & \text{else},
\end{cases}
\]
where $\tau$ is the eigenvalue of $\iota_{v,k}$ or $\bar \iota_{v,k}$ acting on the newvector for $\pi_v$. 
\end{cor}

We can compute $a_N(k,i)$ exactly in many cases: 
\begin{prop}\label{Eexplicitformula}
The following formulas hold for $\wtd E_{v,k}$ from Corollary \ref{localepstestfunction}:: 
\begin{itemize}
\item self-dual case, $N$ even,
\[
\wtd E_{v,k} := \sum_{i=0}^{N/2} (-1)^i \binom {N/2}i ( J_{v,k-2i}^{-1} w_N \rtimes \theta ) \bar \1_{K_{1,v}(k-2i)},
\]
\item conjugate self-dual case, $N$ even, $v$ non-split:
\[
\wtd E_{v,k} := \sum_{i=0}^{N/2} (-1)^i \binom {N/2}i ( J_{v,k-2i}^{-1} w_N \rtimes \bar \theta ) \bar \1_{K_{1,v}(k-2i)},
\]
\item conjugate self-dual case, $v$ split
\[
\wtd E_{v,k} := \sum_{i=0}^N (-1)^i \binom Ni ( J_{v,k-i}^{-1} w_N \rtimes \bar \theta ) \bar \1_{K_{1,v}(k-i)}.
\]
\end{itemize}
\end{prop}

\begin{proof}
This follows from the 
formulas in Propositions \ref{twistedactionoldvectors} and \ref{twistedactionoldvectorsconj} and the identity 
\[
\sum_{i = 0}^{b+1} (-1)^i \binom {b+1}i \binom{a - i + b}{b} = \begin{cases}
    1 & a=0 \\ 0 & \text{else}
\end{cases},
\]
for $a \in \Z^+$. The combinatorial identity in turn follows from the simpler identity $ \sum_{i=0}^{b+1}(-1)^i\binom{b+1}{i}i^k = 0$ when $ k <b+1$, as first proved by Euler (see \cite{Go78}).

See also \cite{Bin17} for a similar version of this result for untwisted traces. 
\end{proof}

\subsubsection{Global Test Function}
Putting our local test functions globally:
\begin{dfn}
In the self-dual case, if $\mf n = \prod_v \mf p_v^{k_v}$ is an ideal of $F$, define the test function 
\[
\wtd E^\infty_{\mf n} = \prod_v \wtd E_{\mf p_v^{k_v}}
\]
on $\wtd G_N(\A^\infty)$ in terms of the $\wtd E_{\mf p_v^{k_v}}$ from Corollary \ref{localepstestfunction}. 
\end{dfn}
In the conjugate-self dual case, we need to twist by a central character at non-split places to deal with the corresponding discrepancy in \ref{epsformulaconj}:
\begin{dfn}\label{def globalE}
In the conjugate self-dual case, let $\mf n = \prod_v \mf p_v^{k_v}$ be a conductor as in \S\ref{sec conductorsconj}. Choose $a \in E^\times$ with $\bar a = -a$ (by Hilbert 90) and define the test function 
\[
\wtd E^\infty_{\mf n} = \prod_{v \text{ split}} \wtd E_{\mf p_v^{k_v}} \prod_{v \text{ ns.}} \delta_{a^{N-1}} \star \wtd E_{\mf p_v^{k_v}}
\]
on $\wtd G_N(\A^\infty)$ in terms of the $\wtd E_{\mf p_v^{k_v}}$ from Corollary \ref{localepstestfunction} (but now globally indexed). 
\end{dfn}

\begin{thm}\label{globalepsilontrace}
Consider the self-dual case $G_N = \GL_N/F$. Choose~$\mf n$ an ideal of~$\mc O_F$. Then for all generic self-dual automorphic representations~$\pi$ of $G_N$:
\[
\tr_{\wtd \pi^\infty}(\wtd E^\infty_\mf n) = \begin{cases}
\eps(1/2, \pi_\infty)^{-1} \eps(1/2, \pi) & c(\pi) = \mf n, \pi \text{ symplectic-type} \\
\eps(1/2, \pi_\infty, \psi)^{-N+1} & c(\pi) = \mf n, \pi \text{ orthogonal-type} \\
0 & c(\pi) \neq \mf n
\end{cases}
\]
(where the second case may depend on the choice of $\psi$ made to define $\wtd \pi$). 
\end{thm}

\begin{proof}
The $c(\pi) \neq \mf n$ case follows from multiplying together the same cases of Corollary \ref{localepstestfunction}. 

If $c(\pi) = \mf n$, then let $\varphi = \prod_{v \nmid \infty} \varphi_v$ be in the space of newvectors for $\pi^\infty$. Then by Proposition \ref{twistedtracefactor}, Corollary \ref{localepstestfunction}, and \eqref{twistedmultform},
\[
\tr_{\wtd \pi^\infty}(\wtd E^\infty_\mf n) = 
\prod_{v \nmid \infty} \tr_{\wtd \pi_v}(\wtd E_{\mf p_v^k}) = \prod_{v \nmid \infty} \tau_{\mf p_v^k}(\varphi_v).
\]

If $\pi$ symplectic-type, then $N$ is even and $\pi$ has trivial central character so Theorem~\ref{epsformula} reduces this to
\[
\prod_{v \nmid \infty}  \eps(1/2, \pi_v) =  \eps(1/2, \pi_\infty)^{-1}  \eps(1/2, \pi)
\]
which is independent of $\psi$. The result follows. 

In general, the trace is $\eps(1/2, \pi_\infty, \psi)^{-N+1}  \eps(1/2, \pi, \psi)^{N-1}$. However, when $\pi$ is orthogonal-type, \cite[Thm 1.5.3(b)]{Art13} gives that $\eps(1/2, \pi) = \eps(1/2, \pi \times \triv) = 1$. 
\end{proof}

For the conjugate self-dual case, we need to choose $\psi = \psi \circ \sigma$ so we can understand the action of $\wtd E^\infty_\mf n$ on newvectors using \ref{epsformulaconj}:

\begin{thm}\label{globalepsilontraceconj}
Consider the conjugate self-dual case $G_N = \Res^E_F\GL_N$ and define extensions of representations to $\wtd G_N$ using $\psi$ such that $\psi \circ \sigma = \psi$. Choose $\mf n$ a conjugate self-dual conductor (as in \ref{sec conductorsconj}). Then for all generic conjugate self-dual automorphic representations $\pi$ of $G_N$ with central character $\om$:
\[
\tr_{\wtd \pi^\infty}(\wtd E^\infty_\mf n) = \begin{cases}
 \om_\infty(a)^{-1} \eps(1/2, \pi_\infty)^{-1} \eps(1/2, \pi) & c(\pi) = \mf n, N \text{ even}, \; \pi \text{ conj.-symp.} \\
  \om_\infty(a)^{-1} \eps(1/2, \pi_\infty)^{-1} & c(\pi) = \mf n, N \text{ even}, \; \pi \text{ conj.-orth.} \\
  \om_\infty(a)^{-N+1}\eps(1/2, \pi_\infty, \psi)^{-N+1}  & c(\pi) = \mf n, N \text{ odd},\; \pi \text{ conj.-symp.} \\
  1 & c(\pi) = \mf n, N \text{ odd}, \; \pi \text{ conj.-orth.} \\
0 & c(\pi) \neq \mf n
\end{cases}
\]
where $a$ appears in the definition of $\wtd E^\infty_\mf n$. Note that the $\eps$-terms without $\psi$ are standard root numbers with respect to some local $\psi_0 \circ \sigma = \psi_0^{-1}$ as in \ref{localepspm1conj}.
\end{thm}

\begin{proof}
Similar to Theorem \ref{globalepsilontrace}, the trace is $0$ unless $c(\pi) = \mf n$ in which case
\[
\tr_{\wtd \pi^\infty}(\wtd E^\infty_\mf n) = \prod_{v \nmid \infty} \bar \tau_{v,\mf p_v^k}(\varphi_v).
\]
Recall the choice of $a \in E^\times$ such that $\bar a = -a$ to define $\wtd E^\infty_\mf n$. We use \ref{epsformulaconj} with this choice of $a$ at all places: 

In the first and second cases, the central character is conjugate-orthogonal and therefore valued in $\pm 1$ so $\om^{N-1} = \om$. Therefore, the trace is
\[
\om_{\spl}(a)^{N-1} \prod_{v \nmid \infty} \eps(1/2, \pi_v) \prod_{\substack{v \nmid \infty \\ v \text{ n.s.}}} \om_v(a) = \om_\infty(a)^{-1}\eps(1/2, \pi_\infty)^{-1} \eps(1/2, \pi),
\]
using the note after~\ref{epsformulaconj} at split places. For the second case, we additionally use Theorem \ref{globalorthepsconj} to compute that~$\eps(1/2, \pi) = 1$ globally. 

The third case is similar to the first: the trace is equal to 
\[
\eps(1/2, \pi_\infty, \psi)^{-N+1} \om_{\infty}(a)^{-N+1} \eps(1/2, \pi)^{N-1}.
\]
Since the global root number is independent of $\psi$, multiplying together \ref{localepspm1conj} gives that it is $\pm 1$ even for our non-standard $\psi$. Therefore, the third factor vanishes. 

The fourth case follows immediately from the corresponding case of \ref{epsformulaconj}. 
\end{proof}

\begin{note}
If $\pi$ is unramified at a split place $v$, then it can be shown that $\om_v(a) = 1$. Therefore, if $\pi$ is unramified at all split places, $\om_\ns(a) = \om_\infty(a)^{-1}$.
\end{note}

Beware again that for $N$ odd, the test function $\wtd E^\infty_\mf n$ fails to produce any information about root numbers.

\subsection{\lm{$\wtd C^\infty_\mf n$}: Unweighted Counts through a Twisted Trace}\label{sec Cdef}

We now use the twisted trace-Paley-Weiner theorem of \cite{Rog88} to modify $\wtd E_{v,k}$ into a function $\wtd C_{v,k}$ on $\wtd G_{N,v}$ counting representations of conductor $\mf p_v^k$ unweighted by their root number. 

As the Paley-Weiner input:
\begin{prop}\label{PWinput}
Let $\wtd G_v = G_v \rtimes \theta$ be a twisted group and $\wtd f_v$ a test function on $\wtd G_v$. Choose a constant $a_\Theta$ for each Bernstein component $\Theta$ of $G_v$. Then there exists a test function $\wtd f_v'$ on $\wtd G_v$ such that if $\wtd \pi_v$ is an extension to $\wtd G_v$ of the self-dual irreducible representation $\pi_v \in \Theta$, we have
\[
\tr_{\wtd \pi_v} \wtd f_v' = a_\Theta \tr_{\wtd \pi_v} \wtd f_v.
\]
\end{prop}

\begin{proof}
This is a special case of the main theorem of \cite{Rog88} taking in to account the characterization of twisted Bernstein components in \S6 therein. Note that self-dual is the same as $\theta$-stable. 
\end{proof}

In summary, a local twisted test function 
can be modified by rescaling it by a different constant $a_\Theta$ on each Bernstein component. We next see that root numbers are constant on representations of a given conductor belonging to a specific $\Theta$.

\begin{prop}\label{epsconstant}
In the self-dual case, let $\Theta$ be a Bernstein component of $\GL_{N,v}$. Choose $k \geq 0$. Then there is a constant $C_{\Theta, k}$ such that for all self-dual irreducible representations $\pi_v \in \Theta$ with conductor $c(\pi_v) = k$, we have $\eps(1/2, \pi_v) = C_{\Theta, k}$.
\end{prop}

\begin{proof}
Fix a product of supercuspidals  $\boxtimes_{i \in I} \rho_i$ on $\prod_{i \in I} \GL_{n_i}$ with $\sum n_i = N$ that determines a Bernstein component $\Theta$. Possibly twisting by a power of $|\det|$ on each factor, which leaves $\Theta$ unchanged, we assume that for all pairs $(i,j)$, there is no~$s \neq 0$ such that $\rho_i|\det|^s = \rho_j$. 

Now consider a self-dual $\pi_v \in \Theta$, i.e. $\pi_v$ is a subquotient of the parabolic induction of  $\boxtimes_{i \in I} \rho_i |\det|^{s_i}$ for some $s_i \in \C$. 

By the Bernstein-Zelevinsky classification \cite{BZ77}, there is a partition of the sequence $(\rho_i|\det|^{s_i})_i$ into segments $\delta_j$ of the form 
\[
\delta_j = (\sigma_j |\det|^{r_j + \f12(t_j -1)}, \sigma_j |\det|^{r_j + \f12(t_j -3)}, \dotsc, \sigma_j |\det|^{r_j - \f12(t_j - 1)})
\]
such that if $\Delta_j$ is the discrete representation of $GL_{n_jt_j}$ associated to the segment $\delta_j$, then $\pi_v$ is the Langlands quotient of $\boxtimes_j \Delta_j$. Our condition on the $\rho_i$ lets us assume, up to making the correct choice of $r_j$, that the multiset of $t_j$ copies of each $\sigma_j$ is the same as the multiset of $\rho_i$'s.

By \cite[p. 153, (2)]{Hen86}, the root number of $\pi_v$ satisfies 
\begin{equation}\label{epsfactor}
    \eps(s,\pi_v) = \prod_j \eps(s,\Delta_j).
\end{equation}
Furthermore, if $\sigma_j$ is ramified, then recalling that ramified supercuspidals on $\GL_N$ have trivial $L$-factor, we have by \cite[p. 153, (5)]{Hen86}:
\begin{multline}\label{epsramspeh}
\eps(s, \Delta_j) = \prod_{k=1}^{t_j} \eps(s, \sigma_j |\det|^{r_j + \f12(t_j + 1) - k}) \\
= \prod_{k=1}^{t_j} \eps(s, \sigma_j) q_v^{c_{\sigma_j}(r_j + \f12(t_j + 1) - k)} =\eps(s, \sigma_j)^{t_j} q_v^{c_{\sigma_j}r_jt_j} .
\end{multline} 
If $\sigma_j$ is unramified, then it is necessarily of the form $|\det|^s$ on $\GL_1$ and $\Delta_j$ is of the form $\St(t_j)|\det|^{r_j}$ where $\St(t_j)$ is the Steinberg representation on $\GL_{t_j}$. Since $\St(t_j)$ has conductor $t_j-1$ and root number $(-1)^{t_j - 1}$ by, e.g, \cite[(4.1.4)-(4.1.6)]{Ta79}, we get 
\begin{equation}\label{epsstein}
    \eps(s, \Delta_j) = \eps(s,\St(t_j)) q_v^{(t_j-1)r_j} = (-1)^{t_j - 1}  q_v^{(t_j - 1)(1/2-s)} q_v^{(t_j - 1)r_j}.
\end{equation}

Putting the formulas \eqref{epsfactor}---\eqref{epsstein} together, let $J_0 \subseteq J$ be the set of indices $j$ such that $\Delta_j$ is a Steinberg (also define $I_0 \subseteq I$ correspondingly). Then we can compute the conductor
\begin{equation}\label{c1}
c(\pi_v) = \sum_{j \in J_0}(t_j - 1) + \sum_{j \in J \setminus J_0} t_j c_{\sigma_j} = u - u_0 + \sum_{i \in I} c_{\rho_i}
\end{equation}
for $u$ the total number of unramified $\rho_i$ (equivalently, of unramified $\sigma_j$) and $u_0$ the total number of $\Delta_j$ that are Steinberg so that $u-u_0=\sum_{j\in J_0}(t_j-1)$.

Also, since $\pi_v$ is self-dual, there is a decomposition of the indexing set 
\[ 
J = J_d \sqcup J_n
\] 
with $J_d$ consisting of indices such that $\sigma_j$ is self-dual and $r_j=0$, and $J_n$ consisting of pairs of indices $j,j'$ such that $r_j = -r_{j'}$ and $\sigma_j = \sigma^\vee_{j'}$ (in particular $c_{\sigma_j} = c_{\sigma^\vee_{j'}})$. 
Therefore, when multiplying \eqref{epsramspeh} and \eqref{epsstein} over all $j$, all the powers of $q_v^{r_j}$ cancel out, leaving:
\begin{multline*}
\eps(1/2, \pi_v) = \prod_{j \in J_0} (-1)^{t_j - 1} \prod_{j \in J \setminus J_0} \eps(1/2, \sigma_j)^{t_j} = (-1)^{u - u_0} \prod_{i \in I \setminus I_0} \eps(1/2, \rho_i) = \\
= (-1)^{u - u_0} \prod_{i \in I} \eps(1/2, \rho_i).
\end{multline*}
Finally, by \ref{c1}, this becomes
\[
(-1)^{c_\pi - \sum_{i \in I} c_{\rho_i}} \prod_{i \in I} \eps(1/2, \rho_i),
\]
which is our $C_{\Theta, c(\pi_v)}$ that only depends on the $\rho_i$ and $c(\pi_v)$. 
\end{proof}

\begin{prop}\label{epsconstantconj}
Fix a character $\psi$ to define local $\eps$-factors. In the conjugate self-dual case, let $\Theta$ be a Bernstein component of $\GL_{N,v}$ for $v$ non-split. Choose $k \geq 0$. Then there is constant $C_{\Theta, k}$ such that for conjugate self-dual irreducible representations $\pi_v \in \Theta$ with conductor $c(\pi_v) = k$, we have $\eps(1/2, \pi_v, \psi) = C_{\Theta, k}$.
\end{prop}

\begin{proof}
This is the same argument as Proposition \ref{epsconstant} since invariance of the $v$-adic norm under conjugation means that conjugate-self duality forces the same pairing condition on the $r_j$. 
\end{proof}

Now we are ready to define the local components of our test function that gives an unweighted count of representations:

\begin{cor}\label{localC}
In either the self-dual or conjugate self dual case, there is a test function $\wtd C_{v,k}$ on $\wtd G_{N,v}$ such that for all generic (conjugate) self-dual $\pi_v$
\[
\tr_{\wtd \pi_v} \wtd C_{v,k} = \begin{cases} 1 & c(\pi_v) = k \\
0 & \text{else}
\end{cases}. 
\]
\end{cor}

\begin{proof}
For all $\pi_v$ in Bernstein components~$\Theta$ such that $\tr_{\wtd \pi_v} \wtd E_{v,k} \neq 0$, we have that~$C_{\Theta,k}^{-(N-1)}$ from Proposition \ref{epsconstant} (resp. \ref{epsconstantconj}) is the same as the~$\tau$ from Corollary~\ref{localepstestfunction} by Proposition~\ref{epsformula} (resp. \ref{epsformulaconj}). Therefore, we just apply Proposition~\ref{PWinput} to~$\wtd E_{k,v}$ with constants~$a_\Theta = C_{\Theta,k}^{-(N-1)}$.
\end{proof}

Putting these together globally:

\begin{dfn}\label{Cdef}
If $\mf n = \prod_v \mf p_v^{k_v}$ is an ideal of $F$, define the test function on $\wtd G_N(\A^\infty)$:
\[
\wtd C^\infty_{\mf n} = \prod_v \wtd C_{\mf p_v^{k_v}}
\]
in terms of the $\wtd C_{\mf p_v^{k_v}}$ from Corollary \ref{localC} (but now globally indexed) in either the self-dual or conjugate self dual case. 
\end{dfn}

This of course satisfies:

\begin{cor}\label{globalC}
In either the conjugate self-dual or self-dual case, for all generic self-dual representations $\pi^\infty$ of $(G_N)^\infty$,
\[
\tr_{\wtd \pi^\infty} \wtd C^\infty_\mf n = \begin{cases} 1 & c(\pi^\infty) = \mf n \\
0 & \text{else}.
\end{cases}
\]
\end{cor}

\begin{note}
Unlike the case of $\wtd E^\infty_\mf n$ (specifically against orthogonal or conjugate-orthogonal representations), the trace identity characterizing $\wtd C^\infty_\mf n$ applies to more general irreducible representations of $(G_N)^\infty$ than just the ones that are factors of automorphic representations. This will be important for our spectral argument to understand its transfers in Section \ref{sec stable plancherel}. 
\end{note}

\section{Asymptotics for the Cuspidal Spectrum: Setup}\label{sec tfdecomp}
To compute our counts, we need to isolate certain parts of the automorphic spectrum on $G_N$ through the trace formula:
\begin{itemize}
    \item In the self-dual or orthogonal/symplectic case, the cuspidal, self-dual representations of either orthogonal or symplectic-type with a given  infinitesimal character $\lb$ at infinity.
    \item In the conjugate self-dual or unitary case, the cuspidal, conjugate self-dual representations with a given infinitesimal character $\lb$ at infinity. 
\end{itemize}
We do this using the endoscopic classification through the formalism of refined shapes of parameters from \cite[\S4,5]{DGG22}. While the characterization of the discrete spectrum of $\GL_N$ in \cite{MW89} might appear to be enough on the spectral side, giving a tractable geometric-side expression requires the full methods of \cite{DGG22}. 

We freely use notation analogous to \cite{DGG22} throughout. In the case of $G = \SO_{2N}^\eta$, we will suppress the bars from \cite{Art13} that represent orbits under the outer automorphism and take the unbarred notations to represent these orbits. The~$G = \SO_{2N}^\eta$ case is not a main focus of this paper and this change is irrelevant for other $G$. 

Recall also the list of simple endoscopic groups of $\wtd G_N$ from \S\ref{sec twistedendoscopicgroups} and assumption~\ref{assump EF} on our number fields. 

\begin{note}
A comment on the use of the endoscopic classification and the induction of \cite{Tai17}: First, we use the it to separate globally symplectic-type and orthogonal-type automorphic representations as this distinction cannot even be defined otherwise. An alternate method is using a working sufficient condition that if $\pi$ has local component $\pi_v$ with only a single symmetry type, then $\pi$ has that same symmetry type. Through this, we can pick out some representations with fixed symmetry type using just the twisted trace formula as in \cite{TW24}. 

However, while the purely twisted-trace-formula strategy is conceptually simpler, the lack of a twisted analogue of the geometric expansion of \cite{Art89} makes it much more analytically and computationally involved. In our case, we trade-off the abstract complexity of the inductive analysis of \cite{Tai17} to make the explicit computations at the bottom tractable, particularly for our very complicated test functions. This also allows inputting the much stronger error-term bounds of \cite{ST16}.  
\end{note}

\subsection{Infinitesimal characters}\label{sec infchars}
Recall that the infinitesimal character $\lambda$ is an invariant of irreducible representations of a reductive Lie group $G(\R)$. We have $\lb \in \Om_G \bs \Hom(\mf t, \C)$, where~$\mf t$ is a Cartan subalgebra of $G(\R)$. This data is the same as a map from Weyl orbits of $X^*(\wh{\mf t}) = X_*(\mf t) \to \C$, which is further the same as a semisimple conjugacy class in~$\wh{\mf g}$.

\subsubsection{Infinitesimal characters in the classification}
We need a uniform way to talk about infinitesimal characters at infinity for all $G \in \wtd{\mc E}_\sm(N)$. 

In the self-dual case, the embedding $\wh G \into \GL_N(\C)$ induces a map from semisimple conjugacy classes in $\wh{\mf g}_\infty$ to those in $M_N(F_\infty \otimes_\R \C)$. This map is injective except when $G=\SO^\eta_{2n}$, where its fibers are orbits under the outer automorphism. We will elide this technicality since everything in the $\SO_{2n}^\eta$ case is up to outer automorphism and therefore consider an infinitesimal character of $G$ as a semisimple matrix up to conjugacy: in other words an unordered sequence
\[
\lb = (\lb_1, \lb_2, \dotsc,  \lb_N)
\]
with $\lb_i \in F_\infty \otimes_\R \C$. 

In the conjugate self-dual case, since the map $\wh G \into \GL_N(\C) \times \GL_N(\C)$ is a restriction from $\Ld G \into \Ld G_N$, it is determined by its first coordinate. Therefore, the infinitesimal character can also be represented by a semisimple conjugacy class in $M_N(F_\infty \otimes_\R \C)$ and therefore a similar unordered sequence.

It is also useful to package the tuple $\lb_v = (\lb_{1,v},...,\lb_{n,v})$  as the generating function $\sum_j X^{\lb_{j,v}}$, which, by abuse of notation, we will also denote $\lb_v$. In this way, if $\bigoplus_i \tau_{i,v}[d_i]$ is a local Arthur parameter such that each $\tau_{i,v}$ has infinitesimal character $\lb_{v}^{(i)} = \sum_{j=1}^{n_i}X^{\lb_{j,v}^{(i)}}$, then we have the infinitesimal character assignment
\begin{equation}\label{infcharform}
\lf( \bigoplus_i \tau_{i,v}[d_i] \ri)_\infty \mapsto \sum_i \lb_{v}^{(i)} \sum_{l=1}^{d_i} X^{\f{d_i+1}2 - l}.
\end{equation}
It can be seen from this that the character of $\tau[d]$ determines that of $\tau$. 

\subsubsection{Integrality}\label{integralinfchars}
The infinitesimal characters of finite-dimensional representations of the $G_v$ for $v|\infty$ will play a special role. We express their classification in terms of generating functions.

For the self-dual case:
\begin{itemize}
\item
If $G = \Sp_{2n}$, then $\wh G = \SO_{2n+1}$ and the $\lb_v$ are elements of $1 + \Z[X, X^{-1}]$ invariant under $X \mapsto X^{-1}$ and with each coefficient $0$ or $1$.  
\item
If $G = \SO_{2n+1}$, then $\wh G = \Sp_{2n}$  and the $\lb_v$ are elements of $X^{1/2} \Z[X, X^{-1}]$ invariant under $X \mapsto X^{-1}$ and with each coefficient $0$ or $1$.
\item
If $G = \SO_{2n}$, then $\wh G = \SO_{2n}$ and the $\lb_v$ are elements $\Z[X, X^{-1}]$ invariant under $X \mapsto X^{-1}$ and with each coefficient $0$ or $1$ except the constant coefficient which is $0$ or $2$. 
\end{itemize}

For the conjugate self-dual case:
\begin{itemize}
    \item If $N$ is odd, then the $\lb_v$ are elements of $X^{1/2} \Z[X, X^{-1}]$ with each coefficient $0$ or $1$.
    \item If $N$ is even, then the $\lb_v$ are elements of $\Z[X, X^{-1}]$ with each coefficient $0$ or $1$. 
\end{itemize}

\begin{dfn} \label{def integral}
Call $\lb_v$ \emph{integral} if it is one of the cases above.
\end{dfn}

Note that this is the same as the $C$-algebraic condition of \cite{BG14}. 

\subsection{Central Characters}\label{sec cchartransfer}
We collect some useful lemmas on central characters:

\subsubsection{Central Characters and Shapes}
Every parameter $\psi$ has a central character~$\eta_\psi$, of order $\leq 2$ in the self-dual case. Through class field theory, we will abuse notation and use $\eta_\psi$ to denote the corresponding character of $\Gamma_F$. Note that
\begin{equation}\label{ccharform}
\eta_{\bigoplus_i \tau_i[d_i]} = \prod_i \eta_{\tau_i}^{d_i}.
\end{equation}

\subsubsection{Central Characters and Transfer}
We need the central character results described in \cite[\S10]{GGP11}. We use throughout as in \cite[\S5.1]{KS99} that any $G \in \wtd{\mc E}_\el(N)$ comes with an injection $(Z_{G_N})_\theta \into Z_G$ from the algebraic group of coinvariants. 

First, as is known to experts:
\begin{thm}\label{thm cchareps}
In the self-dual case, let $G \in \wtd{\mc E}_\sm(N)$ be either $\Sp^\eta_{2n}$ or $\SO_{2n}^\eta$ and~$\varphi_v$ be an $L$-parameter for $G_v$. Then the central character $\om_v$ of $\pi_v \in  \Pi_{\varphi_v}$ is determined by $\om_v(-1) = \eps(1/2, \varphi_v)/\eps(1/2, \det \varphi_v)$. 
\end{thm}

\begin{proof}
 This is claimed as an expected property in \cite[\S10]{GGP11}. A proof follows from the main result of \cite{LR05} together with a consistency statement between Langlands-Shahidi epsilon factors and those from Arthur's classification (see e.g. \cite[B.2]{AHO23} for a discussion). 
\end{proof}

\begin{lem}\label{packettoom}
Let $G \in \wtd{\mc E}_\sm(N)$ and $\varphi_v$ an $L$-parameter for $G_v$. Then the central character $\om_v$ of $\pi_v \in \Pi_{\varphi_v}$ is constant on $\Pi_{\varphi_v}$ (i.e. determined by $\varphi_v$). 
\end{lem}

\begin{proof}
As explained in \cite[Rmk 3.2.3]{Mok15}, the endoscopic character identities \cite[Thm 2.2.1(a)]{Art13} and \cite[Thm 3.2.1(a)]{Mok15} together with the transformation property for transfer factors given by the generalization of 5.1.C 
in \cite[p.53]{KS99} can be used to show this for the central character restricted to the image of the injection $(Z_{G_{N,v}})_\theta \into Z_{G_v}$. In the conjugate self-dual case, this is all of $Z_{G_v}$. 

In the self-dual case, $G_v$ is abelian if $N =1,2$. 
When $N > 2$, $Z_{G_v}$ is trivial or~$\pm 1$. It is trivial when $\wh G \cong \Sp_{2n}$. Otherwise, $\om(-1) = \eps(1/2, \varphi_v)/\eps(1/2, \det \varphi_v)$ by Theorem \ref{thm cchareps}. 
\end{proof}

We emphasize another consequence of the transfer factor transformation property after \cite[5.1.C]{KS99}:

\begin{lem}\label{cchartransfer}
Let $G \in \wtd{\mc E}_\sm(N)$ and $\varphi_v$ an $L$-parameter for $G_v$. Let $\om_v$ be the common central character of $\pi_v \in \Pi_{\varphi_v}$ as in Lemma \ref{packettoom}. 

Then, $\om_v$ restricted to the image of $(Z_{G_{N,v}})_\theta \into Z_{G_v}$ uniquely determines the (conjugate) self-dual central character $\om'_v$ of the representation $\td \pi_{\varphi_v}$ of $\wtd G_{N,v}$ corresponding to $\varphi_v$.
\end{lem}

With notation as in the statement of Lemma \ref{cchartransfer}, we will call this $\om'$ the \emph{transfer} of $\om$. Note that because of the possible non-triviality of $\lb_C$ in the generalization after \cite[5.1.C]{KS99}, the transfer $\om'$ is not necessarily pullback to $Z_{G,v}$. In fact, in the self-dual case, the image of $(Z_{G_{N,v}})_\theta \into Z_{G_v}$ is trivial, so $\om'_v$ is actually completely determined by this $\lb_C$ and therefore $G$:
\begin{itemize}
    \item If $N$ is odd, $\om'_v = \eta$ for $G = \Sp_{N-1}^\eta$. 
    \item If $N$ is even, $\om'_v = \eta$ for $G = \SO_N^\eta$ and $\om'_v = 1$ for $G = \SO_{N+1}$.
\end{itemize}

In the conjugate self-dual case, it can be checked by working over the algebraic closure that $(Z_{G_{N}})_\theta \into Z_{G}$ is given on points\footnote{we ignore the difference between this map and its inverse which depends on the specific choice of transfer factors. In our eventual application, the distinction will be moot due to arbitrary choices of infinite-place isomorphisms $\overline E_v \iso \C$ needed to parameterize infinitesimal characters.} by $\beta \mapsto \beta/ \bar \beta$. By Hilbert 90, this is a bijection over local fields, so the map $\om_v \mapsto \om'_v$ is an injection.
\begin{itemize}
    \item If $N$ is odd, the determinant of a conjugate symplectic/orthogonal parameter is the same type. Therefore, the image of $\om_v \mapsto \om'_v$ is all characters coming from conjugate symplectic or orthogonal
    parameters respectively depending on the types of parameters of $G$. 
    
    \item If $N$ is even, the determinant of a conjugate self-dual parameter is always conjugate-orthogonal. Therefore, the image of $\om_v \mapsto \om'_v$ is all characters coming from conjugate orthogonal parameters and doesn't depend on $G$.
\end{itemize}

In the case of $U^+_N$ we can be even more precise:

\begin{lem}\label{U+cchar}
Consider $G = U_N^+ \in \wtd{\mc E}_\sm(N)$. Then the transfer $\om_v \mapsto \om'_v$ is pullback through $Z_{G_{N,v}} \onto (Z_{G_{N,v}})_\theta \iso Z_{G_v}$. 
\end{lem}

\begin{proof}
This follows by the endoscopic transfer identity \cite[Thm 3.2.1(a)]{Mok15}  and \cite[Thm 2.5.1(c)]{Mok15}: if $\pi_v$ has parameter $\varphi_v$, then $\om_v$ has parameter $\det \varphi_v$ which necessarily lands in $\Ld Z_{G_v}$. 
\end{proof}

\subsection{Shapes for the Self-Dual Case}

We need a version of the theory of refined shapes for the 
self-dual case. Informally, a shape is a coarsening of the notion of an Arthur parameter, and corresponds to the part of discrete spectrum which we will 
isolate in our arguments. We use notation analogous to \cite[\S2]{DGG22} and refer readers to \cite{Art13} for more background.  
The key difference with \cite{DGG22} is that the auxiliary data $\eta$ will now be a central character instead of a parity. 

\subsubsection{Cuspidal Shapes}

One of the central results of the endoscopic classification, \cite[Theorem 1.4.1]{Art13}, is that each self-dual cuspidal parameter $\tau \in \Psi(N)$ factors through a unique simple endoscopic datum $H$. We first show that for the parameters with which we are concerned, this group $H$ is determined by the central and infinitesimal characters of $\tau$.

\begin{prop}\label{simpleshapesselfdual}
Let $\tau \in \Psi(N)$ be self-dual, cuspidal parameter with integral infinitesimal character $\lb$ at infinity and central character $\eta$. Then the endoscopic datum $G \in \wtd{\mc E}_\sm(N)$ through which $\tau$ factors is entirely determined by $\lb$ and $\eta$. 

Furthermore:
\begin{itemize}
\item
$\lb$ is the infinitesimal character of a finite-dimensional representation of $H$.
\item
If $G =\SO_{2n+1}$, then $\eta$ is trivial.
\item
If $G$ is a form of $\SO_{2n}$, then $G = \SO_{2n}^\eta$. In this case, for each $v|\infty$, $\eta$ is trivial on $\Gal_{F_v}$ if and only if $n$ is even. In particular, $G$ necessarily has discrete series.
\end{itemize}
\end{prop}

\begin{proof}
 This is the 
 argument of \cite[\S3.11]{CR15}, but we allow 
 for non-trivial central character. 

\noindent \underline{Using $\lb$}

We will see that $\lb$ determines $\wh G$. First, $\lb$ is integral so by the version of Clozel's purity lemma stated in \cite[Lemma 3.13]{CR15}, for each $v|\infty$, the $L$-parameter $\varphi_v : W_\R \to \GL_n \C$ associated to~$\tau_v$ is a direct sum of copies of the trivial representation~$\1$, the character $\eps_{\C/\R} : W_\R \onto \R^\times \to \pm 1$, and, for~$w \in \Z$, the representations
\[
I_w = \Ind_{\C^\times}^{W_\R} z^{-w/2} \bar z^{w/2}.
\]
Note that $I_w = I_{-w}$ and that $I_0 = \1 \oplus \eps_{\C/\R}$. 

Then $\lb_v$ is a list of numbers: a $0$ for each $\1$ or $\eps_{\C/\R}$ and a pair $(w/2, -w/2)$ for each $I_w$. In addition $I_w$ can be conjugated into an orthogonal/symplectic group if and only if $w$ is even/odd. Since $\tau$ comes from a simple endoscopic group that has either orthogonal or symplectic dual, it follows that all the $w$ have the same parity. 

In total, the non-zero numbers in $\lb_v$ are either be all integers or all half-integers for each $\lb_v$, depending on whether $\tau_v$ has image in an orthogonal or symplectic group, thereby determining
~$\wh G$. In addition, 
it follows from our definition of integral that~$\lb_v$ is the infinitesimal character of a finite-dimensional representation on $G$. 

\noindent \underline{Using $\eta$}

Next,~$\eta$ determines the remaining part of the datum corresponding to $G$. By the discussion on \cite[p. 10]{Art13}, we have three cases:
\begin{itemize}
\item
If $\wh G \cong \Sp_{2n}$, then $G = \SO_{2n+1}$ and $\eta$ is necessarily trivial. 
\item
If $\wh G \cong \SO_{2n + 1}$, then $G = \Sp^\eta_{2n} \cong \Sp_{2n}$. 
\item
If $\wh G \cong \SO_{2n}$, then $G = \SO^\eta_{2n}$. 
\end{itemize}
This completely determines $G$. We can also compute $\eta|_{\Gal_{F_v}}$ by computing $\det \tau_v$. In the $\wh G \cong \SO_{2n}$ case, $\tau_v$ is a sum of $n$ different summands $I_{w_i}$ for $w_i$ even and therefore acts trivially on $\Gal_{F_v}$ if and only if $n$ is even. Therefore, $G$ necessarily has discrete series under our Assumption \ref{assump EF}. 
\end{proof}

\subsubsection{Refined Shapes in the Self-Dual Case}
Motivated by the above, we define:
\begin{dfn}\label{def refinedshape}
An \emph{refined shape} is a sequence
\[
\Delta = (T_i, d_i, \lb_i, \eta_i)_i
\]
up to permutation and where $(T_i, d_i)$ are positive integers, $\lb_i$ is an infinitesimal character of rank $T_i$, and $\eta_i$ is a central character. 

We say that $\psi \in \Delta$, or that $\psi$ has refined global shape $\Delta$ if $\psi = \bigoplus_i \tau_i[d_i]$ with each $\tau_i$ of rank $T_i$ and such that each $\tau_i$ has central character $\eta_i$ and infinitesimal character $\lb_i$ at infinity. 

Let $\Psi_\el^\Delta$ be the set of all elliptic, self-dual parameters on $\GL_n$ of refined shape $\Delta$. 
\end{dfn}

Beware that $\Psi_\el^\Delta$ is empty unless all the bulleted conditions on $\lb$ and $\eta$ of Proposition \ref{simpleshapesselfdual} hold. 

\begin{dfn}
The refined shape $\Delta$ is \emph{integral} if the total infinitesimal character of $\Delta$ as in \eqref{infcharform} is integral, in the sense of Definition \ref{def integral}.
\end{dfn}

We also pick out a special refined shape:

\begin{dfn}\label{def simpleshape}
Let $\lb$ be an integral infinitesimal character for a group of rank $N$ and $\eta$ a central character that is trivial if $\lb$ has half-integral entries. Then $\Sigma_{ \eta, \lb}$ is the trivial refined shape $(N, 1, \lb, \eta)$.
\end{dfn}

\subsubsection{Properties}\label{shapepropertiesselfdual}
\begin{prop}\label{shapetogroup}
Let $\Delta$ be an integral refined shape of rank~$N$ such that all the~$\lb_i$ are integral (in the sense of~\ref{def integral}) and $\eta_i$ is trivial whenever $\lb_i$ has half-integer entries. Then there is a group $G(\Delta) \in \wtd{\mc E}_\el(N)$ such that all $\psi \in \Delta$ are parameters for $G$. Furthermore, $G(\Delta)$ has discrete series at infinity.
\end{prop}

\begin{proof}
Consider $\psi = \bigoplus_i \tau_i[d_i] \in \Psi_\el^\Delta$. By the discussion on \cite[pp. 33-34]{Art13}, it suffices to show that $\Delta$ determines two things: 
\begin{itemize}
\item
whether each $\tau_i[d_i]$ has image in an orthogonal or symplectic group. Let the total ranks of the orthogonal
and symplectic
pieces be $N_O$ and $N_S$ respectively,
\item
the central character $\eta$ of the sum of the orthogonal
$\tau_i[d_i]$. 
\end{itemize}
Then we will have
\[
G(\Delta) =
\begin{cases}
 \SO_{N_S + 1} \times \Sp^\eta_{N_O - 1} & N_O \text{ odd} \\
 \SO_{N_S + 1} \times \SO^\eta_{N_O } & N_O \text{ even}
\end{cases}.
\]

To show these two facts, apply Proposition \ref{simpleshapesselfdual} and note that if $d_i$ is odd, whether $\tau_i[d_i]$ has orthogonal or symplectic image is the opposite of that for~$\tau_i$. If~$d_i$ is even, then it is the same. In addition, the central character of the sum of the orthogonal-image~$\tau_i[d_i]$ is determined through \eqref{ccharform} by the central characters of the $\tau_i$ and which $\tau_i[d_i]$ have orthogonal image. 
\end{proof}
This gives us two corollaries:
\begin{cor}\label{grouptoshapes}
Let $G \in \wtd{\mc E}_\el(N)$. Then
\[
\Psi_\disc(G) = \bigsqcup_{\Delta : H(\Delta) = G} \Psi_\el^\Delta.
\]
\end{cor}

\begin{proof}
This follows from the above and the same discussion in \cite[pp. 33-34]{Art13}. 
\end{proof}

\begin{cor}\label{sigmatogroup}
Let $\lb$ be an integral infinitesimal character of rank $N$. Then
\[
G(\Sigma_{\eta,\lb}) = \begin{cases}
\Sp^\eta_{2n} &  N \text{ odd} \\
\SO_{2n+1} & N \text{ even, } \lb \text{ has half-integer entries}\\
\SO^\eta_{2n} & N \text{ even, } \lb \text{ has integer entries.}
\end{cases}
\]
In particular, for every $G \in \wtd{\mc E}_\sm(N)$ and infinitesimal character $\lb$ of a finite-dimensional representation of $G_\infty$, there is $\eta$ such that $H(\Sigma_{ \eta,\lb}) = G$ (beware that $\Psi^{\Sigma_{\eta, \lb}}_\el$ is empty when the conditions from Proposition \ref{simpleshapesselfdual} don't hold). 
\end{cor}

We collect some more properties of refined shapes $\Delta$:
\begin{itemize}
\item
By \cite[(1.4.9)]{Art13}, All $\psi \in \Delta$ correspond to the same pair $(\mc S_\psi,s_\psi)$ up to conjugation. Call the common values $S_{\Delta}$ and $s_{\Delta}$. 
\item
All $\psi \in \Delta$ have the same Arthur-$\SL_2$.
\item
The infinitesimal character at infinity and central character of $\psi \in \Delta$ is determined by formulas \eqref{infcharform} and \eqref{ccharform}. 
\end{itemize}

Finally, we relate integral refined shapes $\Delta$ to a class of Archimedean parameters defined in \cite{AJ87}, which we call AJ-parameters; see also \cite[\S 4.3]{DGG22}:

\begin{lem}
Let $\Delta$ be an integral refined shape. Then there exists an AJ-parameter~$\psi_\infty^\Delta$ 
such that for all $\psi \in \Delta$, $\psi_\infty = \psi_\infty^\Delta$. Furthermore, the induced localization map $\mc S_\Delta \to  \mc S_{\psi_\infty^\Delta}$ is also determined by $\Delta$. 
\end{lem}

\begin{proof}
This is more-or-less the same proof as \cite[Lem. 4.3.4]{DGG22}

We have that $\psi_\infty$ is determined by the $\tau_{i, \infty}$. However, checking the cases from~\cite[3.13]{CR15}, $\tau_{i, \infty}$ is a discrete parameter with infinitesimal character matching that of a finite-dimensional representation so it is the unique discrete series parameter of its infinitesimal character $\lb_i$. This forces $\psi_\infty$ to be an AJ-parameter. 

The localization statement comes similarly as in \cite[Lem. 4.3.4]{DGG22} by inspecting the formula \cite[(1.4.9)]{Art13} which holds both locally and globally. 
\end{proof}

\subsection{Shapes in the Conjugate Self-Dual Case}
For easy reference, we recall the definition of refined shapes in the conjugate self-dual case from \cite[\S 4]{DGG22}:
\begin{dfn}
In the conjugate self-dual case, a \emph{refined shape} is a sequence
\[
\Delta = (T_i, d_i, \lb_i, \eta_i)_i
\]
up to permutation and where the $T_i$ and $d_i$ are positive integers, $\lb_i$ is an infinitesimal character of rank $T_i$, and $\eta_i = \pm 1$.
\end{dfn}

As the in the self-dual case, the $\eta_i$ (which are now just a simple signs) are needed to determine $G(\Delta)$. 
\begin{dfn}
In the conjugate self-dual case, the refined shape $\Delta$ is \emph{integral} if the total infinitesimal character of $\Delta$ as in equation \eqref{infcharform} is integral. 
\end{dfn}
As explained in \cite[\S4]{DGG22}, all the analogous properties to \S\ref{shapepropertiesselfdual} for integral refined shapes  hold in the conjugate self-dual case. The arguments are in fact simpler since all $G \in \wtd{\mc E}_\sm(N)$ have discrete series at infinity.

\subsection{Unrefined Shapes}\label{sec unrefinedshapes}
Given a refined shape $\Delta$, we can forget some data to give an unrefined shape $\Box := \Box(\Delta) = ((T_i, d_i, \eta_i))_i$:
\begin{itemize}
    \item In the self-dual case, we forget all information about the infinitesimal character except whether it takes integral or half-integral values. For indexing consistency, we fold this information into $\eta$ which now becomes a pair $(\om, \pm 1)$ of a quadratic central character and a sign.
    \item In the conjugate self-dual case, we forget the infinitesimal character. 
\end{itemize}
This still uniquely determines the group $G(\Delta) = G(\Box)$ as in \cite[Prop 4.3.1]{DGG22} and Proposition \ref{shapetogroup}. We also define trivial unrefined shape
\[
\Sigma_\eta := \Box(\Sigma_{\lb, \eta}). 
\]

Given infinitesimal character $\lb$ for $G(\Delta)_\infty$, Let $\Box_\lb$ 
be the set of $\Delta$ with total infinitesimal character $\lb$ such that $\Box(\Delta) = \Box$. We will sometimes shorthand
\[
(\lb_1, \dotsc, \lb_i) \in \Box_\lb \implies (T_i, d_i, \lb_i, \eta_i)_i \in \Box_\lb. 
\]
since $\Box$ already contains the data of the $T_i, d_i, \eta_i$. 

We define $\lb[d]$ to be the total infinitesimal character of the simple shape $(T, d, \lb, \eta)$ as in \eqref{infcharform}. Similarly, let $\boxtimes_i \lb_i$ be the total infinitesimal character of $(T_i, 1, \lb_i, \eta_i)_i$---i.e. concatenating all the coordinates of each $\lb_i$ and ordering them increasingly.

Finally, to an unrefined shape $\Box = (T_i, d_i, \eta_i)_i$ we associate a ``reduced group''
\[
G_F(\Box) := \prod_i G((T_i, d_i, \eta_i)).
\]
This is not in general an element of $\wtd{\mc E}_\el(N)$ and can be thought of as the smallest group through which $\psi \in \Box$ functorially factor. It will appear in bounds on growth rates of counts of automorphic reps $\pi$ associated to $\psi \in \Box$.

\subsection{Norms of Infinitesimal Characters}\label{sec infcharnorms}
To prove our bounds, we will compare two different norms of infinitesimal characters $\lb$ of some $G \in \wtd{\mc E}_\sm(N)$: 
\begin{itemize}
\item the dimension of the finite-dimensional representation of $G$ corresponding to $\lb$:
\[
\dim \lb := \dim^G \lb := C_G \prod_{\alpha \in \Phi_+(G)} \langle \alpha, \lb \rangle 
\]
where the constant $C_G$ only depends on $G$,
\item a minimum over positive coroots of $G$:
\[
m(\lb) := m^G(\lb) := \min_{\alpha \in \Phi_+(G)}  \langle \alpha, \lb \rangle.
\]
\end{itemize}

Given an unrefined shape $\Box$, we can also define
\[
\dim_\Box \lb := \dim^G_\Box \lb :=  \max_{(\lb_i)_i \in \Box_\lb} \prod_i \dim \lb_i
\]
Note that this can be $0$ if $\Box_\lb$ is empty. 

Given a group $G$, we can compute the number of positive roots:
\[
P_G := \f12 (\dim G - \rank G).
\]
This gives us a bound: 
\begin{lem}\label{dimboundbym}
\[
\dim_\Box \lb \leq (\dim \lb) m(\lb)^{P_{G_F(\Box)} - P_G}.
\]
\end{lem}

\begin{proof}
The factors $\langle \alpha, \lb \rangle$ in the Weyl dimension formula for $\dim_\Box \lb$ are always a subset of those in $\dim \lb$. However, $\dim_\Box \lb$ has $P_G - P_{G_F(\Box)}$ fewer factors.
\end{proof}

\section{Weight Asymptotics for the Cuspidal Spectrum}\label{sec weightasymptotics}
\subsection{Setup and Trace Formula Decompositions}\label{sec weightsetup}
We now apply the techniques of \cite[\S5]{DGG22} to compute our asymptotics. Fix either the self-dual or conjugate self-dual case. Let $G \in \wtd{\mc E}_\sm(N)$. Choose: 
\begin{itemize}
\item
a regular, integral infinitesimal character $\lb$ for $G_\infty$ and $\EP_\lb$ the associated Euler-Poincar\'e function, endoscopically normalized as in \cite[\S 3.3]{DGG22},
\item
a set of unramified places $S$ and a test function $f_S$ in the truncated Hecke algebra $\ms H^\ur(G_S)^{\leq \kappa}$ (see \cite{ST16}),
\item
a test function $f^{S, \infty}$ at other places.
\end{itemize}
Choose $\Box = (T_i, d_i, \eta_i)$ a shape for $G$. We use the strategy of \cite[\S 9]{DGG22} to understand the asymptotics in $\lb$ of
\begin{multline*}
I^G_\Box(\EP_\lb f^{S, \infty} f_S) := \sum_{\Delta \in \Box_\lb} I^G_\Delta(\EP_\lb f^{S, \infty} f_S) := \sum_{\Delta \in \Box_\lb} \sum_{\psi \in \Delta} I^G_\psi(\EP_\lb f^{S, \infty} f_S) \\
= \sum_{\Delta \in \Box_\lb} \sum_{\psi \in \Delta} \sum_{\pi \in \Pi^G_\psi} m_\pi \tr_\pi(\EP_\lb f^{S, \infty} f_S) ,
\end{multline*}
where $I^G_\psi$ is the trace against the part of the discrete spectrum of $G$ associated to 
$\psi$ as in \cite[\S2.6]{DGG22} through \cite[Thm 1.5.2]{Art13} or \cite[Thm 2.5.2]{Mok15}.

In actuality, we will instead use the stable versions:
\begin{multline*}
S^G_\Box(\EP_\lb f^{S, \infty} f_S) := \sum_{\Delta \in \Box_\lb } S^G_\Delta(\EP_\lb f^{S, \infty} f_S) \\:= \sum_{\Delta \in \Box_\lb } \sum_{\psi \in \Delta} m_\psi \eps_\psi(s_\psi) |\mc S_\psi| \tr_{\Pi^G_\psi}(\EP_\lb f^{S, \infty} f_S),
\end{multline*}
which are expressed in terms of traces against $A$-packets and certain signs and constants that appear in the stable multiplicity formulas \cite[4.1.2]{Art13}, \cite[5.6.1]{Mok15} (note that we ignore the $\sigma$ terms since all our parameters are elliptic). 

The key formula for our final application is, the following, where we recall from \S\ref{sec unrefinedshapes} that in the self-dual case $\Sigma_\eta$ is the unrefined shape determined by a rank, a character $\eta$, and a choice of parity of entries of $\lb$:
\begin{prop}\label{stablecount}
In either the self-dual or conjugate self-dual case, let $\eta$ be such that $G(\Sigma_{\eta}) = G \in \wtd{\mc E}_\sm(N)$. Then for any infinitesimal character $\lb$ for $G$, 
\[
\sum_{\psi \in \Sigma_{\eta, \lb}} \tr_{\wtd \pi^\infty_\psi}(f^\infty) = \begin{cases}
(1/2) S^G_{\Sigma_\eta}(\EP_\lb (f^\infty)^G) & G = \SO_{2n}^\eta  \\
S^G_{\Sigma_\eta}(\EP_\lb (f^\infty)^G) & \text{else.}
\end{cases}
\]
\end{prop}

\begin{proof}
This is an easier version of the argument in \cite[Prop 5.3.1]{DGG22}. We note that for simple shapes, the constants in the stable multiplicity formula are all trivial except for $m_\psi$ which is $2$ in the even orthogonal case. 
\end{proof}

Following the above proposition, we wish to compute asymptotics in~$\lambda$ of $S^G_{\Sigma_\eta}(\EP_\lb (f^\infty)^G)$. In the rest of this section, we show inductively that the main term in these asymptotics is a geometric term consisting of evaluation at central elements. This mirrors our previous strategy to compute level asymptotics in~\cite[\S 5]{DGG22}. In Sections \ref{sec shalika germs}
 and \ref{sec stable plancherel}, we will compute the relevant evaluations at central elements for the functions $\wtd E_{\mf n}^\infty$ and $\wtd C_{\mf n}^\infty$ previously constructed.

\subsection{Geometric Estimation of Traces of \lm{$\EP_\lambda$}}
As input, we use the weight aspect bounds of \cite{ST16} generalized to deal with nontrivial center as in \cite{Dal22}. 

\subsubsection{Computation of Constants}
For $G/F$ a reductive group, $\lb$ a regular integral infinitesimal character of $G_\infty$, $\om_\lb$ the central character of the corresponding finite-dimensional representation of $G_\infty$, and $f^\infty$ a test function on $G(\A^\infty)$, the main term in our asymptotics is given by a linear combination of evaluation of $f^\infty$ at central elements:
\begin{equation}\label{mainterm}
\Lambda(G, \om_\lb, f^\infty) :=  \tau'(G) \sum_{\gamma \in Z_G(F)} \om^{-1}_\lb(\gamma) f^\infty(\gamma). 
\end{equation}
Here, $\tau'(G)$ is a modified Tamagawa number from \cite[\S6.6]{ST16} under canonical times Euler-Poincar\'e measure. 

Note that in our cases, $Z_G(F)$ always intersects the support of $f^\infty$ in a finite set.

\subsubsection{Statement}

\begin{thm}[Slight variant/Special case of {\cite[Thm. 9.1.1]{Dal22}}]\label{seedbound}
With notation as above, there are constants $A,B,C$ with $C \geq 1$ and depending only on $G$ such that,
\[
(\dim \lb)^{-1} S^G(\EP_\lb f_S f^{S, \infty}) = \Lambda(G, \om_\lb, f^\infty)+ O(m(\lb)^{-C} q_S^{A + B \kappa} f_S(1)),
\]
Where $\Lambda$ is as in \eqref{mainterm}. 
\end{thm}

\begin{proof}
Hyperendoscopy techniques as in the proof of \cite[Thm. 9.1.1]{Dal22} let us write $S^G$ in terms of $I^H$ for smaller groups. Then this follows from the exact same argument. 
\end{proof}

\subsection{Preliminary Bounds for the Induction}
Since we do not need to keep track of functions at split places, the bounds of \cite[\S 7]{DGG22} simplify. Recall the definitions of $\mc T := \mc T_\Delta, \mc T_i := \mc T_{\Delta,i}$ from \cite[\S6.2.4]{DGG22}.

\begin{prop}\label{newshapebound}
Let $\Delta= (T, d, \lb, \eta)$ be a simple shape and $f_S f^{S, \infty}$ a trace-positive test function on $G(\Delta)^\infty$ with $f_S$ unramified. Then there is a test function~$(f^{S, \infty})'$ on $G(1, d, \lb, \eta)^{S, \infty}$ such that
\[
\lf|S^{Td}_\Delta\lf( (f_S f^{S, \infty})^{Td}\ri)\ri| \leq S^T_{\Sigma_{\eta,\lb}} \lf((\mc T f_S (f^{S, \infty})')^T\ri).
\]
\end{prop}

\begin{proof}
This is the same argument as \cite[Prop. 7.3.2]{DGG22} with empty ${S_s}$. We use a slightly modified version of \cite[Lem. 6.3.7]{DGG22}, noting that the proof does not actually require the exponent $d$. 
\end{proof}

\begin{prop}\label{fullnewshapebound}
Let $\Delta= (T_i, d_i, \lb_i, \eta_i)_i$ be a shape of rank $N$ and $f_S f^{S, \infty}$ a trace-positive test function on $G(\Delta)^\infty$ with $f_S$ unramified. Then there are test functions $f^{S, \infty}_i$ on each $G(\Sigma_{\eta_i, \lb_i })^{S, \infty}$ such that
\[
|S^N_\Delta((f_S f^{S, \infty})^N)| \leq \prodf_i S^{T_i}_{\Sigma_{\eta_i, \lb_i }} \lf((f^{S, \infty}_i\mc T_i f_S )^{T_i}\ri)
\] 
\end{prop}

\begin{proof}
This is like the proof of \cite[Prop 7.3.3]{DGG22}: apply \cite[Cor 7.2.2]{DGG22} with empty $S_s$ and then Proposition \ref{newshapebound}.
\end{proof}

\subsection{Inductive Result}
Now we can induct to get asymptotics of the contribution of $\Sigma_\eta$ to the stable trace:
\begin{thm}\label{weightasymptotics}
Let $G \in \wtd{\mc E}_\sm(N)$ in either the self-dual or conjugate self-dual case and let $\lb$ be an infinitesimal character for $G$. Then there are $A,B,C$ with $C \geq 1$ such that 
\begin{multline*}
(\dim \lb)^{-1} S_{\Sigma_\eta}^G(\EP_\lb f_S f^{S, \infty}) \\
= \begin{cases}
\Lambda(G, \om_\lb, f^\infty)+ O(m(\lb)^{-C} q_{S}^{A + B \kappa}f_S(1)) & G \neq \SO_{2, \triv} \\
0 & G = \SO_{2, \triv},
\end{cases}
\end{multline*}
where $\Sigma_\eta$ is the unrefined trivial shape such that $G(\Sigma_\eta) = G$ and $\Lambda(\cdot)$ is given by~\eqref{mainterm}. 

In addition, if $\Box$ is an unrefined shape and $G = G(\Box)$,
\begin{equation} \label{eq bound nontrivial shapes}
(\dim_\Box \lb)^{-1} |S_\Box^G(\EP_\lb f_S f^{S, \infty})| \leq C(G, \Box, f^\infty) + O(m(\lb)^{-C} q_{S}^{A + B \kappa}f_S(1))
\end{equation}
for some inexplicit constant $C(G, \Box, f^\infty)$ depending only on the arguments. 
\end{thm}

\begin{proof}
This is the same induction-on-classification-rank argument as in \cite[Thm. 9.3.2]{DGG22}. First, \cite[Lem. 9.3.1]{DGG22} reduces to the case where $f^\infty$ is trace-positive. In the self-dual case, there are three bases cases (those when $G$ is abelian):
\begin{itemize}
    \item $N=1$. Then $\Sigma_\eta$ is trivially the only shape so the bound follows from a very special case of Theorem \ref{seedbound}.
    \item $N=2$, $G = \SO_{2}^\triv$. Then there are no simple parameters since $\Ld G$ embeds into a split torus. 
    \item $N = 2$, $G = \SO_{2}^\eta$, $\eta \neq \triv$: then the only parameters with non-trivial shape are those for $\eta_1 \boxtimes \eta_2$ with $\eta = \eta_1 \eta_2$. Since each component $(\eta_i)_v$ for $v|\infty$ is valued in $\pm 1$, these only appear in our traces for finitely many $\lb$. Therefore, we can bound $S^G = S^G_{\Sigma_\eta} + E$ for an error $E$ that is constant in $\lb$ and apply Theorem~\ref{seedbound}. 
\end{itemize} 
In the conjugate self-dual case, the bases cases are $N=1$ and $G = U^\pm_1$ which follows as in the $N=1$ case above. 

For the inductive step, first note that by \cite[5.3.1]{DGG22} or an analogous statement in self-dual setting:
\[
S_{\Sigma_\eta}^G(\EP_\lb f_S f^{S, \infty}) = S^G(\EP_\lb f_S f^{S, \infty}) - \sum_{\Box \neq \Sigma_\lb}S_{\Box}^G(\EP_\lb f_S f^{S, \infty}).
\]
If we could show that the terms in the rightmost sum the have smaller growth in~$m(\lb)$ than~$\dim(\lb)$, then Theorem \ref{seedbound} would prove the bound we want for $\Sigma_\eta$. For this, outside of our base cases, we always have $P_{G_F(\Box)} < P_G$ so $\dim(\lb)^{-1} \dim_\Box(\lb)  = O( m(\lb)^{-1})$ by Lemma~\ref{dimboundbym}. In particular, knowing the second bound would imply that all the terms in the sum are lower order in $m(\lb)$ and still exponential in $q_v^\kappa$ by \cite[Lem. 6.2.3]{DGG22}.

Therefore, it suffices to show the second bound \eqref{eq bound nontrivial shapes}. For this,
\[
|S_{\Box}^G(\EP_\lb f_S f^{S, \infty})| \leq \sum_{\Box_\Delta = \Box} |S^N_\Delta((f_S f^{S, \infty})^N)| 
\]
Each summand can be bounded by Proposition \ref{fullnewshapebound}, again using \cite[Lem.6.2.3]{DGG22} to bound $\mc T_i f_S(1)$:
\begin{multline*}
\prodf_i S^{T_i}_{\Sigma_{\eta_i, \lb_i }} \lf((f^{S, \infty}_i\mc T_i f_S )^{T_i}\ri) \\
=  \lf( \prod_i \dim \lb_i \ri)\lf( \prodf_i\Lambda(G_i, \om_{\lb_i},  f^{S, \infty}_i\mc T_i f_S )  + O(m(\lb)^{-C} q_{S}^{A + B \kappa} f_S(1))\ri) \\
\leq \dim_\Box(\lb) \lf(\prodf_i  \Lambda(G_i, \om_{\lb_i}, f^{S, \infty}_i \mc T_i f_S ) + O(m(\lb)^{-C} q_{S}^{A + B \kappa}f_S(1))\ri)
\end{multline*}
by the inductive hypothesis.

The result then follows by noting that $|\Box_\lambda|$ is bounded above by a constant only depending on $\Box$ and taking a maximum over all constants that are in the exponent of $q_v$. Note that we can remove the dependence of the $\Lambda(G_i, \om_{\lb_i}, f^{S, \infty}_i \mc T_i f_S )$ on $\om_{\lb_i}$ through upper-bounding the sum defining it by a sum of absolute values. 
\end{proof}

\subsection{From \lm{$S$} to \lm{$I$}}
For future use, we note that the bound also applies to $I$-terms, even though we will not use it in this specific paper. 

\begin{thm}\label{thm:shapeboundI}
Choose $G^* \in \wtd{\mc E}_\el(N)$ and let $G = G^*$ in the self-dual case or an extended pure inner form of $G^*$ in the conjugate self-dual case. If $\lb \in \Box_\lambda$, the bounds in Theorem \ref{weightasymptotics} hold with the $S^G_\Box(\EP_\lb f_S f^{S,\infty})$ replaced by $I^G_\Box(\EP_\lb f_S f^{S,\infty})$ after also replacing all $f^\infty$ terms on the right with $(f^\infty)^{G^*}$.
\end{thm}

\begin{proof}
First, note that the case of non-quasisplit $G$ is conditional on \cite{KMSW14} and its publicly unavailable reference ``KMSb''.  

For the proof, we apply \cite[Cor 10.2.2]{DGG22} (or its exactly analogous self-dual version), referring to the notation therein and noting that
\[
I^G_\Box(\EP_\lb f_S f^{S, \infty}) = \sum_{\Delta \in \Box_\lb} \sum_{\pi_0 \in \psi^\Delta_\infty} \tr_{\pi_0}(\EP_\lb) m^G(\pi_0, \Delta, f_Sf^{S, \infty}).
\]

If $\Box$ is stable (in particular when it is $\Sigma_\lb$), the expansion in \cite[Cor 10.2.2]{DGG22} only has one term with $s=1$ so $I^G_\Box(\EP_\lb f_S f^{S,\infty}) = S^{G^*}_\Box(\EP_\lb (f_S f^{S,\infty})^{G^*})$. 

Otherwise, this follows from applying Theorem \ref{weightasymptotics} to each term of \cite[Cor 10.2.2]{DGG22} or its self-dual analog. The dominant contribution comes from $s=1$. 
\end{proof}

\begin{note}
For a shape $\Box$ with some $d_i > 1$, there is a subset $\Phi_\Box \subseteq  \Phi_+(G)$ such that  $\langle \alpha, \lb \rangle$ is constant over $\alpha \in \Phi_\Box$. The $\alpha \in \Phi_\Box$ are exactly those not appearing in the computation of $\dim_\Box \lb$. 

An inspection of the argument above shows that the error term $O(m(\lb)^{-1})$ in Theorems \ref{weightasymptotics} and \ref{thm:shapeboundI} can be improved to an analogous $O(m_\Box(\lb)^{-1})$ where~$m_\Box(\lb)$ only minimizes over $\alpha \notin \Phi_\Box$. This is irrelevant to the current paper where we only consider $\Box = \Sigma$. 
\end{note}

\section{Shalika Germs and \lm{$(\wtd E^\infty_\mf n)^G$}} \label{sec shalika germs}
To apply Theorem \ref{weightasymptotics}, we need to understand the main term \eqref{mainterm} given by a sum over central elements for $(\tilde f^\infty)^G = (\wtd E^\infty_\mf n)^G$, the function whose trace is a count of automorphic forms weighted by their root number. In this case, we wish to show that the main term \eqref{mainterm} vanishes. 

\subsection{Transfers at Non-Regular Elements through Shalika Germs}\label{sec shalikaargument}
In this section, we consider the general problem of understanding values of transfers at central elements. Consider a (possibly twisted) reductive 
group $\wtd G$ and $H \in \mc E_\el(\wtd G)$.  We will eventually apply this to $\wtd G = \wtd G_N$ and $H$ a classical group as in \S\ref{sec twistedendoscopicgroups}. 

We first note that for central $\gamma \in H$, 
\[
(\wtd f^\infty)^H(\gamma) = \mc O_\gamma((\wtd f^\infty)^H) = \mc{SO}_\gamma((\wtd f^\infty)^H).
\]
If $\gamma$ were strongly regular, then we could understand this through the standard orbital integral transfer identities as in \cite{KS99}:
\begin{equation}\label{transferidentity}
\mc{SO}^H_\gamma((\wtd f^\infty)^H) = \sum_i \Delta^{\wtd G}_H(\gamma, \wtd \gamma_i) \mc O^{\wtd G}_{\wtd \gamma_i}(\wtd f^\infty),
\end{equation}
where $\gamma$ is a norm of the stable class containing the rational-class representatives~$\wtd \gamma_i$ and the $\Delta^{\wtd G}_H$ are a choice of transfer factors. These would allow to evaluate~$\mc{SO}_\gamma((\wtd f^\infty)^H)$ in terms of orbital integrals of $f^\infty$. For arbitrary semisimple elements, doing so is much more complicated.

The standard technique (e.g. \cite[\S3]{Kot88}) to resolve this issue is comparing Shalika germ expansions. This was implemented in full generality \cite[Thm 2.4.A]{LS07} for non-twisted endoscopy. We mimic the argument in the twisted case using a comparison of transfer factors which was the main technical result of \cite{WTE}. This argument is known to experts and seems implicit in the stabilization of the twisted trace formula---we just reconstruct it to deduce some bounds from the details. 

We first recall the definition of non-regular stable orbital integrals: 

\begin{dfn}
Let $H/F_v$ be quasiplit reductive and $\gamma \in H$ an arbitrary semisimple element. If the $\gamma_i$ are representatives for the rational classes in the stable class of $\gamma$, then
\[
\mc{SO}_\gamma := \sum_i e(H_{\gamma_i}) \mc O_{\gamma_i},
\]
where $e(H_{\gamma_i})$ is the Kottwitz sign of the necessarily connected $H_{\gamma_i}$ (i.e. $(-1)^r$ where~$r$ is the difference between the split ranks of $H_{\gamma_i}$ and its quasisplit inner form). 
\end{dfn}

The main goal of what follows is Proposition \ref{transfervanishing}. For this, we compare two Shalika germ expansions of \eqref{transferidentity}: with respect to $H$ on the left-hand side, and with respect to $\wtd G$ on the right-hand side. The latter relies on a descent property for twisted transfer factors, which we begin by recalling below. 

\subsubsection{Descent for Twisted Transfer Factors}\label{sectiontransferdescent}
The descent property we need is the main result of \cite{WTE}; an English-language summary of a special case can be found in \cite{Var09}. We begin with a very general setup: $G/F_v$ is an arbitrary reductive group, and we let~$\wtd G := G \rtimes \theta$ for an automorphism $\theta$ fixing a pinning. Let $H \in \mc E_\el(\wtd G)$ be such that there is a transfer map defined to $H$ instead of to a $z$-extension (this will be satisfied for all $(\wtd G, H)$ appearing in our self-dual and conjugate self-dual cases). 

Fix semisimple $\td \gamma \in \wtd G_v$ and $\gamma \in H_v$ such that $\gamma$ is the norm of $\td \gamma$ (i.e. $\td \gamma$ is the stable class that is the transfer of $\gamma$). As explained in \cite[\S1.3-4]{LS07}, since $H$ is quasisplit, we can without loss of generality replace $\gamma$ by a stable conjugate so that~$H_\gamma = H^0_{\gamma}$ is quasisplit. Then \cite[\S3.5-6]{WTE} constructs an $\bar H$ such that:
\begin{itemize}
    \item $\bar H$ is an endoscopic group of $(G_{\td \gamma}^0)^\scn$ (Waldspurger more precisely constructs the entire endoscopic quadruple). 
    \item $(H^0_\gamma)^{\scn}$ is a quasisplit ``non-standard endoscopic group'' for $\bar H$ as in \cite[\S1.7]{WTE}. In particular, there is transfer of stable conjugacy classes from $\mf h_\gamma$ to $\bar{\mf h}$ constructed from an isomorphism of tori contained in rational Borels (note that this factors through the canonical projection $\mf h_\gamma \to \mf h_{\gamma, \scn}$). 
\end{itemize}

Now, choose strongly $\wtd G$-regular $X_H \in \mf h_\gamma$. Then $X_H$ transfers to a stable class~$X_{\bar H}$ in~$\bar{\mf h}$ and further to a stable class in~$\mf g_{\td \gamma}$. Choose $X_G$ in this stable class of $\mf g_{\td \gamma}$. We can decompose it as
\[
X_G = \wtd X_{G,\scn} + \wtd X_{G,z}, \qquad \wtd X_{G,\scn} \in \mf g_{\td \gamma, \scn},\;  \wtd X_{G,z} \in \Lie(Z_{G_{\wtd \gamma}}). 
\]

Finally, we let $\Delta^{\wtd G}_H$ and $\Delta^{\mf g_{\wtd \gamma, \scn}}_{\mf h_\gamma}$ be the transfer factors from \cite{KS99} for the groups/Lie algebras determined by the sub/super-scripts (but normalized as in \cite{WTE} without the $\Delta_{IV}$-factor).

\begin{thm}[{\cite[Thm 3.9]{WTE}}]\label{transferdescent}
With notation from above, there is a choice of transfer factors and a neighborhood $0 \in U \subset \mf h_\gamma$ such that for $X_H \in U$,
\[
\Delta^{\wtd G}_H(\exp(X_H)\gamma, \exp(X_G)\wtd \gamma) = \Delta^{\mf g_{\wtd \gamma, \scn}}_{\bar{\mf h}}(X_{\bar H}, X_{G,\scn}).
\]
\end{thm}

\begin{proof}
The statement of \cite[Thm 3.9]{WTE} is an equality of relative transfer factors which, as a note afterwards points out, can become an equality of transfer factors after making the correct choices. We also comment on notation for the reader's convenience: our $\wtd \gamma$ is Waldspurger's $y \eta y^{-1}$ and our $\gamma$ is their $\eps$. Our assumptions on $H$ let us ignore all subscripted $1$'s. 

We write out Waldspurger's map $\varphi$ explicitly as it is defined before Remark 3.8. Similarly, we explicitly write out $\bar G[y]$ as it is defined at the end of \S3.5. 
\end{proof}

As a nice corollary: 
\begin{cor}\label{nonregtransferfactors}
With notation from above, assume that $\dim H_\gamma = \dim G_{\wtd \gamma}$. Then $\Delta^{\wtd G}_H(\exp(X_H)\gamma, \exp(X_G)\wtd \gamma)$ is constant for $X_H$ close enough to $0$. 

In particular, we can define $\Delta^{\wtd G}_H(\gamma, \wtd \gamma)$ to be this constant value. 
\end{cor}

\begin{proof}
Such corollaries were originally pointed out in \cite[\S2.4]{LS07}: by the dimension assumption, $\bar H$ is necessarily an elliptic endoscopic group of $(G^0_{\wtd \gamma})_\scn$ of the same dimension, so it is necessarily the quasisplit inner form. Therefore $\Delta^{\mf g_{\wtd \gamma, \scn}}_{\bar{\mf h}}(X_{\bar H}, X_{G,\scn}) = 1$ always. The result follows from Theorem~\ref{transferdescent} noting that different choices of transfer factors differ by a constant value. 
\end{proof}

\subsubsection{Shalika Germs}
We recall some standard material on Shalika germs (see e.g.~\cite[\S II.2]{MW16}). Let $\wtd G = G \rtimes \theta$ be a possibly twisted reductive group over $F_v$ and fix semisimple $\td \gamma \in \wtd G_v$. For any test function $f$ on $\wtd G_v$, there is a neighborhood~$1 \in U_f \subseteq G_{v, \td \gamma}$ on which the Shalika germ expansion holds:  if~$\mf u_f = \exp^{-1}(U_f) \subseteq \mf g_{\td \gamma}$, then for all regular semisimple $X \in \mf u_f$,
\begin{equation}\label{germexpansion}
\mc O_{\exp(X)\td \gamma}(f) = \sum_{U \in \mc N(\mf g_{\wtd \gamma})} \Gamma^{\mf g_{\wtd \gamma}}_U(X) \mc O_{\exp(U)\wtd \gamma}(f)
\end{equation}
where $\mc N(\mf g_{\td \gamma})$ is the set of nilpotent $G_{v,X}$-orbits in $\mf g_{\td \gamma}$ and $\Gamma^{\mf g_{\td \gamma}}_U(X)$ is the Shalika germ---a smooth function on $\mf g_{\td \gamma}^\rss$ independent of $f$ such that for all $\alpha \in F_v^\times$: 
\begin{equation}\label{shalikahomog}
\Gamma^{\mf g_{\wtd \gamma}}_U(\alpha^2 X) = |\alpha|_v^{-\dim \mc O_U} \Gamma^{\mf g_{\wtd \gamma}}_U(X).
\end{equation}

\subsubsection{Comparing Germ Expansions}
We can now mimic the argument of \cite[2.4.A]{LS07}. Fix a reductive, possibly twisted, $\wtd G = G \rtimes \theta$ over $F_v$, $H_v \in \mc E_\el(\wtd G_v)$ and semisimple $\gamma \in H_v$. 

Let $T_H$ be a maximal torus of $H_v$ that is also an elliptic maximal torus of $H_\gamma$. Let $j_i$ for $1 \leq i \leq l$ be representatives for the $H_v$-conjugacy classes of admissible embeddings $T_H \into H_v$. 

For any test function $f$ on $G_v$, there is a neighborhood $U \subseteq T_H$ such that the Shalika germ expansion holds for $f^H$ on each $j_i(U)$. Then for strongly regular $X_H \in \mf u = \exp^{-1}(U)$,
\begin{equation}\label{germexpansiononH}
\mc{SO}_{\exp(X_H)\gamma}(f^H) = \sum_i \sum_{U \in \mc N(\mf h_{\gamma_i})} \Gamma^{\mf h_{\gamma_i}}_U(j_i(X_H)) \mc O_{\exp(U)\gamma_i}(f^H)
\end{equation}
where we use the shorthand $\gamma_i := j_i(\gamma)$.
This gives us:
\begin{lem}\label{shalikatransferformula}
Choose a test function $f^H$ on $H$, semisimple $\gamma \in H$, and strongly regular $X_H \in \mf t_H$ (i.e. the Lie algebra of our elliptic torus $T_H$ of $H_\gamma$). Then the function
\[
r(\alpha) := \mc{SO}_{\exp(\alpha^2X_H)\gamma}( f^H)
\]
for small enough $\alpha$ is of the form
\[
r(\alpha) = C \mc{SO}_\gamma(f^H) + \sum_{d_i} C_{d_i} |\alpha|_v^{-d_i}
\]
for $C \neq 0$, some set of $d_i > 0$, and arbitrary constants $C_{d_i}$.  
\end{lem}

\begin{proof}
This is the intermediate result of \cite[\S3]{Kot88} up until the third equation on page 639. Summarizing the argument, Kottwitz computes the homogeneous of degree-$0$ part of \eqref{germexpansiononH} after applying \eqref{shalikahomog}: first, the main result of \cite{Rog81} computes all the $\Gamma_1$ terms on the elliptic torus $T_H$. Then, formulas from \cite{Kot86} simplify the remainders into the definition of the non-regular stable orbital integral\footnote{We warn that this definition is not simply the sum of orbital integrals and instead involves ``Kottwitz signs'' of centralizers as coefficients.}.

Kottwitz uses a clever choice of measure normalization to simultaneously remove the formal degree of the Steinberg from the formula in \cite{Rog81} and to simplify the relation between embeddings $j_i$ and the non-regular stable orbital integral. 
\end{proof}

Fix $X := X_H \in \mf t_H$ such that  $\exp(\alpha^2 X)\gamma$ is strongly regular for some open set of small enough $\alpha$. Then, setting $\wtd \gamma_i = \tdj_i(\gamma)$, we get the other side of the comparison from~\eqref{transferidentity}:
\begin{multline} \label{eq first expansion stable}
\mc{SO}_{\exp(X)\gamma}(f^H) \\ = \f{D^G(\wtd \gamma \exp(\tdj (X)))^{1/2}}{D^H(\gamma \exp(X))^{1/2}} \sum_i \Delta^{\wtd G}_H(\exp(X)\gamma, \tdj_i(\exp(X))\wtd \gamma_i) \mc O_{\tdj_i(\exp(X))\wtd \gamma_i} (f),
\end{multline}
where we again use Waldspurger's convention for transfer factors (leaving out the $\Delta_{IV}$ term) and the $D$ terms are Weyl discriminants:
\[
D^H(\gamma) = |\det(1 - \Ad(\gamma)|_{\mf h /\mf h_\gamma})|_v
\]
(note that the $D^G$-term is stable-conjugacy invariant and therefore independent of the suppressed $i$). 

We can further restrict our choice of $X$ so that the quotient of $D$-terms simplifies to an expression
\[
\f{D^G(\wtd \gamma)^{1/2}}{D^H(\gamma)^{1/2}}\f{D^{\mf g_{\td \gamma}}(\tdj(X))^{1/2}}{D^{\mf h_\gamma}(X)^{1/2}}
\]
that involves a similar term for Lie algebras:
\[
D^\mf h(X) = |\det(\ad(X)|_{\mf h /\mf h_X})|_v.
\]
By Theorem \ref{transferdescent}, for small enough $X_H$, the transfer factors in \eqref{eq first expansion stable} simplify to
\[
\Delta^{\mf g_{\td \gamma_i, \scn}}_{\bar{\mf h}_i}(X_{\bar H_i}, \tdj_i(X)_\scn)
\]
after constructing appropriate $\bar H_i$ that is an endoscopic group of $(G_{\td \gamma_i}^0)_\scn$ as in \S\ref{sectiontransferdescent} and where $X_{\bar H_i}$ is the (non-standard) transfer of $X$ to $\bar H_i$. 

We can then Shalika-germ expand again, this time on $\wtd G$ and around $\tilde \gamma_i$; for $X$ small enough, the identity \eqref{eq first expansion stable} becomes:
\begin{multline}\label{germexpansiononG}
\mc{SO}_{\exp(X)\gamma}(f^H) = \f{D^G(\wtd \gamma)^{1/2}}{D^H(\gamma)^{1/2}} \f{D^{\mf g_{\td \gamma}}(\tdj(X))^{1/2}}{D^{\mf h_\gamma}(X)^{1/2}}  \\ \times \sum_i \sum_{U \in \mc N(\mf g_{\wtd \gamma_i, \scn})} \Delta^{\mf g_{\wtd \gamma_i, \scn}}_{\bar{\mf h}_i}(X_{\bar H_i}, \tdj_i(X)_\scn) \Gamma^{\mf g_{\wtd \gamma_i}}_U(\tdj_i(X)) \mc O_{\exp(U)\wtd \gamma_i}(f).
\end{multline}
We recall:
\begin{lem}[{\cite[Lem. 3.2.1]{Fer07}}]
Let $G/F_v$ be a reductive group and $H \in \mc E_\el(G)$. Then for any choice of Lie algebra transfer factors, there is a quadratic character $\chi_{G,H} : F_v^\times \to \pm 1$ such that the transfer factors satisfy:
\[
\Delta(\alpha X_H, \alpha X_G) =  \chi_{G,H} (\alpha) \Delta(X_H, X_G)
\]
for all $\alpha \in F_v^\times$. 
\end{lem}

In particular, since the quadratic character $\chi_{G,H}$ vanishes on $\alpha^2$, the summand in \eqref{germexpansiononG} for a single pair $(i,U)$ scales with $\alpha$ like:
\begin{multline}\label{germexpansiononGhomog}
\f{D^{\mf g_{\wtd \gamma}}(\alpha^2 \tdj(X))^{1/2}}{D^{\mf h_\gamma}(\alpha^2 X)^{1/2}}  \Delta^{\mf g_{\wtd \gamma_i, \scn}}_{\bar{\mf h}_i}(\alpha^2 X_{\bar H_i}, \alpha^2 \tdj_i(X)_\scn) \Gamma^{\mf g_{\wtd \gamma_i}}_U(\alpha^2 \tdj_i(X)) \\= C_{i,U} |\alpha|^{-\dim \mc O_U + \dim G_{\td \gamma} - \dim H_\gamma}  
\end{multline}
for some constant $C_{i,U}$ and small enough $\alpha$. 

\subsubsection{Results}
As a consequence that is surely known to experts but does not seem to be written down anywhere, we get the following key input for the vanishing of the main term in our global trace formula expressions:
\begin{prop}\label{transfervanishing}
Let $f$ be a test function on $\wtd G_v$ and $\gamma, \wtd \gamma_i$ as above. Assume that for all $i$ and all nilpotent orbits $U$ of $\mf g_{\td \gamma_i}$,
\[
\dim \mc O_U = \dim G_{\wtd \gamma} - \dim H_\gamma \implies \mc O_{\exp(U) \wtd \gamma_i}(f) = 0.
\]
Then $\mc{SO}_{\gamma}(f^H) = 0$. 

(Note that if $\dim G_{\wtd \gamma} = \dim H_\gamma$, we only need to check that $\mc O_{\wtd \gamma_i}(f) = 0$.)
\end{prop}

\begin{proof}
This comes from comparing the expansion in Lemma \ref{shalikatransferformula} to equations \eqref{germexpansiononG} and \eqref{germexpansiononGhomog} for $X := X_H \in \mf t_H$. Note that we get much stronger information about the terms in \eqref{germexpansiononGhomog}, in particular, that many of them need to vanish. However, we don't actually need this extra information. 
\end{proof}

We can also derive a more refined bound which we expect to be useful in a future work studying the level-aspect version of this problem:

\begin{prop}\label{transferbound}
Let $f$ be a test function on $\wtd G_v$ and $\gamma, \td \gamma_i$ as above. Assume that for all $i$ and all nilpotent orbits $U$ of $\mf g_{\td \gamma_i}$ such that 
\[
\dim \mc O_U = \dim G_{\wtd \gamma} - \dim H_\gamma,
\]
there is a bound
\[
|D^G(\wtd \gamma_i)^{1/2} \mc O_{\exp(U) \wtd \gamma_i}(f)| \leq M. 
\]
Then we can bound
\[
|D^H(\gamma)^{1/2}\mc{SO}_{\exp(X)\gamma}(f^H)| \leq MC_{G_v,H_v}
\]
for some constant $C_{\td G_v, H_v}$ depending only on $v,\td G,H$.  
\end{prop}

\begin{proof}
Let $\mc U$ be the set of such $(i,U)$. Comparing Lemma \ref{shalikatransferformula} and \eqref{germexpansiononG}, 
\[
|D^H(\gamma)^{1/2}\mc{SO}_{\exp(X)\gamma}(f^H)| \leq M \lf |\sum_{(i,U) \in \mc U}  C_{i,U}(f)\ri|,
\]
where the $C_{i,U}$ are as in \eqref{germexpansiononGhomog}. The number of terms in the sum can be bounded by a constant depending only on $G$ so it suffices to bound each term. Furthermore, since local fields have finitely many extensions of each degree, there are only finitely many possibilities for the pair $(H_\gamma, G_{\td \gamma_i})$ so it suffices to allow the $C_{i,U}$ bound to also depend on these centralizers. 

For our special $(i,U) \in \mc U$ where \eqref{germexpansiononGhomog} is locally constant, we compute~$C_{i,U}$ by plugging in a small enough fixed $\alpha_0^2 X_0$. ``Small enough'' depends on the neighborhood~$U$ in Theorem \ref{transferdescent} and the validity of the Shalika germ expansions. All these only depends on $H_\gamma, G_{\td \gamma_i}, H, G$. In other words, there is a small enough $\alpha_0^2 X_0$ depending only on $H_\gamma, G_{\td \gamma_i}, H, G$. This finishes the argument.
\end{proof}

Finally, we get an exact formula in special cases---a twisted version of \cite[Lemma 2.4.A]{LS07}:
\begin{cor}\label{equisingulartransfer}
Let $f$ be a test function on $\wtd G_v$ and $\gamma, \td \gamma_i$ as above. Assume that $\dim G_{\wtd \gamma} = \dim H_\gamma$. Then
\[
\mc{SO}_\gamma(f^H) = \sum_i e(G_{\wtd \gamma_i})\Delta^{\wtd G}_H(\gamma, \td \gamma_i) \mc O_{\wtd \gamma_i}(f),
\]
where $e(\cdot)$ is the Kottwitz sign (and recalling the definition of transfer factors at these potentially non-regular elements from Corollary \ref{nonregtransferfactors}). 
\end{cor}

\begin{proof}
We again compare Lemma \ref{shalikatransferformula} to \eqref{germexpansiononG}. The argument of \cite[2.4.A]{LS07} exactly carries over after replacing their discussion involving 1.6.A with Corollary \ref{nonregtransferfactors}. We point out that the signs $e(G_{\wtd \gamma_i})$ again come from 
\cite{Rog81} through another use of the aforementioned argument in \cite[\S3]{Kot88} and that the right-hand side should be though of as the ``correct'' definition of a non-regular $\kappa$-orbital integral. 
\end{proof}

\subsection{Application to \lm{$(\wtd E^\infty_\mf n)^G$}: Self-Dual Case}\label{sec E(1)}
We apply the result of \ref{sectiontransferdescent} to 
show the vanishing at central elements of the function $(\wtd E^\infty_\mf n)^G$ defined in \ref{Edef}. 

In this section, we will only study the self-dual case with $N$ even and $G = \SO_{N+1} \in \wtd{\mc E}_\sm(N)$ where root numbers are most interesting. The conjugate self-dual case is significantly more computationally involved and will be put off until next section. 

Note first that the class $1 \rtimes \theta \in \wtd G_N$ is the one that transfers to $1 \in G$ since $\theta$ preserves the standard pinning. To eventually use Proposition \ref{transfervanishing}, we note:
\[
(G_N)_{1 \rtimes \theta} = 
\begin{cases}
\Sp_N & N \text{ even} \\
\mathrm O_N & N \text{ odd}
\end{cases}.
\]
Furthermore, when $N$ is even, the stable class of $1 \rtimes \theta$ is the same as its rational class since $\Sp_N$ is semisimple and simply connected implying that $H^1(F_v, (G_N)_{1 \rtimes \theta}) = 0$ by a standard result on $p$-adic groups.

\subsubsection{Local Result}

We first study the indicator functions of cosets used to build the local factors of $\wtd E^\infty_\mf n$:
\begin{prop}\label{cosettransfervanishing}
Let $N$ be even, and choose $G = \SO_{N+1} \in \wtd{\mc E}_\sm(N)$ (i.e, $\wh G \cong \Sp_N$). Let $v$ be a place of $F$, choose $k > 0$, and let
\[
\td f_{v,k} := (J_{v,k}^{-1} w_N \rtimes \theta) \bar \1_{K_{1,v}(k)}.
\]
Then for all $\gamma \in Z_{G_v}$, $\td f_{v,k}^G(\gamma) = 0$.
\end{prop}

\begin{proof}
First, $\td f_{v,k}$ is the indicator function of the coset
\[
K_{1,v}(k) (J_{v,k}^{-1} w_N \rtimes \theta) = K_{1,v}(k) \begin{pmatrix}  & & & \pm \varpi_v^{-k} \\ & & \iddots &  \\  & \pm \varpi_v^{-k} & & \\ 1 &&& \end{pmatrix}  \rtimes \theta.
\]  
Elements of this coset 
have a non-zero entry in the bottom-left corner since  elements of $K_{1,v}(k)$ necessarily have a non-zero entry in the bottom-right when $k > 0$. 

However, the identity on $G$ transfers to the stable class of $1 \rtimes \theta$ in $\wtd G_N$ and all elements in this class are of the form $A \rtimes \theta$ with
\[
w_N A^T w_N^{-1} = A.
\]
In particular, since $N$ is even, the signs in $w_N$ force all elements in the stable class of $1 \rtimes \theta$ to have bottom left corner $a$ satisfying $a = -a \implies a = 0$. Therefore our coset doesn't intersect the stable class that transfers to the identity so $\mc O_{1 \rtimes \theta}(\td f_{v,k}) = 0$.

We then apply Proposition \ref{transfervanishing}. For the identity element, $\dim G_{1 \rtimes \theta} = \dim H$ so we only need to vanishing at $U = 0$ which follows by the above argument. Therefore $\mc{SO}_1(\td f_{v,k}^H) = \td f_{v,k}^H(1) = 0$. Since $\SO_{N+1}$ has trivial center, we are done.
\end{proof}

\subsubsection{Global Result}
We can then put these together to get a global result:
\begin{cor}\label{Etransfervanishing}
Choose $G = \SO_{N+1} \in \wtd{\mc E}_\sm(N)$ (i.e, $\wh G \cong \Sp_N$). Choose conductor 
\[
\mf n = \prod_v \mf p_v^{v(\mf n)}.
\]
If there exists finite place $v$ satisfying either:
\begin{itemize}
    \item $v(\mf n)$ is odd,
    \item $v(\mf n) > N$,
\end{itemize}
then $(\wtd E^\infty_\mf n)^G(1) = 0$. Otherwise, 
\[
(\wtd E^\infty_\mf n)^G(1) = \prod_v (-1)^{v(\mf n)/2}\binom{N/2}{v(\mf n)/2}.
\]
\end{cor}

\begin{proof}
First, $\wtd E^\infty_\mf n$ has $\wtd E_{v, v(\mf n)}$ as its factor at $v$. By Proposition \ref{Eexplicitformula}, $\wtd E_{v, v(\mf n)}$ is a linear combination of the $\td f_{v,k}$ from Proposition \ref{cosettransfervanishing} with $\td f_{v,0}$ appearing with coefficient $0$ if the stated 
conditions on $v$ hold, and coefficient $(-1)^{v(\mf n)/2}\binom{N/2}{v(\mf n)/2}$ otherwise. 

Next, Proposition~\ref{cosettransfervanishing} gives that for $k > 0$, $\td f_{v,k}^H(\gamma) = 0$. The twisted fundamental lemma gives that $\td f_{v,0}^H(\gamma) = \bar \1_{K_v \rtimes \theta}^H(1) = \bar \1_{K_{H,v}}(1) = 1$. The result follows. 
\end{proof}

\begin{note}
When $N$ is odd or $G \cong \SO_{N, \eta}$, $\wtd E^\infty_\mf n$ has positive trace against every cuspidal $A$-packet on $G$. Therefore, by Proposition~\ref{spectralpositivity} and Lemma \ref{existenceofcondN}, we instead have that both
\[
(\wtd E^\infty_\mf n)^G (1) \pm (\wtd E^\infty_\mf n)^G(-1) > 0.
\]
so the above computation needs to somehow fail to not contradict.

When $N$ is odd, we can see this in the failure of the argument of Proposition~\ref{transfervanishing} due to the negative sign for the bottom-left corner becoming positive. When~$N$ is even and $\wh G$ is orthogonal, $G$ no longer has the same dimension as $(G_{N,v})_\theta$, so the formula of Corollary~\ref{equisingulartransfer} doesn't apply and non-trivial unipotent orbital integrals in $\wtd G_{N,v}$ contribute to central values of $(\wtd E^\infty_\mf n)^G$. 
\end{note}

\section{Application to \lm{$(\wtd E^\infty_\mf n)^G$}: Conjugate Self-Dual Case}\label{sec E(1)conj}
The conjugate self-dual case is significantly more complicated. As preliminary computations:

\subsection{Twisted Conjugacy Classes}
\subsubsection{Transfers}\label{ss transfer central elements}
As in the self-dual case, the stable class of $1 \rtimes \bar \theta \in \wtd G_N$ transfers to the identity in any $G \in \wtd{\mc E}_\sm(N)$ since $\bar \theta$ fixes the standard pinning. Also,
\[
(G_N)_{1 \rtimes \bar \theta} = U^{E/F}_N. 
\]
As in the generalization after~\cite[5.1.C]{KS99}, we can compute transfers from~$Z_{G_N}(F_v) \rtimes \bar \theta$ to $Z_G(F_v)$:
\[
\beta \rtimes \bar\theta \mapsto \beta/\bar \beta,
\]
implicitly using $Z_{G_N}(F_v) = E_v^\times$ and $Z_G(F_v) = U_1(E_v)$. This map is surjective onto~$Z_G(F_v)$ by Hilbert 90.

\subsubsection{The Stable Class of the Identity}\label{identityclassdecompconj}
We compute how the stable ``identity class''~$1 \rtimes \bar \theta$ decomposes into rational classes. We recall (e.g. \cite[\S III.3]{Lab11}) that for $\gamma$ semisimple, the rational classes inside the stable class of $\gamma$ are in bijection with
\[
\ker(H^1(F, G^0_\gamma) \to H^1(F, G)).
\]
Applying this formula, consider the non-trivial case when $v$ is non-split. Then
\[
\ker(H^1(F_v, U^{E_w/F_v}_N) \to H^1(F_v,\Res_{F_v}^{E_w} \GL_N)) = H^1(F_v, U^{E_w/F_v}_N) = \Z/2\Z
\]
by Shapiro's lemma, Hilbert 90, and Kottwitz's formula from \cite{Kot86} for the unitary cohomology group. Therefore, the stable conjugacy class of $1 \rtimes \bar \theta$ breaks up into two rational conjugacy classes. 

We will need to distinguish the rational classes. Working over $\overline F_v$, $\wtd G_N(\overline F_v) \cong (\GL_N \times \GL_N)(\overline F_v)$ and conjugation acts by swapping the coordinates, so the stable conjugacy class over $F_v$ consists of all elements of the form $B \rtimes \bar \theta$ with $B \in G_N(F_v) \cong \GL_N(E_w)$ and
\[
B = \bar \theta(B)^{-1} \; (= w_N \bar B^T w_N).
\]
In particular, $\det B = \det \bar B$, so $B \rtimes \theta \mapsto \det(B)$ is a map from the stable conjugacy class to $F_v^\times$. We can also compute that all values of this invariant can be realized by matrices in the stable class (e.g, pick an appropriate anti-diagonal matrix). 

Since
\[
A(B \rtimes \bar \theta) A^{-1} = ABw_N \bar A^T w_N \rtimes \bar \theta,
\]
rational conjugacy can only scale the above determinant invariant by elements of~$N_{E_w/F_v}(E_w^\times)$. In total,
\[
B \rtimes \bar \theta \mapsto \det(B) \in F_v^\times/N_{E_w/F_v}(E_w^\times) \cong \Z/2\Z
\]
is a complete invariant determining the rational conjugacy class of $B \rtimes \bar \theta$. 

For $B \rtimes \bar \theta$ in the stable class of arbitrary $\gamma \rtimes \bar \theta$ for $\gamma \in Z_{G_{N,v}}$, we can instead use the invariant $\det(B)/\gamma^N \in F_v^\times/N_{E_w/F_v}(E_w^\times)$. 

\subsubsection{Transfer Factor Lemmas}
We finally need some information about transfer factors for $G = U^+_N \in \wtd{\mc E}_\sm(N)$:

\begin{lem}\label{lem U+tfidentity}
Consider $G = U^+_N \in \wtd{\mc E}_\sm(N)$ and $e \rtimes \bar \theta$ a representative for the other rational class in the stable class of $1 \rtimes \bar \theta$. Then for any choice of transfer factors:
\[
\Delta^{\wtd G_N}_G(1, 1\rtimes \bar\theta) =  \Delta^{\wtd G_N}_G(1, e \rtimes \bar \theta).
\]
\end{lem}

\begin{proof}
Any strongly regular $\gamma \rtimes \bar \theta$ close enough to $1 \rtimes \bar \theta$ would have stable conjugate~$\gamma' \rtimes \bar \theta$ close to $e \rtimes \bar \theta$. By Corollary \ref{nonregtransferfactors}, it suffices to show the result for all such pairs $\gamma \rtimes \bar \theta$ and $\gamma' \rtimes \bar \theta$ in some small enough neighborhood.

For this, we apply \cite[5.1.D]{KS99}. The hypercocycle $\kappa_T$ is determined by the cocycle $a_T$ defined on pp. 39-40 in loc. cit. Using notation therein, in the case of~$H = U^+_N$, we have that $\wh H = \wh G_1$ exactly. Therefore $a_T$ is trivial so $\kappa_T$ is as well. The result follows. 
\end{proof}

\begin{lem}\label{lem U+tfcenter}
Consider $G = U^+_N \in \wtd{\mc E}_\sm(N)$, and semisimple $\gamma \rtimes \bar \theta \in \wtd G_{N,v}$ and~$\delta \in G_v$. Then, for all $a \in E^\times = Z_{G_{N,v}}$,
\[
\Delta^{\wtd G_N}_G(\delta, \gamma) = \Delta^{\wtd G_N}_G((a/ \bar a)\delta, a \gamma).
\]
\end{lem}

\begin{proof}
This follows from the generalization after \cite[5.1.C]{KS99}. The character $\lb_C$ there must be trivial to be consistent with Lemma \ref{U+cchar}, see \cite[Rmk 3.2.3]{Mok15}. 

Note that for the non-strongly regular case, we prove this first for strongly regular classes close to $\gamma$ and $\delta$ and then use Corollary/Definition \ref{nonregtransferfactors}. 
\end{proof}

\subsection{Number-Theoretic Lemmas}
We also need two number-theoretic lemmas. Beware an indexing technicality: at places $v$ of $F$ that ramify in $E$ to places $w$, we state all results in terms of powers of $\mf p_w$. Let $\mf D_{E_w/F_v}$ be the different ideal.

\begin{lem}\label{lem OEcoboundaries}
Define the map
\[
\phi: \mc O_{E_w}^\times \to \mc O_{E_w}^\times, \quad \phi(x) = \frac{x}{\bar x}.
\]
Then
\[
\phi(1 + \mf p_w^k) \subseteq 1 + \mf D_{E_w/F_v} \cap \mf p_w^k
\]
Furthermore, if $E_w/F_v$ is at most tamely ramified or $k = 0$, then
\[
\phi(1 + \mf p_w^k) = (1+(\mf p_w^k \cap \mf D_{E_w/F_v}))_{Nm = 1}.
\]
\end{lem}
\begin{proof}
For all $x \in \mc O^\times_{E_w}$, we have $\phi(x) \equiv 1 \mod \mf D_{E_w/F_v}$ since $x - \bar x \in \mf D_{E_w/F_v}$ by \cite[IV.2]{Se13}. This gives the inclusion. 

For the equality, let $j\geq 0$ be such that $\mf D_{E_w/F_v} = \fp_w^j$. It suffices to show that $H^1(\Gal(E_w/F_v),1+\mf p_w^k)$ vanishes for all~$k \geq j$. If~$E_w$ is unramified, then $\mf D_{E_w/F_v} = \mc O_{E_w}$, and the vanishing for all~$k$ follows from \cite[VII.7.1.2]{NSW13cohomology}. This lemma also gives the desired vanishing tamely ramified case for $k\geq 1$ (since then $\mf D_{E_w/F_v} = \mf p_w$).

It remains to check the case when $v$ is wildly ramified and $k=0$: that $\phi(\mc O_{E_w}^\times) = (1+(\mf p_w^k \cap \mf D_{E_w/F_v}))_{Nm = 1}$. This is equivalent to checking that in the long exact sequence associated to 
\[
0 \to 1+\mf D_{E_w/F_v} \to \mc O^\times_{E_w} \to \mc O^\times_{E_w}/(1+\mf D_{E_w/F_v}) \to 0,
\] 
the natural map $H^1(\Gal(E_w/F_v),1+\mf D_{E_w/F_v}) \to H^1(\Gal(E_w/F_v),\mc O^\times_{E_w})$ is trivial. We check instead that 
\begin{equation} \label{eq cohomology inj}
H^1(\Gal(E_w/F_v),\mc O^\times_{E_w}) \hookrightarrow H^1(\Gal(E_w/F_v),\mc O^\times_{E_w}/(1+\mf D_{E_w/F_v})).
\end{equation}
A computation with the long exact sequence associated to $\mc O_{E_w}^\times \hookrightarrow E_w^\times \twoheadrightarrow \Z$ shows that $H^1(\Gal(E_w/F_v),\mc O^\times_{E_w})\simeq \BZ/2\BZ$, generated by the class of $\varpi/\overline \varpi$. To get \eqref{eq cohomology inj}, it suffices to show that this class survives under reduction modulo $1+\mf D_{E_w/F_v} =: 1+\mf p_w^j$. Since $ \varpi/\overline \varpi \in (1+ \mf p_w^{j-1}) \setminus  (1+ \mf p_w^{j})$ \cite[IV.1-2]{Se13}, the cocycle is nontrivial in the quotient. Furthermore, since $\mc O^\times_{E_w}/(1+\mf D_{E_w/F_v})$ is a trivial Galois module, any coboundary is trivial. In total, $\varpi/\overline \varpi$ represents a nontrivial class. 
\end{proof}

\begin{lem}\label{lem linearalgebracomp} 
Let $N$ be odd. There is $A \in \GL_N(\mc O_{E_w})$ such that $w_N \bar A^T w_N^{-1} = yA$ if and only if $N_{E_w/F_v}(y) = 1$ and either:
\begin{enumerate}
\item
$E_w/F_v$ is unramified,
\item
$E_w/F_v$ is tamely ramified and $y^N \not\equiv -1 \pmod{\mf p_w}$,
\item
$E_w/F_v$ is wildly ramified and $y^N \notin 1 + \mf p_w$ or $y \in 1 + \mf D_{E_w/F_v}$.
\end{enumerate}
\end{lem}

\begin{proof}
We instead construct $B = w_N A \in \GL_N(\mc O_{E_w})$, which would have to satisfy:
\[
\bar B^T =  \bar A^T w_N^{-1} = w_N^{-1} w_N \bar A^T w_N^{-1} = (-1)^{N-1} w_N y A = y B. 
\]
If $y$ can be written as $\alpha/\overline \alpha$ for $\alpha \in \mc O_E^\times$, then $B$ can be chosen to be diagonal with~$\alpha$'s on the diagonal.

If not, then any such $B$ has only zeros on the diagonal. Let its entries be $b_{ij}$ so that
\[
\det B = \sum_\sigma \sgn(\sigma) \prod_i b_{i \sigma(i)},
\]
where by the assumption on $B$, we can sum over just permutations that don't have fixed points. Comparing terms for $\sigma$ and $\sigma^{-1}$,
\[
\sgn(\sigma^{-1}) \prod_i b_{i \sigma^{-1}(i)} = \sgn(\sigma) \prod_i b_{\sigma(i)i} = \sgn (\sigma)  y^N \prod_i \bar b_{i\sigma(i)}.
\]
Since $N$ is odd, all permutations with no fixed points aren't equal to their inverses so we can therefore pair them up and write $\det B$ in the form $R + y^N \bar R$ where $R$ can be freely chosen to be any element of $\mc O_{E_w}$. 

If $y^N \not\equiv -1 \pmod{\mf p_v}$, this can be arranged to be in $\mc O_{E_w}^\times$ by choosing $R \in \mc O_{F_v}^\times$. Otherwise, there is $R$ such that $R + y^N \bar R \in \mc O_{E_w}^\times$ if and only if there is $R$ such that~$R/\bar R \not\equiv 1 \pmod{ \mf p_v}$. Summarizing, the possible $y$'s are exactly those of norm $1$ for which at least one of the following hold:
\begin{enumerate}
\item
There is $R \in \mc O_{E_w}^\times$ such that $R/\bar R \not\equiv 1 \pmod{ \mf p_w}$,
\item
$y$ can be written in the form $\alpha/\overline \alpha$ for $\alpha \in \mc O_{E_w}^\times$,
\item
$y^N \not\equiv -1 \pmod{\mf p_w}$.
\end{enumerate}
As in Lemma \ref{lem OEcoboundaries}, (1) holds if and only if $E_w/F_v$ is unramified. For $E_w/F_v$ tamely ramified, (2) implies that $y \equiv 1 \pmod{\mf p_w}$ which together with $\mf p_v \nmid 2$ implies (3). 

If $E_w/F_v$ is wildly ramified, then $\mf p_w | 2$ so (3) is equivalent to $y^N \notin 1 + \mf p_w$. Also~(2) in our setting is equivalent to $y \in 1 + \mf D_{E_w/F_v}$ by Lemma~\ref{lem OEcoboundaries}. 
\end{proof}

\subsection{Local Results}\label{sec fvkconj}
Now we are ready for the main results. We again study coset indicators $\td f_{v,k}$ used to build up $\wtd E^\infty_\mf n$ in Corollary \ref{localepstestfunction}. Beware that we use local indexing for $\td f_{v,k}$; see \S\ref{sec conductorsconj}. 

\subsubsection{Non-Intersection of Support}

\begin{prop}\label{cosettransfervanishingconj} 
In the conjugate self-dual case, let $N$ be even and $G \in \wtd{\mc E}_\sm(N)$. Let $v$ be a place of $F$ that doesn't split in $E$ and $w$ the corresponding place of $E$. Let 
\[
\td f_{v,k} := 
(J_{v,k}^{-1} w_N \rtimes \bar \theta) \bar \1_{K_{1,v}(k)}. 
\]
Then for all $\gamma \in Z_{G_v} \subseteq E_w^\times$, the support of $\td f_{v,k}^G$ intersects the stable conjugacy class of $\gamma$ only if:
\begin{enumerate}
\item
$-\gamma \in \phi(1+\mf p_w^k)$ (see Lemma \ref{lem OEcoboundaries}), and:
\item
either:
\begin{enumerate}
    \item $E_w/F_v$ is unramified,
    \item $E_w/F_v$ is ramified and $k$ is even.
\end{enumerate}
\end{enumerate}
\end{prop}

\begin{proof}

As explained in \ref{ss transfer central elements}, the element $\gamma \in U_{E/F}(1)_v = Z_{G_v}$ transfers to the stable conjugacy class $\beta \rtimes \bar \theta$ of $\wtd G_N$ with $\beta/\bar\beta =\gamma$. This stable class  consists exactly of elements of the form  
\begin{equation}\label{eq stablebetacondition}
\beta A \rtimes \bar \theta \text{ with } w_N \overline A^T w_N^{-1} = A.
\end{equation}
We want to compare this stable class with the support of $\td f_{v,k}$, which 
consists of all elements of the form
\[
X (J_{v,k}^{-1} w_N \rtimes \bar \theta) = X \begin{pmatrix}  & & & \pm \varpi_v^{-k} \\ & & \iddots &  \\  & \pm \varpi_v^{-k} & & \\ 1 &&& \end{pmatrix}  \rtimes \bar \theta
\]
for $X \in K_{1,v}(k)$. We write $X$ in block matrix form
\begin{equation}\label{eq fvksupport}
X = \begin{pmatrix}
X_0 & \vec v_1 \\
\vec v_2 & x
\end{pmatrix} \iff  X (J_{v,k}^{-1} w_N \rtimes \bar \theta) =  \begin{pmatrix}
\vec v_1 & \varpi_v^{-k} X_0 w_{N-1} \\
x & \varpi_v^{-k} \vec v_2 w_{N-1}
\end{pmatrix} \rtimes \bar \theta.
\end{equation}
We now assume that there exists such an element $X (J_{v,k}^{-1} w_N \rtimes \bar \theta)$ in the stable class of $\beta \rtimes \bar\theta$ and derive restrictions on $\gamma$.

\noindent \underline{Necessary Condition 1:}

As in the proof of Proposition \ref{cosettransfervanishing}, we have a restriction from the bottom-left matrix entry $x$ of $X (J_{v,k}^{-1} w_N \rtimes \theta)$. Since $N$ is even, condition \eqref{eq stablebetacondition} gives that $x = \beta \alpha$ for $\alpha$ satisfying $\overline \alpha = - \alpha$. However, we also have $x \equiv 1 \pmod{\mf p_w^k}$. Therefore, the class of $\beta \rtimes \theta$ only intersects the support of $f_k$ if there is $\alpha \in E_w$ such that:
\begin{align} \label{eq cong conditions on alpha}
    \bar \alpha = -\alpha, \quad \alpha \beta \in \mc O_{E_w}, \quad \text{and} \quad \alpha\beta \equiv 1 \pmod{\mf p_w^k}.
\end{align}
This is equivalent to something simpler: if $\alpha, \beta$ satisfy conditions \eqref{eq cong conditions on alpha}, then there is $h = \alpha\beta -1 \in \mf p_w^k$ such that 
\[
\gamma = -(1 + h)/(1 + \bar h).
\]
Conversely, given $h \in \mf p_w^k$, $\gamma$ satisfying the above equation, and $\beta$ such that $\beta/ \bar \beta = \gamma$, setting $\alpha = (1 + h)/\beta$ satisfies conditions \eqref{eq cong conditions on alpha}. Finally, by Lemma \ref{lem OEcoboundaries}, the $\gamma = -(1 + h)/(1 + \bar h)$ condition is further equivalent to
\begin{equation}\label{eq gammacondition}
-\gamma \in \phi(1 + \mf p_w^k)
\end{equation}
for the map $\phi$ of Lemma \ref{lem OEcoboundaries}. This is condition (1).

\noindent\underline{Necessary Condition 2:}

We also have a restriction from the top-right $(N-1) \times (N-1)$ block in the right-hand side of \eqref{eq fvksupport}. Condition \eqref{eq stablebetacondition} requires that
\[
w_{N-1} \overline{(\beta^{-1} \varpi^{-k} X_0 w_{N-1})}^T w_{N-1}^{-1} = -(\beta^{-1} \varpi^{-k} X_0 w_{N-1}),
\]
where the minus sign appears since ``restricting'' the conjugation by $w_N$ to the block doesn't exactly give $w_{N-1}$. In particular $X_0 w_{N-1}$ is an eigenvector for the conjugate-linear involution on $GL_{N_1}(\mc O_{E_w})$ given by $Y \mapsto w_{N-1} \overline Y^T w_{N-1}^{-1}$, with eigenvalue 
\[
-(\overline \beta/\beta)(\overline \varpi^k/\varpi^k) = -\gamma^{-1} (\overline \varpi^k/\varpi^k).
\] 
Assume that $k>0$. Since $\vec v_2 \equiv 0 \pmod{\mf p_w^k}$ and $X$ is invertible over $\mc O_{E_w}$, the matrix $X_0$, and therefore $X_0 w_{N-1}$ must also be invertible over $\mc O_{E_w}$. 

In particular, $- \gamma^{-1} (\overline \varpi^k/ \varpi^k)$ satisfies the conditions of Lemma \ref{lem linearalgebracomp}.

\noindent \underline{Sufficiency}

In the converse direction, if $\gamma$ and $\beta$ satisfy \eqref{eq cong conditions on alpha} and $y = -\gamma^{-1}(\overline \varpi/\varpi)$ satisfies the conditions of Lemma \ref{lem linearalgebracomp} for $\GL_{N-1}(\mc O_{E_w})$, we can choose $X_0 w_{N-1} \in \GL_{N-1}(\mc O_{E_w})$ satisfying  
\[
w_{N-1} \overline{(X_0 w_{N-1})^T} w_{N-1}^{-1}  = -(\overline \beta/\beta)(\overline \varpi^k/\varpi^k)(X_0 w_{N-1}). 
\]
Then choosing $\alpha$ as above, the matrix 
\[
A = \begin{pmatrix}
& \varpi^{-k} \beta^{-1} X_0 w_{N-1} \\
\alpha &
\end{pmatrix}
\]
satisfies $w_N \bar A^T w_N^{-1} = A$ and $\beta A \rtimes \bar \theta$ is in the support of $\td f_{v,k}$.

\noindent \underline{Simplification of Condition 2} 

Finally, \eqref{eq gammacondition} lets us simplify condition 2. By Lemma \ref{lem OEcoboundaries}, condition \eqref{eq gammacondition} gives that $-\gamma^{-1}(\varpi/\overline \varpi)^{-k} \equiv (\varpi/\overline \varpi)^{-k}$ mod both $\mf p_w$ and $\mf D_{E_w/F_v}$. This congruence class mod either ideal is independent of the choice of $\varpi$ since different choices make the quotient vary only by $\alpha/\bar \alpha \equiv 1 \pmod{\mf D_{E_w/F_v}}$ for $\alpha \in \mc O_{E_w}$. Our desired conditions from Lemma~\ref{lem linearalgebracomp} therefore become
\begin{itemize}
    \item If $v$ is unramified: nothing
    \item If $v$ is tamely ramified: $(\varpi/\overline \varpi)^{-(N-1)k} \equiv 1 \pmod{\mf p_w}$
    \item If $v$ is wildly ramified: $(\varpi/\overline \varpi)^{-(N-1)k} \not\equiv 1 \pmod{\mf p_w}$ or $(\varpi/\overline \varpi)^{-k} \equiv 1 \pmod{\mf D_{E_w/F_v}}$.
\end{itemize}

A tamely ramified quadratic extension is generated by a square root of a uniformizer, so in this case $\varpi/\overline \varpi \equiv -1 \pmod{\mf p_w}$. Therefore, since $N-1$ is odd, the condition in Lemma~\ref{lem linearalgebracomp} becomes that $k$ is even. 

In the wildly ramified case, we note since $\mf D_{E_w/F_v}$ is exactly generated by $\varpi - \overline \varpi$, we have that $\varpi/\overline \varpi$ is a non-trivial element of $(1 + \varpi^{-1} \mf D_{E_w/F_v})/(1 + \mf D_{E_w/F_v})$. In particular $(\varpi/\overline \varpi) \equiv 1 \pmod{\mf p_w}$ always. On the other hand, since $v | 2$, the quotient is a $2$-group so we again have that $(\varpi/\overline \varpi)^{-k} \in 1 + \mf D_{E_w/F_v}$ holds if and only if $k$ is even. 
\end{proof}

\subsubsection{Orbital Integral Vanishing}
Our support result immediately gives:
\begin{cor}\label{cosettransfervanishingconjcor}
Consider the setup of Proposition \ref{cosettransfervanishingconj}. For all $\gamma \in Z_{G_v}$ that do not satisfy the conditions therein, we have $f^G_{v,k}(\gamma) = 0$. 
\end{cor}

\begin{proof}
Since $ \dim (G_v)   = \dim((G_{N,v})_{\gamma \rtimes \theta})$,  this follows from Proposition \ref{transfervanishing}. 
\end{proof}

In the case of $G = U^+_N$, we can say more:

\begin{prop}\label{cosettransfervanishingconj+}
Consider the setup of Proposition \ref{cosettransfervanishingconj} and specialize to the case where $G = U_N^+$. Assume that $\gamma$ satisfies the conditions therein. Then:
\begin{enumerate}
\resume
\item if $\gamma'$ also satisfies the conditions of Proposition \ref{cosettransfervanishingconj}, then $\td f^G_{v,k}(\gamma) = \td f^G_{v,k}(\gamma')$,
\item if $v$ is unramified, $f_{v,k}^G(\gamma)$ is non-zero with sign $(-1)^k$,
\item if $v$ is ramified, $f_{v,k}^G(\gamma) = 0$ whenever $k > 0$. 
\end{enumerate}
\end{prop}

\begin{proof} We prove the points in order:

\noindent\underline{Constancy}

For $i=1,2$, consider $\gamma_i \in Z_{G_v}$ satisfying the conditions of Proposition~\ref{cosettransfervanishingconj} and let $\gamma_i$ be a norm of~$\beta_i \rtimes \bar \theta \in Z_{G_{N,v}} \rtimes \theta$; we can choose $\beta_i \equiv a \pmod{\mf p_w^k}$ for some~$a \in E^\times_w$ such that $\bar a = -a$. Let $z = \beta_2/\beta_1$, noting that $z \equiv 1 \pmod{\mf p_w^k}$ so that~$z \in K_{1,v}(k)$. Let $e = z/\bar z = \gamma_2/\gamma_1$.

First, by Lemma \ref{lem U+tfcenter},
\[
\Delta^{\wtd G}_G(e\gamma, z \wtd \gamma) = \Delta^{\wtd G}_G(\gamma, \wtd \gamma). 
\]
Second, $z \in K_{1,v}(k)$ gives that $\td f_{v,k}$ is invariant under translation by $z$ so $\mc O_{\wtd \gamma}(\td f_{v,k}) = \mc O_{z \wtd \gamma}(\td f_{v,k})$ for any $\wtd \gamma \in \wtd G_{N,v}$. Let $\beta_i' \rtimes \bar \theta$ be a representative for the other rational class in the stable class of $\beta_i \rtimes \bar \theta$. Putting everything together, the formula in Corollary \ref{equisingulartransfer} shows that 
\begin{multline}\label{identitystableintegral}
\td f_{v,k}^G(\gamma_1)  \\ 
= e(G_{\beta_1 \rtimes \bar \theta}) \Delta^{\wtd G}_G(\gamma_1, \beta_1 \rtimes \bar \theta)  \mc O_{\beta_1 \rtimes \bar \theta}(\td f_{v,k}) + e(G_{\beta_1' \rtimes \bar \theta}) \Delta^{\wtd G}_G(\gamma_1, \beta_1' \rtimes \bar \theta)\mc O_{\beta_1' \rtimes \bar \theta}(\td f_{v,k}) \\
= e(G_{\beta_2 \rtimes \bar \theta}) \Delta^{\wtd G}_G(\gamma_2, \beta_2 \rtimes \bar \theta)  \mc O_{\beta_2 \rtimes \bar \theta}(\td f_{v,k}) + e(G_{\beta_2' \rtimes \bar \theta}) \Delta^{\wtd G}_G(\gamma_2, \beta_2' \rtimes \bar \theta)\mc O_{\beta_2' \rtimes \bar \theta}(\td f_{v,k}) 
\\ = \td f_{v,k}^G(\gamma_2).
\end{multline}
This proves the equality we want.

\noindent\underline{Unramified Calculation}

For the unramified calculation, choose again $\gamma \in Z_{G_v}$ that is a norm of $\beta \rtimes  \bar \theta \in Z_{G_{N,v}} \rtimes \theta$ such that the other rational class inside the stable class of $\beta \rtimes \bar\theta$ is represented by $\beta' \rtimes \bar\theta$. 

In the unramified case, $N_{E_w/F_v}(E_w^\times)$ consists of all elements of even valuation. Therefore, by the discussion in \S\ref{identityclassdecompconj}, we can without loss of generality choose the conjugation orbit of $\beta \rtimes \bar\theta$ to be only supported on elements of even parity while that of $\beta' \rtimes \bar\theta$ is only on those with odd parity. Note that since $\det \beta$ has even parity, this also implies that $G_{\beta \rtimes \bar \theta}$ is the quasisplit inner form of $G_v$ while  $G_{\beta' \rtimes \bar \theta}$ is the non-quasisplit one. In particular, since $N$ is even, $e(G_{\beta \rtimes \bar \theta}) = 1$ and $e(G_{\beta' \rtimes \bar \theta}) = -1$. 

However, on the support of $\td f_{v,k}$, the determinant has constant valuation $-k(N-1)$ which has the same parity as $k$. Therefore, only one of the orbital integrals in \eqref{identitystableintegral} is non-zero---the one for $\beta \rtimes \bar \theta$ if $k$ is even and that for $\beta' \rtimes \bar \theta$ if $k$ is odd. 

Finally, Lemma \ref{lem U+tfidentity} forces the two transfer factors to have the same sign. The consistency of our choice with the twisted fundamental lemma forces this common sign to be positive. Combining this with the above two paragraphs finishes the sign computation.

\noindent\underline{Ramified Calculation}

As in the unramified calculation, choose $\gamma, \beta \rtimes \bar \theta, \beta' \rtimes \bar \theta$. This computation is more convenient to do using a different presentation of $\wtd G_{N,v}$. We realize 
\[
\wtd G_N(E_w) \cong \GL_N(E_w)^2 \rtimes s
\]
where $s$ is the involution switching the coordinates and $\wtd G_{N,v}$  is embedded through
\[
g(J^{-1}_\mf n w_N \rtimes \bar \theta) \mapsto (g, J^{-1}_\mf n \bar g^{-T} J_\mf n) \rtimes s.
\]
If $E_w/F_v$ is ramified there is $\eps \in \mc O_{F_v} \setminus N_{E_w/F_v}(E_w^\times)$. Consider the $N \times N$ matrix 
\[
E = \begin{pmatrix}
\eps &  &  &  \\
 & \ddots & & \\
& & \eps &  \\
& & & 1
\end{pmatrix}. 
\]

Recall the form of elements in $K_{1,v}(k)(J_\mf n^{-1} w_N \rtimes \bar \theta)$ given in \eqref{eq fvksupport}. If $k \geq 1$, any such element additionally of the form in \eqref{eq stablebetacondition} must satisfy $\vec v_1 = \vec v_2 = 0$ since otherwise their coordinates have different valuations but must also be equal up to sign. In particular, $E$ commutes with all $X$ such that $X(J_\mf n^{-1} w_N \rtimes \bar \theta)$ is of the form of both \eqref{eq fvksupport} and \eqref{eq stablebetacondition}. Furthermore, $J^{-1}_\mf n \bar E^{-T} J_\mf n = E$, so in total,
\[
(E, 1) X(J_\mf n^{-1} w_N \rtimes \bar \theta) (E, 1)^{-1} = EX(J_\mf n^{-1} w_N \rtimes \bar \theta). 
\]
considering $(E,1)$ as an element of $G_N(E_w) = \GL_N(E_w) \times \GL_N(E_w)$. In particular, this preserves both the conditions of \eqref{eq fvksupport} and \eqref{eq stablebetacondition}---in other words, conjugation by $(E,1)$ sends the intersection of support of $\td f_{v,k}$ with the stable orbit of any $\beta \rtimes \bar \theta$ for $\beta \in Z_{G,v}$ to itself.

Finally, note that since $N$ is even, the determinants of $XE$ and $X$ descend to the two different classes in $F_v^\times / N_{E_w/F_v}(E_w^\times)$ so conjugation by $(E,1)$ also swaps the two rational orbits in the stable orbit. Since the measures on the different centralizers used to define orbital integrals in a stable class are related by conjugation in the algebraic closure, this gives that the two orbital integrals in \eqref{identitystableintegral} are equal. By the same argument as the unramified calculation, this implies that the transfer must vanish. 
\end{proof}

We need two extra arguments. When $v$ is ramified and $k = 0$ we will use the spectral arguments of Section \ref{sec stable plancherel}:

\begin{lem}\label{cosettransfervanishingramk0}
Consider the same setup as Proposition \ref{cosettransfervanishingconj} and assume additionally that $G = U^+_N$, $v$ is ramified, and $k=0$. Then for $\gamma \in Z_{G_v}$, $\td f_{v,0}^G(\gamma) > 0$ whenever $\gamma \equiv -1 \pmod{\mf D_{E_w/F_v}}$. 
\end{lem}

\begin{proof}
By Proposition \ref{cosettransfervanishingconj+} (1), it suffices to show that $\td f_{v,0}^G(-1) > 0$. However, as in Theorem \ref{globalepsilontraceconj}, the function $\delta_a \star \td f_{v,0}$ is trace-positive since unramified representations have trivial standard root number. Therefore, since conjugate self-dual unramified representations of $\GL_N(E_w)$ exist and are both conjugate-orthogonal and -symplectic, Proposition \ref{spectralpositivity} shows that $(\delta_a \star \td f_{v,0})^G(1) > 0$. 

Finally, by Lemma \ref{lem U+tfcenter}, $(\delta_a \star \td f_{v,0})^G(1) = \td f_{v,0}^G(a/\bar a) = \td f_{v,0}^G(-1)$. 
\end{proof}

Finally, in the case of $v$ split:
\begin{lem}\label{cosettransfervanishingsplit}
Consider the same setup as Proposition \ref{cosettransfervanishingconj} except that $v$ now splits in $E$ and we allow odd $N$. Then we can choose
\[
(\td f_{v,k})^G =  \bar \1_{K_{1,v}(k)}.
\]
In particular
\[
\td f_{v,1}^G(\gamma) = \begin{cases}
\vol(K_{1,v}(k))^{-1} & \gamma \in \mc O_{E_w}^\times \text{ and } \gamma \equiv 1 \pmod{\mf p^k_v} \\
0 & \text{ else}
\end{cases}.
\]
\end{lem}

\begin{proof}
This follows from Lemma 6.1.1(3) in \cite{DGG22} using that $J_{v,k}^{-1} w_N \rtimes \theta$ is an involution that normalizes $K_{1,v}(k)$. 
\end{proof}

\subsection{Global Result}
We now put these together to get a global vanishing result:

\begin{cor}\label{Etransfervanishingconj}
Let $N$ be even and consider $G = U_{N,+} \in \wtd{\mc E}_\sm(N)$. Fix a conductor 
\[
\mf n = \prod_v \mf p_v^{k_v}
\]
globally indexed as in \S\ref{sec conductorsconj}. If there exists a ramified place $v$ of $F$ such that either:
\begin{itemize}
    \item $k_v$ is half-integral,
    \item $k_v > N/2$,
\end{itemize}
then $(\wtd E^\infty_\mf n)^G(\gamma) = 0$ for all $\gamma \in Z_G(F)$.
\end{cor}

\begin{proof}
Recall the choice of $a \in E^\times$ such that $a/\bar a = -1$ from Definition \ref{def globalE}. Let $\td f_{\mf p_v^k}$ be as in Proposition \ref{cosettransfervanishingconj} except globally indexed. Then $\wtd E^\infty_\mf n$ is a linear combination of terms of the form
\[
\td f_{\mf n'} = \prod_{v \text{ split}}  \td f_{\mf p_v^{k_v}} \prod_{v \text{ ns}} \delta_{a^{N-1}} \star \td f_{\mf p_v^{k_v}}. 
\]
First, note that by Lemma \ref{lem U+tfcenter}, for all test functions $\wtd \varphi$ on $\wtd G_{N,v}$,
\[
(\delta_{a^{N-1}} \star \wtd \varphi)^G(\gamma) = \wtd \varphi^G((-1)^{N-1}\gamma).
\]
Therefore, by Proposition \ref{cosettransfervanishingconj+}, we have $\td f_{\mf n'}^G(\gamma) = 0$ for all $\gamma \in Z_{G,v}$ if there is a ramified $v$ such that $v(\mf n') > 0$. By Proposition \ref{Eexplicitformula}, the bulleted conditions exactly guarantee that all of the $\td f_{\mf n'}$ making up $\wtd E^\infty_\mf n$ satisfy this. 
\end{proof}

Unlike in the self-dual case of Corollary~\ref{Etransfervanishing}, without the conditions on $\mf n$, there can be multiple summands $\td f_{\mf n'}$ with non-zero transfers appearing with different signs. We of course expect the terms \eqref{aa} for these $\td f_{\mf n'}$ to have very different magnitudes and therefore not cancel. However, the techniques established here do not allow us to computing fine enough information on the magnitudes of transfers to show this. We therefore give a partial non-vanishing result which will allow us, in some instances, to describe the bias in signs as the weight varies:
\begin{cor}\label{Etransfernonvanishingconj}
Consider the setup of Corollary \ref{Etransfervanishingconj} and assume that $\mf n$ does not satisfy the conditions therein. Let $\mf n_\ur$ be the part of $\mf n$ without any ramified factors. Then
\[
\prod_{v \nmid \infty} (-1)^{k_v} \sum_{\gamma \in Z_G(F)} \om(\gamma) (\wtd E^\infty_\mf n)^G(\gamma) > 0 
\]
for any character $\om$ on $Z_{G} \subseteq E^\times$ such that $\om$ is trivial on $(1 + \mf n_\ur \mf D_{E/F})\cap \mc O_E^\times$ and non-trivial on $(1 + \mf n' \mf D_{E/F} ) \cap \mc O_E^\times$ for all proper $\mf n' | \mf n_\ur$. 
\end{cor}

\begin{proof}
Recall the notation in the proof of Corollary~\ref{Etransfervanishingconj}, in particular the choice of $a \in E^\times$ such that $\bar{a}/a=-1$. As in that proof, by Lemma~\ref{lem U+tfcenter}, all test functions $\wtd \varphi$ on $\wtd G_{N,v}$ satisfy 
\[
(\delta_{a^{N-1}} \star \wtd \varphi)^G(\gamma) = \wtd \varphi^G((-1)^{N-1}\gamma).
\]
The function $\wtd E^\infty_\fn$ is a linear combination of $f_{\mf n'}$ for conductors $\mf n' | \mf n$.  If $\mf n'$ has a ramified divisor, Proposition \ref{cosettransfervanishingconj+} (3) and Corollary \ref{cosettransfervanishingconjcor} give that $f^G_{\mf n'}(\gamma) = 0$ for all $\gamma$.

If $\mf n' \mid \fn$ has no ramified divisors, multiply together the appropriate cases of Corollary \ref{cosettransfervanishingconjcor}, Proposition \ref{cosettransfervanishingconj+} (1), and Lemmas \ref{cosettransfervanishingsplit} and \ref{cosettransfervanishingramk0} over all $v$. This gives that for $\gamma \in Z_G$, $\td f_{\mf n'}(\gamma)^G$ is non-zero and constant for $\gamma \in \mc O_E^\times$ with $\gamma \equiv -1 \pmod{\mf n' \mf D_{E/F}}$ and zero otherwise (using the $k=0$ case of Lemma \ref{lem OEcoboundaries} to find the image of $\phi$). Therefore, for such $\mf n'$
\begin{equation}\label{aa}
\sum_{\gamma \in Z_G(F)} \om(\gamma) \td f_{\mf n'}^G(\gamma) \neq 0
\end{equation}
whenever $\om$ is trivial on $(1 + \mf n' \mf D_{E/F}) \cap \mc O_E^\times$ and is $0$ otherwise. Furthermore, when this term is non-zero, Proposition \ref{cosettransfervanishingconj+}(2) shows that it has sign $\prod_{v \text{ ur}} (-1)^{k_v}$. 

If $\mf n$ does not satisfy the conditions of Corollary \ref{Etransfervanishingconj}, then Proposition \ref{Eexplicitformula} says that there is $\td f_{\mf n'}$ appearing in the linear combination defining $\wtd E^\infty_\mf n$ with $\mf n'$ having no ramified divisors. Note that all such $\mf n'$ divide $\mf n_\ur$. The conditions on $\om$ guarantee that \eqref{aa} vanishes for all of these $\mf n'$ except $\mf n_\ur$ itself. Therefore the result follows after noting that $\td f_{\mf n_\ur}$ appears with coefficient $\prod_{v \text{ ram.}} (-1)^{k_v}$ in $\wtd E^\infty_\mf n$. 
\end{proof}

Of course, we expect the full result:

\begin{conj}\label{expectedEtransfernonvanishingconj}
Corollary \ref{Etransfernonvanishingconj} should hold for all $\om$ that are trivial on $(1 + \mf n_\ur \mf D_{E/F}) \cap \mc O_E^\times$. 
\end{conj}

\section{A Spectral Argument and \lm{$(\wtd C^\infty_\mf n)^G$}} \label{sec stable plancherel}

We will also need to understand the main term \eqref{mainterm} for $(\wtd C^\infty_\mf n)^G$. Since $C_\mf n^\infty$ was defined spectrally, this will necessarily happen through a spectral argument:
\subsection{A Stable Fourier Inversion Formula}\label{stableinversion}

\subsubsection{Setup}\label{sec:inversionsetup}
Specifically, we use a ``stable'' Fourier inversion result for a formally defined ``stable'' Fourier transform. Fix $G \in \wtd{\mc E}_\sm(N)$ in either the self-dual or conjugate self-dual case and a finite set of finite places $S$. Let $\check G_S^\temp$ be the tempered unitary dual of $G_S$.

Let $\check G_S^{\st, \temp}$ be the set of tempered $L$-packets/parameters for $G_S$ (or in the case of $\SO_{2n}$, outer-automorphism orbits of these). Since tempered $L$-packets partition tempered representations by either \cite[Thm 1.5.1(b)]{Art13} or \cite[Thm 2.5.1(b)]{Mok15}, there is a map 
\[
\rho : \check G_S^\temp \to \check G_S^{\st, \temp}.
\]
Associated to $\rho$, we have pushforward and normalized pullback maps between functions on $\check G_S^\temp$ and those on $\check G_S^{\st, \temp}$:
\[
\rho_* f(\Pi) = \sum_{\pi \in \rho^{-1}(\Pi)} f(\pi), \qquad \bar \rho^*f(\pi) = |\rho^{-1}(\rho(\pi))|^{-1} f \circ \rho(\pi),
\]
where $|\rho^{-1}(\rho(\pi))|$ is the cardinality of the packet $\Pi$ containing $\pi$.
Dual to these we also have pullback and normalized pushforward between functionals on both spaces, denoted $\rho^*, \bar \rho_*$. Note that for all functions $f$ and functionals $\mu$,
\[
\rho_* \bar \rho^* f = f, \qquad \bar \rho_* \rho^* \mu = \mu.
\]
Conversely, $\bar \rho^* \rho_* f$ and $\rho^* \bar \rho_* \mu$ can be though of as averaging over $L$-packets. 

Next, for any $U_S \subseteq G_S$ such that $U_S \cap Z_{G_S}$ is compact, define
\[
\mc{PW}(G_S, U_S) = \{ \wh f : \check G^\temp_S \to \C : f \in \mc C_c^\infty(U_S)\},
\]
the Paley-Wiener space of Fourier transforms of smooth, compactly supported functions with support in $U_S$. We can similarly define an unrestricted $\mc{PW}(G_S)$ of all compactly supported functions. By looking at the image under pushforward, we also define $\mc{PW}^\st(G_S, U_S)$ and $\mc{PW}^\st(G_S)$.

To avoid the use of the Ramanujan conjecture\footnote{this proven for regular integral infinitesimal character at infinity in many of our cases in \cite{Shin11}, \cite{KS20}, \cite{KS23} through constructing associated Galois representations. However, these results are not yet comprehensive so we will instead use a lower-tech argument.}, we need to extend these constructions to the non-tempered specturm in way that is consistent with the endoscopic classification. Since non-tempered $L$-packets are not necessarily made up of irreducible representations, we need to work with a modified Fourier transform for non-tempered $\pi$:
\[
\wh f(\pi) := \tr_{\mc I_\pi} f
\]
where $\mc I_\pi$ is the unique parabolic induction of a tempered representation tensor a character that has $\pi$ as a Langlands quotient. Then, using the definition of non-tempered $L$-packets (see e.g. \cite[pp. 44-45]{Art13}), there is an extension our map
\[
\rho: \check G_S \to \check G_S^\st
\]
to have image in the space of all $L$-parameters such that for $f \in \mc C_c^\infty(G_S)$
\[
\rho_* \wh f(\Pi) = \tr_\Pi f
\]
for even non-tempered $L$-packets $\Pi$. Of course, our modified $\wh f$ still satisfies the Fourier inversion formula since it is unchanged on the support of $\mu^\pl$.

\subsubsection{Global Input}\label{stableinversionglobalinput}

Now we are ready for the main proof:
\begin{prop}\label{stablefourierinversioninput}
Let $G \in \wtd{\mc E}_\sm(N)$ in either the self-dual ($G \neq \SO_2)$ or conjugate self-dual case. Let $S$ be a finite set of finite places of $F$ and $U_S \subseteq G_S$ such that~$U_S \cap Z_{G_s}$ is compact. Then there exists a sequence of measures $\mu_n$ on $\check G_S^{\st}$ such that for all~$\wh f_S \in \mc{PW}(G_S, U_S)$, we have
\[
\lim_{n \to \infty} \rho^* \mu_n(\wh f_S) = \mu_{G_S}^\pl(\wh f_S). 
\]
\end{prop}

\begin{proof}
Since $Z_{G_\infty}$ is compact for the groups $G$ under consideration, there always exists compact open $K^{S, \infty} \subseteq G^{S, \infty}$ such that $U_S \times K^{S, \infty} \cap Z_G(F) = 1$. Let $\wh f_S$ be the Fourier transform of $f_S \in \mc C_c^\infty(U_S)$. 


Let $\lambda$ be a regular, integral infinitesimal character, and let
\[
f_\lb = (\dim \lb)^{-1}\EP_\lb f_S \bar \1_{K^{S, \infty}}.
\] 
Our strategy is to define 
\[ 
\mu_\lambda(\wh f_S) := S_{\Sigma_\eta}^G(f_\lambda), 
\] 
show that $\mu_\lambda$ is indeed a measure on $\check G_S$ pulled back from $\check G_S^{\st}$, and to show convergence to the Plancherel measure as $m(\lb) \to \infty$.

First, the equalities \cite[(5.3.1)]{DGG22} together with the twisted endoscopic character identities for simple $A$-parameters shows that
\begin{multline*}
\mu_\lb(f_\lb) = S^G_{\Sigma_\eta}(f_\lb) \\
= (\dim \lb)^{-1}\sum_{\Pi \in \Sigma_{\eta, \lb}} \lf(\sum_{\pi^{S, \infty} \in \Pi^{S, \infty}} \dim((\pi^{S, \infty})^{K^{S, \infty}}) \ri) \lf(\sum_{\pi_S \in \Pi_S} \tr_{\pi_S}(f_S) \ri), 
\end{multline*}
which by inspection is of the form $\rho^* \mu_\lb^\st$ for some $\mu_\lb^\st$ on $\check G^\st_S$.

For the convergence, we have from Theorem \ref{weightasymptotics} that
\[
S^G_{\Sigma_\eta}(f_\lb) = \tau'(G) f_S(1) \bar \1_{K^{S, \infty}}(1) + O(m(\lb)^{-C})
\]
using that our choice of $K^{S, \infty}$ guarantees that all non-identity terms in the 
sum \eqref{mainterm} defining $\Lambda$ vanish. Evaluating constants and using Fourier inversion shows that
\[
\lim_{m(\lb) \to \infty}S^G_{\Sigma_\eta}(\tau'(G)^{-1}\vol(K^{S, \infty}) f_\lb) = \mu^\pl(\wh f_S).
\] 
Combining gives that
\begin{equation}\label{limitinproof}
\lim_{m(\lb) \to \infty} \rho^* \mu^\st_\lb(\wh f_S) =  \tau'(G) \vol(K^{S, \infty})^{-1} \mu^\pl(\wh f_S).
\end{equation}
which is what we wanted after choosing an appropriate sequence of $\lb$.
\end{proof}

\subsubsection{Final Result}
Rephrasing this in familiar language:
\begin{thm}[Stability of Plancherel Measure]\label{stablefourierinversion}
Let $G \in \wtd{\mc E}_\sm(N)$ in either the self-dual or conjugate self-dual case and $S$ a finite set of finite places of $F$. Then for all $f_S \in \mc C_c^\infty(G_S)$,
\[
f_S(1) = \int_{\check G_S} \tr_{\pi_S}(f_S) \, d\mu^{\pl}(\pi_S) = \int_{\check G_S^\st} \sum_{\pi_S \in \Pi_S} \tr_{\pi_S}(f_S) \, d\bar \rho_* \mu^\pl(\Pi_S).
\]
\end{thm}

\begin{proof}
The first equality is by Fourier inversion.  The second is equivalent to proving that for all $U_S$ such that $U_S \cap Z_{G_S}$ is compact and $\wh f_S \in \mc{PW}(G_S, U_S)$,
\[
\mu^\pl(\wh f_S) = \rho^* \bar \rho_* \mu^\pl(\wh f_S). 
\]
Let $\mc{PW}^\st(G_S, U_S) := \rho_*\left(\mc{PW}(G_S, U_S)\right)$.
Proposition \ref{stablefourierinversioninput} gives that, as a linear functional on $\mc{PW}(G_S, U_S)$, the Plancherel measure $\mu^\pl$ is of the form $\rho^* \mu$ for some linear functional $\mu$ on $\mc{PW}^\st(G_S, U_S)$. 

The result then follows since $\bar \rho_* \circ \rho^* = \id$ on the linear dual  of $\mc{PW}^\st(G_S, U_S)$. 
\end{proof}
We note a corollary that might be of independent interest (though it follows from the formal degree conjecture whenever that is known---e.g. from \cite{BP21}):
\begin{cor}\label{cor formaldegree}
Let $G \in \wtd{\mc E}_\sm(N)$ in either the self-dual or conjugate self-dual case and $v$ a finite place. For any supercuspidal $L$-packet $\Pi_v$ of representations of~$G_v$ (i.e. whose $L$-parameter is irreducible and trivial on the Deligne-$\SL_2$), all $\pi_v \in \Pi_v$ have the same formal degree. 
\end{cor}

\begin{proof}
All representations $\pi_v$ in a supercuspidal $L$-packet are supercuspidal\footnote{This is part of the construction of the LLC through the endoscopic classification.}. Therefore, their compactly supported matrix coefficients $c_{\pi_v}$ satisfy that $\wh c_{\pi_v} = \1_{\{\pi_v\}}$ so that $c_{\pi_v}(1)$ is the formal degree of $\pi_v$. The result then follows from applying Theorem \ref{stablefourierinversion}. 
\end{proof}

\subsection{Positivity Test}\label{sec positivitytest}
We now derive from Theorem \ref{stablefourierinversion} a positivity test:

\subsubsection{Understanding $G_S^{\st, \temp}$}
We first need to understand $\check G_S^{\st, \temp}$ and $\bar \rho_* \mu^\pl$ better through Harish-Chandra's Plancherel formula. Without loss of generality, we assume $S = \{v\}$ is a singleton. See \cite{Sil96} for a review of the Plancherel formula which we will loosely follow and the proof of 3.5.1 in \cite{Art13} for a very terse summary focused on applications to the endoscopic classification. 

First, $\check G^\temp_v$ decomposes into pieces called $\mf l$-components. Consider the set of pairs $(M, \sigma)$ of Levis $M$ of $G_v$ and discrete series (modulo the  center) $\sigma$ of~$M$ up to conjugation by~$G_v$. If we fix a minimal Levi~$M_0$ and $M \supset M_0$, then the set $ M^*/\Om^G_M$ of unramified unitary characters of $M$ mod $\Om^G_M := N_{G_v}(M)$ acts on the set of $(M', \sigma)$ with $M'$ conjugate to $M$.

Let $\mc O$ be the set of orbits of this action. Each $[M, \sigma] \in \mc O$ 
can be given the structure of quotient of a product of circles by a finite group through the corresponding structure on $M^*/\Om^G_M$. If $[M, \sigma] \in \mc O$, define
\[
\check G_v[M, \sigma] := \bigsqcup_{(M, \sigma) \in [M, \sigma]} \{\text{irred. subquotients } \pi \text{ of } \mc I_{P_M}^G(\sigma)\},
\]
where $\mc I_{P_M}^G$ is the normalized parabolic induction through a parabolic $P_M$ with Levi factor $M$ (its irreducible subquotients do not depend on the choice of $P$).

Then
\[
\check G^\temp_v = \bigsqcup_{[M, \sigma] \in \mc O}\check G_v[M, \sigma].
\]
In addition, for any open $U \subseteq [M,\sigma]$, the set
\[
\bigsqcup_{(M, \sigma) \in U} \{\text{irred. subq. } \pi \text{ of } \mc I_{P_M}^G(\sigma)\}
\]
is open in the Fell topology (this can be seen for example by the characterization of the restriction of the Fell topology to the infinite-dimensional irreducible representations in $\check G_v$ in terms of matrix coefficients and the trace Paley-Weiner theorem). 

We now establish an analoguous decomposition for $\check G_v^{\st, \temp}$. If $\varphi$ is a tempered parameter of $G$, there exists a Levi $M$ such that $\varphi$ is the unique-up-to-conjugation pushforward of discrete parameter $\varphi_M$ for $M$. As part of the construction of LLC through the endoscopic classification, 
\begin{equation}\label{temperedLpacket}
\Pi_\varphi = \bigsqcup_{\sigma_M \in \Pi_{\varphi_M}} \{\text{irred. subq. } \pi \text{ of } \mc I_{P_M}^G(\sigma_M)\}
\end{equation}
(see the discussion before Proposition 6.6.1 in \cite{Art13} and the discussion in \S7.6 in \cite{Mok15}). 

Therefore, given a  Levi $M$ and discrete packet $\Pi_M$ of $M$, we can similarly define a stable $\mf l$-component $\check G^\st_v[M, \Pi_M]$ which also inherits a real manifold-quotient structure from $M^*/\Om^G_M$ and gives disjoint union over an analogous $\mc O^\st$:
\[
\check G^{\st, \temp}_v = \bigsqcup_{[M, \Pi_M] \in \mc O^\st } \check G^\st_v[M, \Pi_M].
\]
In this way the map $\rho : \check G_v^\temp \to \check G_v^{\st, \temp}$ is continuous. Note also that since $\mu^\pl$ is supported on all of $\check G_v^\temp$, its pushforward $\bar \rho_* \mu^\pl$ is supported on all of $\check G_v^{\st, \temp}$.

Our practical punchline for this discussion is that:
\begin{lem}\label{stablefouriertransformcontinuous}
Let $f_v$ be a test function on $G_v$. Then $\rho_* \wh f_v$ is continuous on $\check G_v^{\st, \temp}$. 
\end{lem}

\begin{proof}
It suffices to show this on each stable $\mf l$-component $\check G^\st_v [M, \Pi_M]$. For each~$\pi_M \in \Pi_M$,
\[
\chi \mapsto \tr_{\mc I_{P_M}^G(\pi_M \otimes \chi)} f_v
\]
is continuous on $\check G_v [M, \pi_M]$ by the trace Paley-Wiener theorem \cite{BDK86} (which actually gives much stronger regularity). However, by \eqref{temperedLpacket}, the function $\rho_* \wh f$ on $\check G_v^\st [M, \Pi_M]$ is just the sum of these functions over $\pi_M$. 
\end{proof}

\subsubsection{Spectral Interpretation of Main Term}
As a second input for the positivity test, we import some formalism from \cite[\S8.3]{Dal22} to interpret \eqref{mainterm} spectrally. For any group $H$, there is a central character map $\Om : \check H \to \check Z_H$. Given a test function $f$ on $H$ and a central character $\om$, we can define the conditional Plancherel expectation as the Radon-Nikodym derivative
\[
E^\pl(\wh f | \om) := \lf. \f{d\Om_*(\wh f \mu^\pl_H)}{d \mu^\pl_{Z_H}} \ri|_\om
\]
which exists and is necessarily continuous as a function on $\check Z_H$.

Now, consider $G \in \wtd{\mc E}_\sm(N)$ in either the self-dual or conjugate self-dual case. Let~$f = \prod_{v\nmid \infty} f_v$ be a test function on $G^\infty$ 
such that $f_v = \bar \1_{K_v}$, for $K_v$ a hyperspecial maximal compact, outside a finite set of places $S$. Then, as a special case of \cite[Prop. 8.3.5]{Dal22}, we can write
\begin{equation}\label{maintermspectral}
\sum_{\gamma \in Z_G(F)} \om^{-1}(\gamma) f(\gamma) =  C \sum_{\om_S \in \Om_\lb(K^{S, \infty})} E^\pl(\wh f_S | \om_S)
\end{equation}
for some constant $C > 0$ depending on various choices of measures and where
\[
\Om_\lb(K^{S, \infty}) = \{\om_S \in \check Z_{G_S} : \om_S|_{Z_G(F) \cap K^{S, \infty}} = \om_\lb|_{Z_G(F) \cap K^{S, \infty}}\}.
\]
In many cases, $Z_{G_S}$ is compact so $\check Z_{G_S}$ is discrete. Then, $E^\pl(\wh f_S | \om_S)$ simplifies as just the integral of $\wh f \, d\mu^\pl_{G_S}$ over some union of connected components of $\check G_S$.

\subsubsection{The Positivity Test}
Putting everything together lets us show:
\begin{prop}\label{spectralpositivity}
Let $G \in \wtd{\mc E}_\sm(N)$ in either the self-dual or conjugate self-dual case.
Let $f$ be a test function on $G(\A)$ equal to $\1_{K_v}$ outside of a finite set of places~$S$. Assume that for all tempered $L$-packets (or in the case of $\SO_{2n}$ the appropriate unions of such) $\Pi_S$ of $G_S$, 
\[
\tr_{\Pi_S}(f) \geq 0.
\]
Then for any character $\om$ of $Z_{G, \infty}$
\[
\sum_{\gamma \in Z_G(F)} \om^{-1}(\gamma) f(\gamma) \geq 0.
\]
Furthermore, the inequality is strict if there exists a tempered $L$-packet $\Pi_S$ such that~$\tr_{\Pi_S}(f) > 0$ and $\om_{\pi_S}|_{Z_G(F) \cap K^{S, \infty}} = \om|_{Z_G(F) \cap K^{S, \infty}}$ for all $\pi_S \in \Pi_S$.
\end{prop}

\begin{proof}
Consider any measurable $X \subseteq \check Z_{G_S}$ and $G_{S, X} := \Om^{-1}(X)$. Then $\check G_{S, X}$ is the exact preimage of a subset $\check G_{S, X}^\st$ of $\check G_S^\st$ since central characters are constant on $L$-packets by Lemma \ref{packettoom}. Then by the definition of $E^\pl$ and Theorem \ref{stablefourierinversion},
\begin{multline*}
\int_X E^\pl(\wh f_S | \om) \, d\mu^\pl_{Z_{G_S}}(\om) := \int_{\check G_{S, X}} \tr_{\pi_S}(f_S) \, d\mu^\pl_{G_S}(\pi_S) \\ 
= \int_{\check G_{S, X}^\st} \sum_{\pi_S \in \Pi_S} \tr_{\pi_S}(f_S) \, d\bar \rho_* \mu^\pl_{G_S}(\Pi_S)  
= \int_{\check G_{S, \om_\lb}^\st} \tr_{\Pi_S}(f_S) \, d\bar \rho_* \mu^\pl_{G_S}(\Pi_S) \geq 0. 
\end{multline*}
This forces $E^\pl(\wh f_S | \om) \geq 0$ for all $\om$ which, by \eqref{maintermspectral}, proves the non-negativity statement. 

For the strict positivity result, if such a packet $\Pi_S$ with appropriate central character $\om_0$ exists, then from Lemma \ref{stablefouriertransformcontinuous} and the fact that every tempered packet is in the support of $\bar \rho_* \mu^\pl_{G_S}$, we get that $\rho_*(\wh f_S) \, d\bar \rho_* \mu^\pl_{G_S}$ is supported at $\Pi_S$. By an argument similar to the previous paragraph, this implies that $E^\pl(\wh f_S | \om) \, d\mu^\pl_{Z_{G_S}}$ is supported at $\om = \om_0$. In particular, one summand $E^\pl(\wh f_S | \om_0)$ appearing in \eqref{maintermspectral} is strictly positive. 
\end{proof}

\subsection{Existence of Representations}\label{sec repexist}
Our strict positivity test requires the existence of local (conjugate) self-dual representations of $GL_N$ with arbitrary conductor and specified central character. We use the shorthand notation for conductors, see \S\ref{sec newvectors}.

\subsubsection{Self-dual case.} First note that the result fails for $N=1$. Self-dual characters of $F_v^\times$ must be valued in $\pm 1$, and so have conductor at most $1$ (though there always does exist a self-dual character of conductor 1.)

\begin{lem}\label{existenceGL2}
Fix a place $v$, an integer $k \geq 0$, and a character $\omega:F_v^\times \to \pm 1$ with conductor $\leq k$. Then there exists a self-dual tempered representation $\pi_v$ of $GL_2(F_v)$ of conductor $c(\pi_v) = k$ and central character $\eta$. Moreover,
\begin{itemize}
    \item If $\eta$ is trivial then $\pi_v$ can be chosen to be symplectic for all $k$, or orthogonal for $k$ even.
    \item If $\eta$ is nontrivial, then $\pi_v$ is necessarily of orthogonal type. If it is cuspidal, then its parameter factors through ${^L}SO_2^\eta$.
\end{itemize}
\end{lem}
\begin{proof}
Self-dual representations of $GL_2$ are characterized by the property that their central character is valued in $\pm 1$. 

For $k$ even and $\eta = 1$, we can take $\pi_v$ to be the principal series associated to characters~$\chi \boxtimes \chi^{-1}$ with $\chi$ unitary of conductor $k/2$. This is self-dual with trivial central character and conductor $k$. The images of the non-irreducible parameters associated to these representations preserve both a symmetric and skew-symmetric form and can arise as the localizations of automorphic representations of either orthogonal or symplectic-type.

For $k \leq 2$ and orthogonal type, we can consider the principal series induced from a pair of self-dual quadratic characters $\chi_1$, $\chi_2$ such that $\chi_1 \chi_2 = \eta$ and $c(\chi_1)+c(\chi_2)=k$. 

For $k=1$ and symplectic-type, consider the principal series induced from a pair of self-dual quadratic characters $\chi_1$, $\chi_2$ such that $\chi_1 \chi_2 = \eta$ and $c(\chi_1)+c(\chi_2)=1$ (in the orthogonal case).

Finally, for the remaining cases with $k\geq2$ we rely on the construction of supercuspidals in \cite{BH06}, from characters $\chi$ of $L_w^\times$, for a quadratic extension $L_w/F_v$ with ramification index $e$. In this construction, if $\chi$ has conductor $\fp_w^{\ell}$, the resulting representation $\pi_\chi$ has central character $\chi\mid_{F^\times_v}$ and conductor $\mf p_v^k$ for $k = \frac{2\ell}{e}+2$, see \cite[Theorems 20.2 and 25.2]{BH06}. This reduces the existence question to that of characters of suitably ramified $L^\times_w$ with a given restriction to $F_v^\times$ and conductor, which follows from standard results on characters of finite abelian groups.

Supercuspidal representations correspond to irreducible $L$-parameters, which factor through a non-abelian subgroup of ${^L}G_N$. This subgroup is $\Ld \Sp_2 = \Ld SL_2$ precisely when the central character $ \eta = \chi \mid_{F^\times}$ is trivial. Otherwise, the $L$-parameter is orthogonal and factors though $\Ld SO_2^\eta$.  
\end{proof}

\begin{lem}\label{existenceofcondN} 
Let $N > 2$ and let $G \in \wtd{\mc E}_\sm(N)$. Fix a place $v$ and a conductor~$k$, and, if $\wh G$ is orthogonal, a central character $\om_v = \pm 1$, that can only be $-1$ if $k > 0$. Assume further that if $G_v = \SO_N^{\eta_v}$, then $k \geq c(\eta_v)$.

Then there is a self-dual tempered representation $\pi_v$ of $G_{N,v}$ with:
\begin{itemize}
\item conductor $k$, 
\item $L$-parameter factoring through $\Ld G_v$, 
\item central character $\om_v$ on $G_v$.  
\end{itemize}
\end{lem}

\begin{proof}
Let $m \in \{1,2\}$ have the same parity as $N$ and consider the representation 
\[ 
\pi_v = \chi_1 \boxplus \cdots \boxplus \chi_{\lfloor \frac{N-1}{2}\rfloor} \boxplus \sigma_m \boxplus \chi^{-1}_{\lfloor \frac{N-1}{2}\rfloor} \boxplus \cdots \boxplus \chi^{-1}_1, 
\] 
for (not necessarily self-dual) characters $\chi_i$, and $\sigma_m$ a self-dual unitary representation of $GL_m$. Then $\pi_v$ is unitary with the same central character as $\sigma_m$ and
\[ 
c(\pi_v) = c(\sigma_m) + 2 \sum_i c(\chi_i), \qquad \eps(1/2,\pi_v) =\eps(1/2, \sigma_m)  \prod_i \chi_i(-1).
\]
In addition, by Lemma \ref{existenceGL2}, we can choose $\sigma_m$ and the $\chi_i$ so that $\pi_v$ is tempered, has arbitrary conductor, and arbitrary self-dual central character $\eta_v$. 

If $N$ is odd, this $\pi_v$ is necessarily orthogonal type and factors through $\Ld G = \Ld Sp_{N-1}^{\eta}$  if the central character of $\pi_v$ on $G_{N,v}$ is $\eta_v$. In addition, if $c > 0$ and $N > 2$, we can also choose the $\chi_i(-1)$ and therefore the root number. 
Since the root number determines the central character $\om_v$ on $G_v$ by Theorem \ref{thm cchareps}, $\pi_v$ can be chosen to satisfy all our desired conditions. 

If $N$ is even and $G = \SO_{N+1}$, we can choose $\sigma_m$ to symplectic as in Lemma~\ref{existenceGL2}, implying that $\pi_v$ is too (as in the discussion on \cite[p. 34]{Art13}). Following Lemma~\ref{existenceGL2}, the representation $\sigma_m$ can be chosen so that $\pi_v$ satisfies all our desired conditions. 

Finally, for $G = \SO_N^\eta$, we choose $\sigma_m$ to have parameter factoring through $\Ld \SO_2^\eta$ as in Lemma~\ref{existenceGL2}. Then, the parameter of $\pi_v$ factors through $\Ld \SO_N^\eta$. We can also choose the $\chi_i(-1)$ appropriately to determine the root number of $\pi_v$ and therefore its central character. 
\end{proof}

\subsubsection{Conjugate self-dual case}\label{sec existenceconj}

In the following lemmas, we fix a non-split place $v$ of $F$ and a place $w$ of $E$ above $v$. We use local indexing conventions for conductors (where ramified conductors are integers instead of half-integers, see \ref{sec conductorsconj}). Let
\begin{equation} \label{eq ram measure}
j_v := \min\{1/2, \{j : 1 + \mf p_w^j \subseteq N_{E/F}(E^\times) \}\}. 
\end{equation}
The $1/2$ is for indexing in order that the term $2j_v-1$, which will appear a lot in this section, satisfies
\[
2j_v - 1 = \begin{cases}
0 & v \text{ unramified} \\
1 & v \text{ tame} \\
\geq 1 & v \text{ wild.}
\end{cases}
\]

\begin{lem}\label{lem embedding GL1} Let $\chi: E_w \to \BC$ be a conjugate self-dual character, so that restriction~$\chi\mid_{F_v^\times} = \kappa \in \pm 1$ is either trivial or the quadratic character corresponding to the extension $E_w$. Then the $L$-parameter of $\chi$ factors through the endoscopic datum $U^\kappa_1$. 
\end{lem}
\begin{proof}
    We write down the explicit embeddings, following \cite[\S 2.1]{Mok15}. Using class field theory, we view the character $\chi$ as a parameter \[ \varphi_\chi: \Gamma_{E_w} \to GL_1 \times \Gamma_{E_w}. \] 
    Recall that \[ \Ld G(1) = (GL_1 \times GL_1) \rtimes \Gamma_{F_v} \] where the action of $\Gamma_{F_v}$ is through the quotient $\Gamma_{E_w/F_v}$ and the nontrivial element swaps the two factors.  Self-duality enables us to extend $\varphi_\chi$ to a parameter \[ \varphi_\chi : \Gamma_{F_v} \to \Ld G(1)  \] as follows. Fixing a representative $z_c$ of $\Gamma_{F_v} \setminus \Gamma_{E_w}$, we let: 
    \begin{align*}
        \varphi_\chi(\sigma) &= (\chi(\sigma),\chi^{-1}(\sigma)) \rtimes \sigma \quad \sigma \in \Gamma_{E_w}\\ 
         \varphi_\chi(z_c) &= (\alpha\chi(z_c^2),\alpha^{-1}) \rtimes z_c.
        \end{align*}
In this way the value of $\varphi_\chi(z_c)$ is determined up to choice of a scalar $\alpha$ (or, equivalently, up to a choice of representative $z_c$). 
Note that since~$z_c^2$ is invariant under conjugation by~$z_c$, we must have $\chi(z_c^2) = \chi(z_c^2)^{-1}$, i.e. $\chi(z_c^2) \in \{\pm 1\}$. Additionally, the condition that $z_c \notin \Gamma_{E_w}$ implies by local class field theory that~$\chi(z_c^2)=\kappa$. 

Next, recall that 
\[ 
\Ld U(1) = GL_1 \rtimes \Gamma_{F_v}
\] 
where again, the action is via the quotient $\Gamma_{E_w/F_v}$ and conjugation by a nontrivial element sends $g \mapsto g^{-1}.$   Each embedding of $\Ld U(1)$ into $\Ld G(1)$ involves a choice: for~$\kappa \in {\pm 1}$ we fix an auxiliary character  $\eta$ of $E_w$  whose restriction of $F_v^\times$ is $\kappa$ and we have 
\begin{align*}
\xi_\kappa : \Ld U(1) &\to \Ld G(1) \\ 
\xi_\kappa(g) &= (g,g^{-1}) \rtimes 1, \quad g \in GL_1(\BC) \\ 
\xi_{\kappa}(\sigma) &= (\eta(\sigma), \eta(\sigma^{-1})) \rtimes \sigma, \quad \sigma \in \Gamma_{E_w} \\ 
\xi_\kappa(z_c) &= (\kappa, 1) \rtimes z_c. 
\end{align*}
For the parameter $\varphi_\chi$ to factor through $\xi_\kappa$, the choice of image of $z_c$ must be compatible with the embedding, i.e. we must have $\kappa = \chi(z_c^2)$, which gives the result. 
\end{proof}

\begin{lem}\label{lem:characterexistenceconj}
 Fix a conductor $k \in \BZ_{\geq 0}$.  Then, for each $G \in \wtd{ \mathcal E}_\sm(1)$, there exists a conjugate self-dual character of $E_w^\times$ of conductor $k$ such that the associated $L$-parameter factors through $\Ld G$ if and only if one the following hold: 
\begin{itemize}
    \item when $\kappa = 1$:
    \begin{itemize}
        \item $E_w$ is unramified, 
        \item $E_w$ is ramified and $k$ is even,
    \end{itemize}
    \item when $\kappa = -1$:
    \begin{itemize}
        \item $E_w$ is unramified
        \item $E_w$ is ramified 
        and either:
        \begin{itemize}
            \item $k = 2j_v-1$,
            \item $k \geq 2j_v$ is even,
        \end{itemize}
    \end{itemize}
\end{itemize}
where $j_v$ is as in \eqref{eq ram measure}.
\end{lem}

\begin{proof}
    The condition of self-duality for characters is $\chi(\overline g) = \chi(g^{-1})$, or equivalently, that $\chi |_{F_v^\times} \in \pm1$, where we abuse notation and denote by $-1$ 
the unique nontrival character of $F^\times/N(E^\times_w)$ where $N = N_{E_w/F_v}$ is the norm map. The sign $\pm 1$ determines which element of $\mc E_\sm(1)$ the parameter $\chi$ factors through, as in Lemma \ref{lem embedding GL1}.

Next, denote $U^0 = \mc O_{E_w}^\times$ and $U^k = 1+\mathfrak p_w^k$. For indexing convenience, denote~$U^{-1}$ to be some set strictly bigger than $E^\times_w$. We now consider necessary and sufficient conditions for the existence of a character of sign $\kappa$ and conductor $k$:

\noindent \underline{$\kappa = 1$}

There exists a character of sign $\kappa=1$ and conductor $k$ if and only if
\begin{equation} \label{eq obs kappa = 1}
        U^{k-1} \not\subset F^\times U^k.
\end{equation}
For $k=0$, this condition is empty.  For $k>0$, we have  
\[ 
\text{\eqref{eq obs kappa = 1}} \iff U^{k-1} \not\subset (F^\times \cap U^{k-1})U^k \iff (U^{k-1} \cap F)/(U^k \cap F) \not \simeq U^{k-1}/U^k
\] 
and, by comparing cardinalities, the latter holds if any only if \begin{itemize}
        \item $E_w$ is unramified, or
        \item  $E_w$ is ramified and $k$ is even.
\end{itemize}
    
\noindent \underline{$\kappa = -1$}:

To determine a necessary and sufficient condition for existence here, we use that given two characters $\chi_1$, $\chi_2$ with conductors $c(\chi_1) > c(\chi_2)$, then $c(\chi_1 \chi_2) = c(\chi_1)$. Furthermore, the parities satisfy $\kappa(\chi_1 \chi_2) = \kappa(\chi_1) \kappa(\chi_2)$. In particular, if $k_0$ is the smallest possible conductor for which there exists $\chi$ with $\kappa = -1$, then for all $k > k_0$, conductor $k$ is possible with $\kappa = -1$ if and only if it is with $\kappa = 1$. Therefore, it suffices to find this $k_0$.

If $E_w$ is unramified, then any character $\chi$ with $\chi(\mc O_E^\times)\equiv 1$ and such that~$\chi$ takes some uniformizer to $-1$ is unramified and has $\kappa = -1$. Therefore, $k_0 = 0$.  

If $E_w$ is ramified, note the existence of a character $\chi$ with $\kappa = -1$  requires that 
\[
F^\times \not\subset N(E^\times) U^k \iff F^\times \cap U^k \subseteq N(E^\times).
\]
This forces $k \geq 2j -1$. To construct $\chi$ with $k = 2j-1$, it suffices to further note that 
\[
1 + \mf p_v^{j-1} \notin N(E^\times) \implies U^{2j - 2} \notin N(E^\times) U^{2j-1}
\]
Therefore, $k_0 = 2j-1$. 
\end{proof}

The full list of possible pairs $(k,l)$ of conductor $c(\pi_v) = k$ and central character conductor $c(\om_{\pi_v}) = l$ is very hard to determine and likely does not have a nice description, particularly in the wildly ramified case. We settle for providing an easy-to-describe subset that covers most relevant possibilities outside of some wild ramification issues.

\begin{lem}\label{existenceofcondNconj}
Let $N \geq 4$ be even, and let $G  = U_N^+ \in \wtd{\mc E}_\sm(N)$ in the conjugate self-dual case. Fix a non-split place $v$ and let~$e_v := e(E_w:F_v)$ be the ramification index of $E_w/F_v$. Fix $k \geq 0$ a conductor, and fix a central character $\om_v$ on $G_v$ transferring to~$\om'_v$ on $G_{N,v}$ with conductor $l \leq k$. 

Then there is a conjugate self-dual tempered representation $\pi_v$ of $G_{N,v}$ with conductor $c(\pi_v) = k$ such that its corresponding $L$-parameter $\psi_{\pi_v}$ factors through~$\Ld G_v$ and has central character $\om_v$ on $G_v$ for the following cases:
\begin{itemize}
\item $l = \max\{k - (2j_v-1),0\}$ when $e_v = 1$ and the largest even number less than this otherwise and:
\begin{itemize}
    \item $v$ unramified: all possible $k$,
    \item $v$ tamely ramified: all possible $k$,
    \item $v$ wildly ramified: $k \leq N/2$ or $k > 4j_v-2$. 
\end{itemize}
\item $l = 0$:
\begin{itemize}
    \item $v$ unramified: all possible $k$,
    \item $v$ tamely ramified: all possible $k$,
    \item $v$ wildly ramified: $k \leq N/2$ or $k \geq 8j_v - 4$
\end{itemize}
\end{itemize}
\end{lem}

\begin{proof}
First, note that $k=l=0$ is always achievable by looking at a self-dual Satake parameter. 

Otherwise, consider $\pi_v$ to be the Langlands quotient of the parabolic induction of blocks
\[
\rho_1 \boxtimes \cdots \boxtimes \rho_t
\]
where each $\rho_i$ is a conjugate-symplectic representation of some even $G_{N_i,v}$. It suffices to consider the case where all $\rho_i$ are trivial except for one that is a principal series with central character $\om'_v$ and possibly some number that are dimension-2 Steinbergs.

Since $c(\om_{\pi_v}) = \sum_i c(\om_{\rho_i})$ and $c(\pi_v) = \sum_i c(\rho_i)$, it suffices to describe the possible values of conductor $c(\rho_i) = k$ and central character conductor $c(\om_{\rho_i}) = l$ that can be achieved with our list of possible $\rho_i$ and choose appropriately amongst these options. 

\noindent \underline{Rank $2$ Principal Series:}

Let $\chi_0$ be a choice of conjugate-symplectic character of $E_v^\times$ with minimum conductor as in the proof of Lemma \ref{lem:characterexistenceconj}. Choose conjugate orthogonal~$\nu_1 \nu_2 = \om'_v$ with conductors $k_1, k_2$ respectively. Consider  $\rho_i$ the parabolic induction of $\chi_0 \nu_1 \boxtimes \chi_0^{-1} \nu_2$ which is conjugate-symplectic since the components are (see the discussion after Definition 2.4.7 in \cite{Mok15}).

There are two possible cases for conductors:
\begin{itemize}
    \item $k_1 = k_2 \geq l$: Then $k = \max\{2k_1, 4j_v-2\}$. 
    \item $k_1 = l$ and $k_1 > k_2$: Then $k = \max\{l, 2j-1\} + \max\{k_2, 2j_v-1\}$. 
\end{itemize}
Note in the above computations that in the ramified cases when $j_v \neq 0$, the $k_i$ are always even and therefore not equal to $2j_v-1$.  Therefore, we never need to worry about the case where the conductor of a product with $\chi_0$ is not the maximum of its factors.  

This in particular realizes the following possible pairs $(k,l)$ (with $\om'_v$ arbitrary of conductor $l$):
\begin{itemize}
    \item $v$ unramified: $k$ even and $l=0$ or $k$ arbitrary and $k=l$
    \item $v$ tamely ramified: $k \geq 4$ a multiple of $4$ and $l=0$, $k = 2$ and $l = 0$, or $k > 2$ odd and $l = k-1$. 
    \item $v$ wildly ramified: $k \geq 4j_v$ a multiple of $4$ and $l=0$, $k = 4j_v-2$ and $l = 0$, or $k > 4j_v-2$ odd and $l = k-(2j_v-1)$. 
\end{itemize}

\noindent \underline{Rank-$4$ Principal Series:}

In the ramified cases when $l=0$, we also need to consider dimension-$4$ principal series. Choose $\nu_1\nu_2\nu_3 = 1$ and $\rho_i$ the parabolic induction of
\[
\chi_0 \nu_1 \boxtimes \chi_0 \nu_2 \boxtimes \chi_0^{-1} \nu_3 \boxtimes \chi_0^{-1}. 
\]
We consider the case when $c(\nu_0) = c(\nu_1) = k_2 \geq c(\nu_2) = k_2$. Then $\rho_i$ has conductor
\[
k = \max\{2k_1, 4j_v - 2\} + \max\{k_2, 2j_v - 1\} + (2j_v - 1)
\]
which can be any integer larger than $8j_v - 4$.

\noindent \underline{Steinberg}

If $\rho_i$ is the Steinberg, then it gives $k = 1$ and $l = 0$.

\noindent \underline{Even-Rank Trivial}

If $\rho_i$ is trivial with even rank, then it is both conjugate-orthogonal and conjugate-symplectic. This gives $k=0$ and $l=0$.  
\end{proof}

\begin{note}
Constructing representations of odd conductor when $N$ is even crucially depended on choosing $G = U^+_N$ through the possible choice of characters being induced. This is a good consistency check for Proposition~\ref{cosettransfervanishingconj}: when $G = U^-_N$, the function $\wtd E^\infty_{v,k} = \wtd C^\infty_{v,k}$ transfers to $0$ when $k$ is odd. Given Proposition~\ref{spectralpositivity}, this is incompatible with the existence of odd-conductor representations and sign $\kappa = -1$. 
\end{note}

At split places things are simpler:
\begin{lem}\label{splitexistenceofcondNconj}
Let $N \geq 3$, and let $G  = U_N^\pm \in \wtd{\mc E}_\sm(N)$ in the conjugate self-dual case. Fix a \emph{split} place $v$ and a conductor $k$. Fix a central character $\om_v$ on $G_v$ transferring to $\om'_v$ on $G_{N,v}$ with conductor $l \leq k$. 

Then there is a conjugate self-dual tempered representation $\pi_v$ of $G_{N,v}$ with conductor $c(\pi_v) = k$ such that its corresponding $L$-parameter $\psi_{\pi_v}$ factors through~$\Ld G_v$ and has central character $\om_v$ on $G_v$. 
\end{lem}

\begin{proof}
Any such $\pi_v$ on $G_{N,v} \simeq \GL_N(F_v) \times \GL_N(F_v)$ is of the form $\pi_0 \boxtimes (\pi_0)^\vee$. In the split place, the $\kappa = \pm$ embeddings are the same, so this has parameter factoring through both. In our conventions, the conductor of $\pi_v$ is that of $\pi_0$, so it suffices to find $\pi_0$ of conductor $k$ and central character $\om'_v = \om_v \circ \phi$. This can be done using principal series. 
\end{proof}

\subsection{Application to \lm{$(\wtd C^\infty_\mf n)^G$}}\label{sec C(1)}
Now we are finally ready to show when our main term \eqref{mainterm} for $f^\infty = (\wtd C^\infty_\mf n)^G$ is positive.

\subsubsection{Self-Dual Case}
Since it doesn't take that much extra work in the self-dual case, we prove a statement for all $G \in \wtd {\mc E}_\sm(N)$ instead of just $G = \SO_{N+1}$ when~$N$ is even (though note that for this eventual case of interest, $Z_G$ is trivial so the statement can be made simpler):

\begin{thm}\label{Ctransferpositivity}
Let $\wtd G_N$ in the self-dual case and $\wtd C^\infty_\mf n$ the test function on $\wtd G_N^\infty$ defined in \ref{Cdef}. Choose $G \in \wtd{\mc E}_\sm(N)$ and let $\om'_v$ be the corresponding central characters on $Z_{G_N,v}$ common to all transfers from $G_v$ (as in the discussion after Lemma \ref{cchartransfer}). 

Fix a central character $\om_\infty$ on $Z_{G_\infty} \subseteq \{\pm 1\}$ and assume:
\begin{itemize}
    \item For all $v$, $v(\mf n) \geq c(\om'_v)$,
    \item If $\om_\infty \neq 1$, $\mf n \neq 1$.
\end{itemize}
Then,
\[
\sum_{\gamma \in Z_G(F)} \om_\infty^{-1}(\gamma) (\wtd C^\infty_\mf n)^G(\gamma) > 0.
\]
\end{thm}

\begin{proof}
We apply Proposition \ref{spectralpositivity} to the set $S$ containing all primes on which $\mf n$ is supported and all finite places at which $\om'_v$ is nontrivial. The input for non-negativity follows from the defining property Corollary \ref{localC} of $\wtd C^\infty_\mf n$ plugged into the endoscopic character identities \cite[Thm 2.2.1(a)]{Art13} or \cite[Thm 3.2.1(a)]{Mok15}. 

For strict positivity, if $\om_\infty \neq 1$, choose $v$ such that $v(\mf n) > 0$ and set $\om_v$ to be the non-trivial character on $Z_{G_v}$. Set all other $\om_{v'} = 1$. Then, Lemma \ref{existenceofcondN} constructs representations $\pi_v$ on each $G_{N,v}$ of conductor $\mf p_v^{v(\mf n)}$, central character $\om'_v$, and that corresponds to a packet on $G_v$ with central character $\om_v$.

Then, if $\pi_S = \prod_{v \in S} \pi_v$ as a representation of $G_{N,S}$, it has conductor $\mf n$ and its corresponding $A$-packet $\Pi_S$ on $G_S$ is made of $\pi_S$ with central character satisfying $\om_{\pi_S}|_{Z_G(F)} = \om_{\pi_\infty}|_{Z_G(F)}$. Therefore, $\Pi_S$ is the necessary input to Proposition \ref{spectralpositivity} for strict positivity (again by the endoscopic character identities). 
\end{proof}

\subsubsection{Conjugate Self-dual case}
In the conjugate self-dual case, globalizing central characters is more complicated. For $v$ a place of $F$, let $e_v$ be its ramification degree in $E$. For $\mf n$ a conductor, let $\mc O_E(\mf n) \subseteq \mc O_E^\times$ be the units that are $1 \pmod{\mf n}$. 

We return to global indexing for conductors (where ramified exponents are half-integers).

\begin{lem}\label{lem ccharglobconj}
Let $G = U^+_N \in \wtd{\mc E}_\sm(N)$. Then for each finite place $v$ of $F$ there is $b_v \in \Z_{\geq -2}$ depending only on $E/F$ such that for all 
\begin{itemize}
    \item $\om_\infty$ a character on $Z_{G_\infty}$ that transfers to $\om'_\infty$ on $Z_{G_N, \infty}$,
    \item conductor $\mf n$ such that
    \begin{itemize}
        \item[(i)] $\mf n$ is an ideal of $\mc O_F$ (i.e. $v(\mf n)$ is an integer for all ramified $v$),
        \item[(ii)] $v(\mf n) \geq b_v + 2/e_v$ for one place or $v(\mf n) \geq b_v + 1/e_v$ for at least two places,
    \end{itemize}
\end{itemize}
there exists a global $\om'$ on $Z_{G_N}(\A_F)$ with infinite component $\om'_\infty$ and such that ${\om'}^\infty$ has conductor exactly $\mf n$. 
\end{lem}

\begin{proof}
Let $\phi(x)=\frac{x}{\bar{x}}$ be the map first defined in \S \ref{sec E(1)conj}. Let $\Xi$ be the roots of unity in $E$ and for each finite place $v$ of $F$, consider the smallest $x \in \Z_{\geq 0}$ such that $\Xi \cap \phi(1 + \mf p_v^x) = 1$. Set $b_v = x$ if $x > 0$ and~$b_v = -2$ if $x = 0$. Note that we are not allowing half-integral $b_v$ when $v$ is ramified. 

The $b_v$ are chosen exactly so that if $\mf n$ is a conductor satisfying conditions (i)-(ii) and $\mf n'$ is a proper divisor of $\mf n$, then there is intermediate proper divisor $\mf n''$ such that $\mf n' | \mf n''$ and $\Xi \cap \phi(\mc O_E(\mf n'')) = 1$. The integrality condition guarantees that there is at least one conjugate-orthogonal character of conductor $\mf n$ by Lemma \ref{lem:characterexistenceconj}, so in addition $\phi(\mc O_E(\mf n'')) \neq \phi(\mc O_E(\mf n))$. 

Next, $Z_G(F) \cap Z_G(\wh{\mc O}_F) = \Xi$ since it consists of integral elements of $E^\times$ that have complex norm one in all their embeddings. It follows that for any compact open $K^\infty \subseteq Z_G(\wh{\mc O}_F)$, 
\[
Z_G(F) \bs Z_G(\A_F) / K^\infty = Z_G(F^\infty)/(\Xi \cap K^\infty) \times Z_G(\A^\infty_F)/Z_G(F)K^\infty.
\]
If $\mf n''$ is a proper divisor of $\mf n$ such that $\Xi \cap \phi(\mc O_E(\mf n'')) = 1$, and $\phi(\mc O_E(\mf n'')) \neq \phi(\mc O_E(\mf n))$, note that:
\begin{multline*}
Z_G(F) \bs Z_G(\A_F) / \phi(\mc O_E(\mf n)) \neq Z_G(F) \bs Z_G(\A_F) / \phi(\mc O_E(\mf n'')) \implies \\
Z_G(F_\infty) \times Z_G(\A^\infty_F)/Z_G(F)\phi(\mc O_E(\mf n)) \neq Z_G(F_\infty) \times Z_G(\A^\infty_F)/Z_G(F)\phi(\mc O_E(\mf n''))  \\
\implies  Z_G(\A^\infty_F)/Z_G(F)\phi(\mc O_E(\mf n)) \neq Z_G(\A^\infty_F)/Z_G(F)\phi(\mc O_E(\mf n'')).
\end{multline*}
Therefore, the properties (i)-(ii) of $\mf n$ guarantees that we can globalize $\om_\infty$ to a character~$\om$ of $Z_G(F) \bs Z_G(\A_F) / \phi(\mc O_E(\mf n))$ that doesn't descend to $Z_G(F) \bs Z_G(\A_F) / \phi(\mc O_E(\mf n'))$ for any proper divisor $\mf n' | \mf n$. 

Setting $\om' = \om \circ \phi$ then suffices by Lemma \ref{U+cchar}. 
\end{proof}

\begin{note}
In many cases, $\Xi \cap (1 + \mf D_{E/F}) = 0$ already, so the condition on $\mf n$ in Lemma~\ref{lem ccharglobconj} is trivial (i.e. all $b_v$ can be chosen to be $-2$). 
\end{note}

\begin{thm}\label{Ctransferpositivityconj}
Let $N \geq 4$ be even, $G = U^+_N \in \wtd{\mc E}_\sm(N)$, and $\wtd C^\infty_\mf n$ the test function on $\wtd G^\infty_N$ defined in \ref{Cdef}. Let $\om_\infty$ be a central character of $G_\infty$. Choose $b_v$ as in Lemma \ref{lem ccharglobconj} and let $\mf n$ be a conductor such that 
\begin{itemize}
    \item For all wildly ramified $v$, $v(\mf n) \leq N/4$ or $v(\mf n) > 2j_v -1$,
    \item There is at least one $v$ such that $\max\{\lfloor v(\mf n) - (j_v - 1/2) \rfloor, 0\} \geq b_v + 2/e_v$ or there are at least two $v$ such that $\max\{\lfloor v(\mf n) - (j_v - 1/2) \rfloor, 0\} \geq b_v + 1/e_v$.
\end{itemize}
Then,
\[
\sum_{\gamma \in Z_G(F)} \om_\infty^{-1}(\gamma) (\wtd C^\infty_\mf n)^G(\gamma) > 0.
\]
\end{thm}

\begin{proof}
Define the modified conductor 
\[
\mf n_c = \prod_v \mf p_v^{\max\{\lfloor v(\mf n) - (j_v - 1/2) \rfloor, 0\}}.
\]
By Lemma \ref{lem ccharglobconj}, we can find a character $\om'$ on $Z_{G_N}(\A_F)$ such that $\om'_\infty$ transfers to $\om_\infty^{-1}$ and $\om'$ has conductor $\mf n_c$. Then, our 
assumptions on $\mf n$ let us proceed as in Theorem \ref{Ctransferpositivity}: multiplying together the outputs of Lemmas \ref{existenceofcondNconj} and \ref{splitexistenceofcondNconj} (using the $l \neq 0$ option for inputs in Lemma~\ref{existenceofcondNconj}) gives a representation $\pi^\infty$ of $G_N(\A^\infty)$ with central character $(\om')^\infty$ and conductor $\mf n$. 

The argument then follows exactly as Theorem~\ref{Ctransferpositivity}. 
\end{proof}

We also prove a version for trivial $\om_\infty$:
\begin{thm}\label{Ctransferpositivityconjzero}
Let $N \geq 4$ be even, $G = U^+_N \in \wtd{\mc E}_\sm(N)$, and $\wtd C^\infty_\mf n$ the test function on $\wtd G^\infty_N$ defined in \ref{Cdef}. Choose the conductor $\mf n$ such that: 
\begin{itemize}
    \item For all wildly ramified $v$, $v(\mf n) \leq N/4$ or $v(\mf n) > 4j_v -2$.
\end{itemize}
Then,
\[
\sum_{\gamma \in Z_G(F)} (\wtd C^\infty_\mf n)^G(\gamma) > 0.
\]
\end{thm}

\begin{proof}
This is the same argument as Theorem \ref{Ctransferpositivity} but instead of globalizing $\om_\infty$ using Lemma \ref{lem ccharglobconj}, we globalize trivial $\om_\infty$ to the trivial character and use the $l=0$ option for inputs in Lemma \ref{existenceofcondNconj}. 
\end{proof}

\begin{note}\label{note conductorsconj}
A more careful analysis of representations in Lemma \ref{existenceofcondNconj} could increase the generality of wildly ramified factors allowed in $\mf n$. Furthermore, by inspection of the proof of Lemma~\ref{existenceofcondNconj}, there are similarly results for $N=2$, but with some extra congruence conditions on the exponents of ramified primes dividing $\mf n$. 
\end{note}

\section{Final Result}
We now discuss implications for equidistribution of epsilon factors:
\subsection{General Setup}
We first describe sets of automorphic representations. 
\begin{dfn}
In either the self-dual or conjugate self-dual case, let~$N > 0$, and consider $\mf n$ a conductor  for $F$ and $\lb$ a (conjugate) self-dual infinitesimal character for $G_{N, \infty}$. Choose $\star$ to be either (conjugate) orthogonal or (conjugate) symplectic---denoted by ``$\orth$'' or ``$\symp$'' respectively. 

Define $\mc{AR}_N^\star(\lb, \mf n)$ to be the set of self-dual, type-$\star$, cuspidal automorphic representations $\pi$ of $\GL_N$ of conductor $\mf n$ and with infinitesimal character $\lb$ at infinity.
\end{dfn}

The most basic question we can ask about distribution of root numbers $\eps(1/2, \pi)$ in $\mc{AR}_{N,\star}(\lb, \mf n)$, is to understand asymptotics of
\[
\f1{|\mc{AR}_N^\star(\lb, \mf n)|} \sum_{\pi \in \mc{AR}_{N, \star}(\lb, \mf n)} \eps(1/2, \pi).
\]
In full generality, we also want to understand weighted distributions---if $\wtd f_S$~is a test function at some set of places $S$ coprime to $\mf n$, we can try to compute the asymptotics of
\[
\f{\sum_{\pi \in \mc{AR}_N^\star(\lb, \mf n)} \tr_{\wtd \pi_S}(\td f_S) \eps(1/2, \pi)}{\sum_{\pi \in \mc{AR}_N^\star(\lb, \mf n)} \tr_{\wtd \pi_S}(\wtd f_S)},
\]
where we recall $\wtd \pi$ is the canonical extension of $\pi$ to $\wtd G_N$.

We first describe the class of functions $f_S$ which our methods allow us to weigh the count by:
\begin{dfn}\label{def fullweighting}
In either the self-dual or conjugate self-dual case, let $S$ be a finite set of finite places of $F$ and $G \in \wtd{\mc E}_\el(N)$. Then a test function $\td f_S \in \ms H(\wtd G_{N,S})$ is a full weighting for $G_S$ if the Fourier transform of $(\td f_S)^G|_{Z_{G_S}}$ is strictly positive on all unramified characters. 

Note that if $(\td f_S)^G|_{Z_{G_S}}$ is supported in $K_S \cap Z_{G_S}$ (in particular, if $Z_{G_S}$ is compact), this simplifies to the condition that $(\td f_S)^G(1) > 0$. 
\end{dfn}
Test functions that are not weightings produce sums over automorphic representations that cancel out except for an asymptotically trivial subset of $\mc{AR}_{\star}(\lb, \mf n)$, making them much harder to study. 

Some examples of full weightings for $G \in \wtd{\mc E}_\el(N)$:
\begin{itemize}
    \item Indicators of hyperspecials $\wtd f_S = \1_{K_S \rtimes \theta}$ give unweighted counts. 
    \item For any test function $f_S \in \ms H(G_S)$, we can find $\td f_S \in \ms H(\wtd G_N)$ such that $(\td f_S)^G = f_S$ by \cite[Prop 3.1.1(b)]{Mok15} or \cite[Prop. 2.1.1]{Art13} (the argument for Mok's extension part (b) also works in Arthur's case). 
    \item If $\wtd f_S$ satisfies that $\tr_{\wtd \pi_S}(f_S) \geq 0$ for all (conjugate) self-dual $\pi_S$ and there are tempered $A$-parameters $\psi_S$ for $G$ with general enough central character such that $\tr_{\wtd \pi_{\psi_S}}(f_S) > 0$, then $f_S$ will be a full weighting similarly to Proposition \ref{spectralpositivity}. 
\end{itemize}

\subsection{Self-Dual Case}
In the self-dual case, $\eps(1/2, \pi) = 1$ unless $N$ is even and $\pi$ is symplectic-type/comes from $G = \SO_{N+1}$. We focus on this case:
\begin{thm}\label{mainthm}
Let $N > 0$ be even, and consider $\mf n$ a conductor  for $F$ and $\lb$ self-dual, symplectic-type, integral infinitesimal character for $G_{N, \infty}$ as in \S\ref{integralinfchars}.
Choose a  finite set of finite places $S$ coprime to $\mf n$ and a test function
\[
\wtd f_S := \1_{K_S \rtimes \theta} \star f_S \in \ms H^\ur(\wtd G_N(F_S))^{\leq \kappa},
\]
such that $\wtd f_S$ is a full weighting for $G = \SO_{N+1}$. For any self-dual automorphic representation $\pi$ of $\GL_N$, define coefficients
\[
w(\pi) := \tr_{\wtd \pi_S}(\wtd f_S).
\]
with respect to the canonical extension of $\pi$ to $\wtd G_N$. Then, if there is place $v$ such that either
\begin{itemize}
    \item $v(\mf n) > N$,
    \item $v(\mf n)$ is odd,
\end{itemize}
we have
\[
\f{\sum_{\pi \in \mc{AR}_N^\symp(\lb, \mf n)} w(\pi) \eps(1/2, \pi)}{\sum_{\pi \in \mc{AR}_N^\symp(\lb, \mf n)} w(\pi)} = O_{\mf n}(m(\lb)^{-C} q_S^{A + B \kappa})
\]
for some constants $A,B,C$ with $C > 1$. Otherwise,
\begin{multline*}
 \f{\sum_{\pi \in \mc{AR}_N^\symp(\lb, \mf n)} w(\pi) \eps(1/2, \pi)}{\sum_{\pi \in \mc{AR}_N^\symp(\lb, \mf n)} w(\pi)} \\=  \eps(1/2, \varphi_\lb) (-1)^{\sum_v v(\mf n)/2}R + O_{\mf n}(m(\lb)^{-C} q_S^{A + B \kappa})
\end{multline*}
for some $R > 0$ and where $\varphi_\lb$ is the parameter with infinitesimal character $\lb$ and trivial central character (i.e. the one coming from the discrete $L$-packet with infinitesimal character $\lb$ on $G$).
\end{thm}

\begin{proof}
Let $\Sigma_\eta$ be the shape containing only cuspidal parameters of symplectic-type---i.e. $H(\Sigma_\eta) = \SO_{2N + 1} = G$. By Theorem \ref{globalepsilontrace} and Proposition \ref{stablecount},
\begin{multline*}
\sum_{\pi \in \mc{AR}_N^\symp(\lb, \mf n)} w(\pi) \eps(1/2, \pi) = \eps(1/2, \varphi_\lb)\sum_{\psi \in \Sigma_{\eta,\lb}} \tr_{\wtd \pi_\psi^\infty}(\wtd E^\infty_\mf n \star f_S)\\
= \eps(1/2, \varphi_\lb)S^G_{\Sigma_\eta}(\EP_\lb (\wtd E^\infty_\mf n \star f_S)^G),
\end{multline*}
after noting that $S$ coprime to $\mf n$ means that $\wtd E_{\mf n, S} = \1_{K_S \rtimes \theta}$.

By similar arguments to \cite[\S5.4]{Dal22}, we have $(\wtd f_S)^G \in \ms H^\ur(G)^{\leq D\kappa}$ for some constant $D$ depending only on $N$ and $G$. Therefore, by the bound of Proposition~\ref{fullnewshapebound} with our $S$:
\begin{multline}\label{eq sd final asymp w}
(\dim \lb)^{-1}S^G_{\Sigma_\eta}(\EP_\lb (\wtd E^\infty_\mf n \star f_S)^G) \\ = \Lambda(G, \om_\lb, (\wtd E^\infty_\mf n \star f_S)^G)  + O_{\mf n}(m(\lb)^{-C} q_S^{A + B \kappa}(\wtd f_S)^G(1)),
\end{multline}
where $\Lambda$ is from \eqref{mainterm} and for some $A,B,C$ with $C \geq 1$. By Corollary \ref{Etransfervanishing}, we can compute
\begin{multline*}
\Lambda(G, \om_\lb, (\wtd E^\infty_\mf n \star f_S)^G)  = (\wtd E^\infty_\mf n )^G (1) (\wtd f_S)^G(1) \\= (\wtd f_S)^G(1)\begin{cases}
0 & \mf n \text{ satisfies conditions} \\
M(-1)^{\sum_v v(\mf n)/2} & \text{else}
\end{cases}
\end{multline*}
for some constant $M > 0$. 

By a similar argument, now using Corollary \ref{localC}, 
\[
\sum_{\pi \in \mc{AR}_N^\symp(\lb, \mf n)} w(\pi) = S^G_{\Sigma_\eta}(\EP_\lb (\wtd C^\infty_\mf n \star f_S)^G)
\]
and
\begin{multline} \label{eq sd final asymp nw}
(\dim \lb)^{-1}S^G_{\Sigma_\eta}(\EP_\lb (\wtd C^\infty_\mf n \star f_S)^G) \\ = (\wtd C^\infty_\mf n)^G(1)(\wtd f_S)^G(1) + O_{\mf n}(m(\lb)^{-C} q_S^{A + B \kappa} (\wtd f_S)^G(1)).
\end{multline}
By Theorem \ref{Ctransferpositivity}, $(\wtd C^\infty_\mf n)^G(1) > 0$. The result follows from dividing \eqref{eq sd final asymp w} by \eqref{eq sd final asymp nw}.
\end{proof}

In particular, when $\mf n$ satisfies the bulleted conditions, the root numbers $\eps(1/2, \pi)$ for $\pi \in \mc{AR}_N^\symp(\lb, \mf n)$ become equidistributed between $\pm 1$ as $m(\lb) \to \infty$. This equidistribution is ``independent'' of any full weighting by Weyl-invariant polynomials on the self-dual Satake parameters at places coprime to $\mf n$. There are also power-saving bounds on how fast equidistribution might be achieved. 

If $\mf n$ does not satisfy the bulleted conditions, then there is an ``alternating'' asymptotic bias between $\eps(1/2, \pi) = \pm 1$.

\subsection{Conjugate Self-Dual Case}
In the conjugate self-dual case, our techniques only apply when $N$ is even. Then, $\eps(1/2,\pi) = 1$ unless $\pi$ is conjugate symplectic/comes from $G = U^+_N$. We focus on this case.

Isolating the conditions from Theorem \ref{Ctransferpositivityconj}:
\begin{dfn}
For a place $v$ of $F$, let $e_v$ be the ramification degree, and recall the numbers $b_v$ from Lemma \ref{lem ccharglobconj} and $j_v$ from \S\ref{sec existenceconj}. Call a conductor $\mf n$ (as in \S\ref{sec conductorsconj}) valid if
\begin{itemize}
    \item For all wildly ramified $v$, $v(\mf n) \leq N/4$ or $v(\mf n) > 2j_v - 1$,
    \item There is at least one $v$ such that $\max\{\lfloor v(\mf n) - (j_v - 1/2) \rfloor, 0\} \geq b_v + 2/e_v$ or there are at least two $v$ such that $\max\{\lfloor v(\mf n) - (j_v - 1/2) \rfloor, 0\} \geq b_v + 1/e_v$. 
\end{itemize}

Call it $0$-valid if
\begin{itemize}
    \item For all wildly ramified $v$, $v(\mf n) \leq N/4$ or $v(\mf n) > 4j_v - 2$,
\end{itemize}
(see note \ref{note conductorsconj} for potentially weaker conditions.)
\end{dfn}

We repeat again that these conditions, while not comprehensive, are quite weak: for example, only the second part of the valid condition doesn't involve wild ramification and it is automatically satisfied if there are no roots of unity in $1 + \mf D_{E/F}$. 

\begin{thm}\label{mainthmconj}
Let $N > 2$ be even and let $\lb$ be a conjugate self-dual infinitesimal character of $G_N$ as in \S\ref{integralinfchars}.  Let $\mf n$ be a conductor for $F$. Assume that $\mf n$ is valid or, if $\lb$ corresponds to the trivial central character, possibly instead $0$-valid. 

Choose a finite set of finite places $S$ coprime to $\mf n$ and a test function
\[
\wtd f_S := \1_{K_S \rtimes \theta} \star f_S \in \ms H^\ur(\wtd G_N(F_S))^{\leq \kappa},
\]
such that $\wtd f_S$ is a full weighting for $G = U^+_N$. For any conjugate self-dual automorphic representation $\pi$ of $G_N$, define coefficients
\[
w(\pi) := \tr_{\wtd \pi_S}(\wtd f_S)
\]
with respect to the canonical extension of $\pi$ to $\wtd G_N$. Then, if there is a ramified place $v$ such that either
\begin{itemize}
    \item $v(\mf n)$ is half-integral,
    \item $v(\mf n) > N/2$,
\end{itemize}
we have
\[
\f{\sum_{\pi \in \mc{AR}_N^\symp(\lb, \mf n)} w(\pi) \eps(1/2, \pi)}{\sum_{\pi \in \mc{AR}_N^\symp(\lb, \mf n)} w(\pi)} = O_{\mf n, \om_\lb}(m(\lb)^{-C} q_S^{A + B \kappa})
\]
for some constants $A,B,C$ with $C > 1$ and where $\om_\lb$ is the central character of the finite-dimensional representation of~$G_\infty$ corresponding to $\lb$. 
\end{thm}

\begin{proof}
This is the same argument as Theorem \ref{mainthm} except we use the analogous results for the conjugate self-dual case. 

Theorem \ref{globalepsilontraceconj} and Proposition \ref{stablecount} interpret
\[
\sum_{\pi \in \mc{AR}_N^\symp(\lb, \mf n)} w(\pi) \eps(1/2, \pi) = \om_\infty(a)^{-1} \eps(1/2, \varphi_\lb)^{-1} S^G_{\Sigma_\eta}(\EP_\lb (\wtd E^\infty_\mf n \star f_S)^G)
\]
for some choice of $a/\bar a = -1$. Then, by Theorem \ref{weightasymptotics}:
\begin{multline*}
(\dim \lb)^{-1}S^G_{\Sigma_\eta}(\EP_\lb (\wtd E^\infty_\mf n \star f_S)^G) \\ = \Lambda(G, \om_\lb, (\wtd E^\infty_\mf n \star f_S)^G)  + O_{\mf n}(m(\lb)^{-C} q_S^{A + B \kappa}(\wtd f_S)^G(1)),
\end{multline*}
where our bulleted conditions on $\mf n$ give that
\[
\Lambda(G, \om_\lb, (\wtd E^\infty_\mf n \star f_S)^G) = 0
\]
by Corollary \ref{Etransfervanishingconj}. 

On the other hand, Corollary \ref{globalC} gives that
\[
\sum_{\pi \in \mc{AR}_N^\symp(\lb, \mf n)} w(\pi) = S^G_{\Sigma_\eta}(\EP_\lb (\wtd C^\infty_\mf n \star f_S)^G)
\]
and Theorem~\ref{weightasymptotics} again gives 
\begin{multline*}
(\dim \lb)^{-1}S^G_{\Sigma_\eta}(\EP_\lb (\wtd C^\infty_\mf n \star f_S)^G) \\ = \Lambda(G, \om_\lb, (\wtd C^\infty_\mf n \star f_S)^G)  + O_{\mf n}(m(\lb)^{-C} q_S^{A + B \kappa}(\wtd f_S)^G(1)). 
\end{multline*}
Let the support of $\mf n$ be $S_0$. Then by \eqref{maintermspectral},
\[
\Lambda(G, \om_\lb, (\wtd C^\infty_\mf n \star f_S)^G) = \sum_{\om_{S_0} \om_S \in \Om_\lb(K^{S, S_0, \infty})} E^\pl((\wtd C_{S_0,\mf n})^G f_S^G | \om_{S_0} \om_S) 
\]
for the set of characters $\Om_\lb(K^{S, S_0, \infty})$ of $Z_{G, S \cup S_0}$ defined there. Each of the summands factors into:
\[
E^\pl((\wtd C_{S_0,\mf n})^G| \om_{S_0})E^\pl(f_S^G | \om_S). 
\]
Since $f_S^G$ is unramified, the second factor is zero unless $\om_S|_{K_S}$ is trivial. Therefore, we can rewrite $\Lambda(G, \om_\lb, (\wtd C^\infty_\mf n \star f_S)^G)$ as
\[
\sum_{\om_{S_0} \in \Omega_\lb(K^{S_0, \infty})} E^\pl((\wtd C^\infty_{\mf n})^G| \om_{S_0}) \lf(\sum_{\om_{S_0} \om_S \in \Om_\lb(K^{S, S_0, \infty})} E^\pl(f_S^G | \om_S)\ri) 
\]
Since $f_S$ is a full weighting, the values of the inner sum are always strictly positive. By Theorem \ref{Ctransferpositivityconj} or \ref{Ctransferpositivityconjzero}, trace-positivity of $\wtd C^\infty_\mf n$, and validity of $\mf n$, all the terms~$E^\pl((\wtd C^\infty_{\mf n})^G| \om_{S_0})$ are non-negative and at least one of them is strictly positive. In total, $\Lambda(G, \om_\lb, (\wtd C^\infty_\mf n \star f_S)^G) > 0$. 

Therefore, we get our result by dividing these two computations. 
\end{proof}

Because we cannot prove Conjecture \ref{expectedEtransfernonvanishingconj}, we do not have a non-equidistribution result analoguous to the second part of Theorem \ref{mainthm}. However, we expect that it does hold for valid (or $0$-valid) $\mf n$ that do not satisfy the bulleted conditions. 

As a replacement, we could use the $\td f_{\mf n}$ from Corollary \ref{Etransfervanishingconj} in our trace formula computation instead of~$\wtd E^\infty_{\mf n}$. This would prove a non-equidistribution result for an average over $\pi$ with $c(\pi) | \mf n$ that is weighted by a product depending on $\mf n/c(\pi)$ of the (signed!) coefficients from Proposition \ref{twistedtraceoldvectorsconj}.

\bibliographystyle{amsalpha}
\bibliography{Tbib}

\newcommand{\etalchar}[1]{$^{#1}$}
\providecommand{\bysame}{\leavevmode\hbox to3em{\hrulefill}\thinspace}
\providecommand{\MR}{\relax\ifhmode\unskip\space\fi MR }
\providecommand{\MRhref}[2]{%
  \href{http://www.ams.org/mathscinet-getitem?mr=#1}{#2}
}
\providecommand{\href}[2]{#2}
\begin{thebibliography}{KMSW14}

\bibitem[AGI{\etalchar{+}}24]{AGIKMS24}
Hiraku Atobe, Wee~Teck Gan, Atsushi Ichino, Tasho Kaletha, Alberto Mínguez,
  and Sug~Woo Shin, \emph{Local intertwining relations and co-tempered
  $a$-packets of classical groups}, 2024.

\bibitem[AHKO23]{AHO23}
Moshe Adrian, Guy Henniart, Eyal Kaplan, and Masao Oi, \emph{Simple
  supercuspidal l-packets of split special orthogonal groups over dyadic
  fields}, 2023.

\bibitem[AJ87]{AJ87}
Jeffrey Adams and Joseph~F. Johnson, \emph{Endoscopic groups and packets of
  nontempered representations}, Compositio Math. \textbf{64} (1987), no.~3,
  271--309. \MR{918414}

\bibitem[AOY22]{AOY23}
Hiraku Atobe, Masao Oi, and Seidai Yasuda, \emph{Local newforms for generic
  representations of unramified odd unitary groups and fundamental lemma},
  2022.

\bibitem[Art89]{Art89}
James Arthur, \emph{The {$L^2$}-{L}efschetz numbers of {H}ecke operators},
  Invent. Math. \textbf{97} (1989), no.~2, 257--290. \MR{1001841}

\bibitem[Art13]{Art13}
\bysame, \emph{The endoscopic classification of representations}, American
  Mathematical Society Colloquium Publications, vol.~61, American Mathematical
  Society, Providence, RI, 2013, Orthogonal and symplectic groups. \MR{3135650}

\bibitem[BDK86]{BDK86}
J.~Bernstein, P.~Deligne, and D.~Kazhdan, \emph{Trace {P}aley-{W}iener theorem
  for reductive {$p$}-adic groups}, J. Analyse Math. \textbf{47} (1986),
  180--192. \MR{874050}

\bibitem[BG14]{BG14}
Kevin Buzzard and Toby Gee, \emph{The conjectural connections between
  automorphic representations and galois representations}, London Mathematical
  Society Lecture Note Series, p.~135–187, Cambridge University Press, 2014.

\bibitem[BH06]{BH06}
Colin~J Bushnell and Guy Henniart, \emph{The local langlands conjecture for gl
  (2)}, vol. 335, Springer Science \& Business Media, 2006.

\bibitem[Bin17]{Bin17}
John Binder, \emph{Refined limit multiplicity for varying conductor}, Int.
  Math. Res. Not. IMRN (2017), no.~22, 6731--6751. \MR{3737319}

\bibitem[BP21]{BP21}
Rapha\"el Beuzart-Plessis, \emph{Plancherel formula for {${\rm
  GL}_n(F)\backslash {\rm GL}_n(E)$} and applications to the {I}chino-{I}keda
  and formal degree conjectures for unitary groups}, Invent. Math. \textbf{225}
  (2021), no.~1, 159--297. \MR{4270666}

\bibitem[BZ77]{BZ77}
I.~N. Bernstein and A.~V. Zelevinsky, \emph{Induced representations of
  reductive {${p}$}-adic groups. {I}}, Ann. Sci. \'Ecole Norm. Sup. (4)
  \textbf{10} (1977), no.~4, 441--472. \MR{579172}

\bibitem[Cog04]{cogdell2004long}
J~Cogdell, \emph{Lectures on l-functions, converse theorems, and functoriality
  for gln}, Lectures on automorphic L-functions \textbf{20} (2004), 1--96.

\bibitem[CR15]{CR15}
Ga\"{e}tan Chenevier and David Renard, \emph{Level one algebraic cusp forms of
  classical groups of small rank}, Mem. Amer. Math. Soc. \textbf{237} (2015),
  no.~1121, v+122. \MR{3399888}

\bibitem[Dal22]{Dal22}
Rahul Dalal, \emph{Sato--{T}ate equidistribution for families of automorphic
  representations through the stable trace formula}, Algebra Number Theory
  \textbf{16} (2022), no.~1, 59--137. \MR{4384564}

\bibitem[DEP]{DEP}
Rahul Dalal, Shai Evra, and Ori Parzanchevski, \emph{Golden gates in {PU(N)}},
  draft available at:
  \url{https://www.mat.univie.ac.at/~rdalal/GoldenGatesDraft.pdf}.

\bibitem[DGG22]{DGG22}
Rahul Dalal and Mathilde Gerbelli-Gauthier, \emph{Statistics of cohomological
  automorphic representations via the endoscopic classification}, 2022.

\bibitem[DS05]{DS05}
Fred Diamond and Jerry Shurman, \emph{A first course in modular forms},
  Graduate Texts in Mathematics, vol. 228, Springer-Verlag, New York, 2005.
  \MR{2112196}

\bibitem[Fer07]{Fer07}
Axel Ferrari, \emph{Th\'eor\`eme de l'indice et formule des traces},
  Manuscripta Math. \textbf{124} (2007), no.~3, 363--390. \MR{2350551}

\bibitem[GGP11]{GGP11}
Wee~Teck Gan, Benedict~H Gross, and Dipendra Prasad, \emph{Symplectic local
  root numbers, central critical l-values, and restriction problems in the
  representation theory of classical groups}, Ast{\'e}risque (2011), No--pp.

\bibitem[Gou78]{Go78}
Henry~W Gould, \emph{Euler's formula for n th differences of powers}, The
  American Mathematical Monthly \textbf{85} (1978), no.~6, 450--467.

\bibitem[Hen86]{Hen86}
Guy Henniart, \emph{On the local langlands conjecture for gl (n): the cyclic
  case}, Annals of Mathematics \textbf{123} (1986), no.~1, 145--203.

\bibitem[Hen00]{Hen00}
\bysame, \emph{Une preuve simple des conjectures de langlands pour gl (n) sur
  un corps p-adique}, Inventiones mathematicae \textbf{139} (2000), 439--455.

\bibitem[HKP10]{HKP10}
Thomas~J. Haines, Robert~E. Kottwitz, and Amritanshu Prasad,
  \emph{Iwahori-{H}ecke algebras}, J. Ramanujan Math. Soc. \textbf{25} (2010),
  no.~2, 113--145. \MR{2642451}

\bibitem[HT99]{HT99}
Michael Harris and Richard~Lawrence Taylor, \emph{The geometry and cohomology
  of some simple shimura varieties}, Princeton University Press, 1999.

\bibitem[Hum20]{Hum20}
Peter Humphries, \emph{Archimedean newform theory for $\mathrm{GL}_n$}, 2020.

\bibitem[JPSS81]{JPSS81}
H.~Jacquet, I.~I. Piatetski-Shapiro, and J.~Shalika, \emph{Conducteur des
  repr\'{e}sentations du groupe lin\'{e}aire}, Math. Ann. \textbf{256} (1981),
  no.~2, 199--214. \MR{620708}

\bibitem[KMSW14]{KMSW14}
Tasho Kaletha, Alberto Minguez, Sug~Woo Shin, and Paul-James White,
  \emph{Endoscopic classification of representations: inner forms of unitary
  groups}, arXiv preprint arXiv:1409.3731 (2014).

\bibitem[Kot86]{Kot86}
Robert~E. Kottwitz, \emph{Stable trace formula: elliptic singular terms}, Math.
  Ann. \textbf{275} (1986), no.~3, 365--399. \MR{858284}

\bibitem[Kot88]{Kot88}
\bysame, \emph{Tamagawa numbers}, Ann. of Math. (2) \textbf{127} (1988), no.~3,
  629--646. \MR{942522}

\bibitem[KS99]{KS99}
Robert~E. Kottwitz and Diana Shelstad, \emph{Foundations of twisted endoscopy},
  Ast\'{e}risque (1999), no.~255, vi+190. \MR{1687096}

\bibitem[KS20]{KS20}
Arno Kret and Sug~Woo Shin, \emph{Galois representations for even general
  special orthogonal groups}, Journal of the Institute of Mathematics of
  Jussieu (2020), 1--92.

\bibitem[KS23]{KS23}
\bysame, \emph{Galois representations for general symplectic groups}, J. Eur.
  Math. Soc. (JEMS) \textbf{25} (2023), no.~1, 75--152. \MR{4556781}

\bibitem[KY12]{KY12}
Satoshi Kondo and Seidai Yasuda, \emph{Local l and epsilon factors in hecke
  eigenvalues}, Journal of Number Theory \textbf{132} (2012), no.~9,
  1910--1948.

\bibitem[Lab11]{Lab11}
Jean-Pierre Labesse, \emph{Introduction to endoscopy: {S}nowbird lectures,
  revised version, {M}ay 2010 [revision of mr2454335]}, On the stabilization of
  the trace formula, Stab. Trace Formula Shimura Var. Arith. Appl., vol.~1,
  Int. Press, Somerville, MA, 2011, pp.~49--91. \MR{2856367}

\bibitem[LMF]{LMFDB}
\emph{The {L}-functions and modular forms database},
  \url{https://www.lmfdb.org}, Online; accessed 1-October-2024.

\bibitem[LR05]{LR05}
Erez~M. Lapid and Stephen Rallis, \emph{On the local factors of representations
  of classical groups}, Automorphic representations, {$L$}-functions and
  applications: progress and prospects, Ohio State Univ. Math. Res. Inst.
  Publ., vol.~11, de Gruyter, Berlin, 2005, pp.~309--359. \MR{2192828}

\bibitem[LS07]{LS07}
R~Langlands and Diana Shelstad, \emph{Descent for transfer factors}, The
  Grothendieck Festschrift: A Collection of Articles Written in Honor of the
  60th Birthday of Alexander Grothendieck (2007), 485--563.

\bibitem[Mar18]{Martin18}
Kimball Martin, \emph{Refined dimensions of cusp forms, and equidistribution
  and bias of signs}, Journal of Number Theory \textbf{188} (2018), 1--17.

\bibitem[Mar23]{Martin23}
\bysame, \emph{Root number bias for newforms}, Proceedings of the American
  Mathematical Society \textbf{151} (2023), no.~09, 3721--3736.

\bibitem[Mat13]{Mat13}
Nadir Matringe, \emph{Essential {W}hittaker functions for {$GL(n)$}}, Doc.
  Math. \textbf{18} (2013), 1191--1214. \MR{3138844}

\bibitem[Mok15]{Mok15}
Chung~Pang Mok, \emph{Endoscopic classification of representations of
  quasi-split unitary groups}, American Mathematical Soc., 2015.

\bibitem[MW89]{MW89}
Colette M{\oe}glin and J-L Waldspurger, \emph{Le spectre residuel de $gl(n)$},
  Annales scientifiques de l'{\'E}cole normale superieure, vol.~22, 1989,
  pp.~605--674.

\bibitem[MW16]{MW16}
Colette Moeglin and Jean-Loup Waldspurger, \emph{Stabilisation de la formule
  des traces tordue. {V}ol. 1}, Progress in Mathematics, vol. 316,
  Birkh\"auser/Springer, Cham, 2016. \MR{3823813}

\bibitem[NSW13]{NSW13cohomology}
J{\"u}rgen Neukirch, Alexander Schmidt, and Kay Wingberg, \emph{Cohomology of
  number fields}, vol. 323, Springer Science \& Business Media, 2013.

\bibitem[Ree91]{Ree91}
Mark Reeder, \emph{Old forms on {${\rm GL}_n$}}, Amer. J. Math. \textbf{113}
  (1991), no.~5, 911--930. \MR{1129297}

\bibitem[Rog88]{Rog88}
J.~D. Rogawski, \emph{Trace {P}aley-{W}iener theorem in the twisted case},
  Trans. Amer. Math. Soc. \textbf{309} (1988), no.~1, 215--229. \MR{957068}

\bibitem[Rog81]{Rog81}
Jonathan~D. Rogawski, \emph{An application of the building to orbital
  integrals}, Compositio Math. \textbf{42} (1980/81), no.~3, 417--423.
  \MR{607380}

\bibitem[Sch13]{Sch13}
Peter Scholze, \emph{The local langlands correspondence for gl n over p-adic
  fields}, Inventiones mathematicae \textbf{192} (2013), 663--715.

\bibitem[Ser13]{Se13}
Jean-Pierre Serre, \emph{Local fields}, vol.~67, Springer Science \& Business
  Media, 2013.

\bibitem[Shi11]{Shin11}
Sug~Woo Shin, \emph{Galois representations arising from some compact {S}himura
  varieties}, Ann. of Math. (2) \textbf{173} (2011), no.~3, 1645--1741.
  \MR{2800722}

\bibitem[Shi24]{Shi24}
\bysame, \emph{Weak transfer from classical groups to general linear groups},
  Essential Number Theory \textbf{3} (2024), no.~1, 19--62.

\bibitem[Sil96]{Sil96}
Allan~J. Silberger, \emph{Harish-{C}handra's {P}lancherel theorem for
  {${\mathfrak p}$}-adic groups}, Trans. Amer. Math. Soc. \textbf{348} (1996),
  no.~11, 4673--4686. \MR{1370652}

\bibitem[ST16]{ST16}
Sug~Woo Shin and Nicolas Templier, \emph{Sato-{T}ate theorem for families and
  low-lying zeros of automorphic {$L$}-functions}, Invent. Math. \textbf{203}
  (2016), no.~1, 1--177, Appendix A by Robert Kottwitz, and Appendix B by Raf
  Cluckers, Julia Gordon and Immanuel Halupczok. \MR{3437869}

\bibitem[SZ88]{SZ88}
Nils-Peter Skoruppa and Don Zagier, \emph{Jacobi forms and a certain space of
  modular forms}, Invent. Math. \textbf{94} (1988), no.~1, 113--146.
  \MR{958592}

\bibitem[Ta{\"i}17]{Tai17}
Olivier Ta{\"i}bi, \emph{Dimensions of spaces of level one automorphic forms
  for split classical groups using the trace formula}, Ann. Sci. {\'E}c. Norm.
  Sup{\'e}r. (4) \textbf{50} (2017), no.~2, 269--344. \MR{3621432}

\bibitem[Tat79]{Ta79}
John Tate, \emph{Number theoretic background}, Automorphic forms,
  representations and L-functions (Proc. Sympos. Pure Math., Oregon State
  Univ., Corvallis, Ore., 1977), Part, vol.~2, 1979, pp.~3--26.

\bibitem[Tom24]{Tom24}
Radu Toma, \emph{The sup-norm problem for newforms of large level on
  $\operatorname{PGL}(n)$}, 2024.

\bibitem[TW24]{TW24}
Yugo Takanashi and Satoshi Wakatsuki, \emph{Asymptotic behavior for twisted
  traces of self-dual and conjugate self-dual representations of
  $\mathrm{GL}_n$}, 2024.

\bibitem[Var09]{Var09}
Sandeep Varma, \emph{Descent and the generic packet conjecture}, Ph.D. thesis,
  Purdue University, 2009.

\bibitem[Wal08]{WTE}
Jean-Loup Waldspurger, \emph{L'endoscopie tordue n'est pas si tordue}, American
  Mathematical Soc., 2008.

\bibitem[Yun13]{Yun13}
Zhiwei Yun, \emph{Orbital integrals and dedekind zeta functions}, 2013.

\bibitem[Zub23]{Zub23}
Nina Zubrilina, \emph{Murmurations}, arXiv preprint arXiv:2310.07681 (2023).

\end{thebibliography}

\end{document}